\documentclass[a4paper]{paper}
\usepackage{amsmath,amsthm,amsfonts,mathrsfs}
\usepackage{amssymb}
\usepackage{tikz-cd}
\usepackage{xspace}
\usepackage{graphicx}
\usepackage{paralist}
\usepackage{hyperref} 

\usepackage{multirow}
\usepackage{diagbox}

\usepackage[a4paper]{geometry} % 调整纸张大小和页边距的包，中括号中规定了纸张大小

\hypersetup{%
    pdfmenubar=true,       % show Acrobat's menu?
    pdffitwindow=false,     % window fit to page when opened
    pdfstartview={FitH},    % fits the width of the page to the window
    pdftitle={Spatiotemporal Imaging with Diffeomorphic Optimal Transportation},            
    colorlinks=true,       % false: boxed links; true: colored links
    linkcolor=red,          % color of internal links
    citecolor=red,        % color of links to bibliography
    filecolor=magenta,      % color of file links
    urlcolor=cyan,           % color of external links
    pdfborder = {0,0,0}
}
\usepackage{acro}
\usepackage{cleveref}
\usepackage{algorithm}
\usepackage[noend]{algpseudocode}
\makeatletter
\def\BState{\State\hskip-\ALG@thistlm}
\makeatother

\Crefname{equation}{}{}
\crefname{equation}{}{}

\usepackage{todonotes}
\usepackage{specialmacros}
\usepackage{paperacronyms}

\begin{document}

\title{
%High-order Wasserstein Distance with Applications to Spatiotemporal Imaging
%Joint Image Reconstruction and Motion Estimation Based on Optimal Transportation  
%High-order Wasserstein Distance Based Joint Image Reconstruction and Motion Estimation 
%Joint Image Reconstruction and Motion Estimation via High-order Wasserstein Distance
%Spatiotemporal Imaging Using  Wasserstein Distance 
Spatiotemporal Imaging with Diffeomorphic Optimal Transportation
}
\author{Chong Chen\thanks{LSEC, ICMSEC, Academy of Mathematics and Systems Science, Chinese Academy of Sciences, Beijing 100190, China (chench@lsec.cc.ac.cn).}}

\maketitle

\begin{abstract}
We propose a variational model with diffeomorphic optimal transportation for joint image reconstruction and motion estimation. The proposed model is a production of assembling the Wasserstein distance with the Benamou--Brenier formula in optimal transportation and the flow of diffeomorphisms involved in large deformation diffeomorphic metric mapping, which is suitable for the scenario of spatiotemporal imaging with large diffeomorphic and mass-preserving deformations. Specifically, we first use the Benamou--Brenier formula to characterize the optimal transport cost among the flow of mass-preserving images, and restrict the velocity field into the admissible Hilbert space to guarantee the generated deformation flow being diffeomorphic. We then gain the ODE-constrained equivalent formulation for Benamou--Brenier formula. We finally obtain the proposed model with ODE constraint following the framework that presented in our previous work. We further get the equivalent PDE-constrained optimal control formulation. The proposed model is compared against several existing alternatives theoretically. The alternating minimization algorithm is presented for solving the time-discretized version of the proposed model with ODE constraint. Several important issues on the proposed model and associated algorithms are also discussed. Particularly, we present several potential models based on the proposed diffeomorphic optimal transportation. Under appropriate conditions, the proposed algorithm also provides a new scheme to solve the models using quadratic Wasserstein distance. The performance is finally evaluated by several numerical experiments in space-time tomography, where the data is measured from the concerned sequential images with sparse views and/or various noise levels.
\end{abstract}

\begin{keywords}
spatiotemporal imaging, joint image reconstruction and motion estimation, quadratic Wasserstein distance, flow of diffeomorphisms, diffeomorphic optimal transportation, mass-preserving deformation 
\end{keywords}

\section{Introduction}
\label{sec:Introduction}
Mathematically, the spatiotemporal (space-time) imaging is typically a kind of time-dependent or dynamic inverse problems, 
which has been gained extensively study (\cite{GiJiDaSc15,ScHaBu18,ChGrOz19} and the references therein). As a representative example, 
when the tomographic imaging (e.g., \ac{CT}, \ac{PET}, \ac{SPECT}, \ac{MRI}) is used for chest or heart inspection, 
the measurements are frequently acquired over a time period in the minute magnitude. If the unavoidable respiration and/or cardiac 
movements are neglected or failed to track and correct, this would lead to the reconstructed images with severe 
degradation \cite{WaViBe99,ScLe00,HaLi10,RaTaZa13}. As a result, it is significant to estimate and compensate for the unknown 
motions of the organs during image reconstruction in spatiotemporal imaging.  

In spatiotemporal imaging, the acquired data is usually a time-series, which is often divided into 
gates by time or amplitude based respiration and/or cardiac gating method. The details about the gating method 
are referred to \cite{LaDaBuScScSc06,DaBuLaScSc07,DaStJiScScSc09,Luci09,GrGeMe14}. 
After gated, the collected data within each gate can be seen as a certain data measured from a static object in a 
fixed time or pseudo-time state \cite{BuDaWuSc09}. Then one can use the different kind of strategies to perform 
spatiotemporal image reconstruction. An intuitive strategy is first performing image reconstruction for each gate independently, 
and then implementing motion correction/compensation by image registration for the reconstructed 
images \cite{DaBuJiSc08,DaLaJiSc08,BaBr11,GiRuBu12,BrSaMaKa15}. However, the data at each 
gate is often interfered by enhanced noise levels due to gating. If it is worse sparse-sampled in order to 
decrease the radiative dose, then the reconstructed images at the first step would be reduced quality or full of artifacts,  
which makes the following motion correction step almost malfunction. The other strategy is joint image reconstruction and motion 
estimation, namely, incorporating the physical motion into the reconstruction process, in which one establishes 
multiple tasks jointly into one single model, and then gains the optimal solution to reconstruct the image at 
each gate. The main idea of this type of methods is to make use of the sufficient information in the data for each step.  
A lot of approaches have been proposed for how to do this, 
such as those in \cite{ScMoFi09,Br10,BlNaRa12,HiSzWaSaJo12,TaCaSr12,LiZhZhGa15,BoTh16,BuDiFrHaHeSi17,BuDiSch18,LuHuBeAr18,ChGrOz19}. 
Additionally, several methods took the spatiotemporal images as the unknows of the optimization/variational 
models, and considered the temporal regularization by locally adjacent images or in sparse matrix 
form \cite{HaLi10,JiLoDoTi10,GaCaSh11,RiSaKnKa12}. 

In particular, a general framework of joint variational model was proposed for spatiotemporal imaging in \cite{ChGrOz19}. 
Actually, the most important component of that model is for motion estimation using the acquired projection data under the assumption that 
the template is given. This problem is termed sequentially indirect image registration, which is the generalization of indirect 
image registration studied in \cite{ChOz18}. More details on the latter are referred to \cite{OkChDoRaBa16,GrChOk18}. 
To solve the former, a consistent growth model based on \ac{LDDMM} was proposed 
in \cite{ChGrOz19}, which can track the flow of motions with large non-rigid and diffeomorphic deformations. It is well-known that 
the \ac{LDDMM} is a fundamental method for diffeomorphic image registration (\cite{Tro98,DuGrMi98,MiTrYo02,BeMiTrYo05,Yo10,BrHo15}). 

Following the general framework presented in \cite{ChGrOz19}, the current work is dedicated to 
proposing a new joint variational model with diffeomorphic optimal transportation, which is useful  
for the spatiotemporal imaging with large mass-preserving and diffeomorphic deformations. The mass-preserving property was often  
considered to be significant in dynamic \ac{PET} imaging, such as the cardiac imaging \cite{GiRuBu12}. To do so, 
being quite different from the method in \cite{ChGrOz19}, this paper uses the the Wasserstein distance with 
Benamou--Brenier formula in optimal transportation to characterize the optimal transport cost among the mass-preserving 
image flows. For the detailed introduction on the theory and applications of optimal transportation, the readers are referred 
to \cite{BeBr00,Vil03,Am03,HaZhTaAn04,EnFr14,MaRuScSi15,ArChBo17,KarRin17,PeCu18}. Furthermore, this work restricts the unknown velocity 
field into the admissible Hilbert space to generate the flow of diffeomorphisms for tracking the involved motions. 
Numerically, the alternating minimization algorithm is presented for solving the proposed model with \ac{ODE} constraint. 
Moreover, under appropriate conditions, the proposed algorithm provides a new method to 
solve the models using $\LpSpace^2$ Wasserstein distance in Benamou--Brenier formulation, such as the extended 
models with diffeomorphic optimal transportation for (sequential) image registration and the related indirect problems. 

The outline of this paper is organized as follows. \Cref{sec:preliminaries} introduces the required mathematical
preliminaries. The new variational model for joint image reconstruction and motion estimation is presented  
in \cref{sec:OT_recon_model}. The numerical implementation is given in \cref{sec:computed_method} to solve 
the proposed model. The numerical test suites are performed in \cref{sec:numerical_experiments} to 
evaluate the performance of the proposed method. \Cref{sec:Discussion} further discusses several important 
issues about the proposed model and related algorithms. Finally, the paper is concluded by \cref{sec:Conclusions}.

\section{Preliminaries}\label{sec:preliminaries}

First of all, we present the problem setting of spatiotemporal imaging. And then we recall the requisite mathematical tools, 
including the basic concept and related results of Wasserstein distance in optimal transportation, and the flow of diffeomorphisms in \ac{LDDMM}. 

\subsection{Spatiotemporal imaging}
\label{sec:problem_setting}

A general framework for spatiotemporal imaging was presented in the previous 
work \cite{ChGrOz19}, which is based on the deformable template of shape theory \cite{Ar45,GrMi07,Yo10}. 
Here we will give a brief introduction. 

Suppose that the time-dependent image required to be reconstructed is defined by 
\[
\signal \colon [t_0, t_1] \times \domain \to \Real^m,
\] where $m$ (generally $m=1$) is denoted as the modality number, 
$[t_0, t_1] \subset \Real^1$ is expressed as the temporal domain,  and $\domain \subset \Real^n$ (often $n = 2$ or $3$) is 
represented as the spatial domain. The domain $\domain$ is assumed to include the support of the images for all of the time, 
which is a bounded, compact, convex open set with the strong local Lipschitz condition 
throughout the paper. Note that the general (pseudo) time domain $[t_0, t_1]$ can be always reparameterized onto $[0,1]$.
It is well-known that the aim of spatiotemporal imaging is to reconstruct a time-dependent image $\signal(t,\Cdot) \in \RecSpace$ by using  
the acquired data $\data(t,\Cdot) \in\DataSpace$ for $t \in [0, 1]$, where $\RecSpace$ denotes the reconstruction function space, 
and $\DataSpace$ represents the data function space. Correspondingly, the general mathematical formulation is written as 
\begin{equation}\label{eq:InvProb}  
     \data(t,\Cdot) = \ForwardOp\bigl(t, \signal(t,\Cdot)\bigr) + \noisedata(t,\Cdot) \quad \text{for $t \in [0, 1]$}. 
\end{equation}
Here the $\ForwardOp(t, \Cdot) \colon \RecSpace \to \DataSpace$ represents a time-dependent linear or nonlinear 
forward operator, which models the forward process without noise or errors on how the image at time state $t$ generates the data.  
For instance, the forward operator is realized by Radon transform in \ac{CT} and \ac{AC} \ac{PET}; the attenuated Radon transform 
is adopted by \ac{SPECT}; the downsampled Fourier transform is used for \ac{MRI} \cite{LiLa99,Na01}. 
For brevity, we denote $\ForwardOp(t, \Cdot)$ by $\ForwardOp_t$. 
Moreover, the $\noisedata(t,\Cdot) \in \DataSpace$ stands for the uncertain noise contained in data, 

As in \cite{ChGrOz19}, based on the deformable template of shape theory, the time-dependent image can be expressed as
\begin{equation}\label{eq:SeparateImage}
\signal(t,\Cdot) := \DefOpV(\diffeo_t, \template), 
\end{equation}
where $\RecSpace \ni \template \colon \domain \to \Real$ denotes the template that is time-independent spatial 
component of the spatiotemporal image, $\LieGroup \ni \diffeo_t \colon \domain \to \domain$ defines the time-dependent 
deformation that manages the temporal evolution of the template, and the $\LieGroup$ stands for the group of diffeomorphisms on $\domain$. 
The $\DefOpV \colon \LieGroup \times \RecSpace \to \RecSpace$ denotes the temporal evolution operator, 
which is required to be a group action of $\LieGroup$ on $\RecSpace$. 
For short, we rewrite $\DefOpV(\diffeo_t, \template) := \diffeo_t . \template$. In other words, 
given the template $\template$ and the deformation $\diffeo_t$, the spatiotemporal image at 
time $t$ can be generated by $\signal(t,\Cdot)=\diffeo_t . \template$. 
Remark that $\diffeo_0 = \Id$ often denotes the identity deformation (mapping). 
By \cref{eq:SeparateImage}, the spatiotemporal imaging \cref{eq:InvProb} can be translated into 
\begin{equation}\label{eq:InvProb_2}  
     \data(t,\Cdot) = \ForwardOp_t(\diffeo_t . \template) + \noisedata(t,\Cdot) \quad \text{for $t \in [0, 1]$}. 
\end{equation}

There exist two frequently used group actions \cite{ChOz18}. The one that adopted in this work is  
given by the form of mass-preserving deformation 
\begin{equation}\label{eq:MassPreservedDeform}
\diffeo_t . \template = \bigl\vert \Diff(\diffeo_t^{-1}) \bigr\vert \template \circ \diffeo_t^{-1},
\end{equation}
where the ``$\circ$" denotes function composition, and $\vert \Diff(\phi) \vert$ is referred to 
the determinant of the Jacobian of $\phi$. Obviously, this deformation changes the intensity 
of the image but preserves its mass. It is well-known that such deformation is required 
in the framework of optimal transportation. An alternative one is given by geometric deformation, 
which is referred to \cite{ChOz18}. 

As proposed in \cite{ChGrOz19}, the general framework for solving the 
spatiotemporal inverse problem is formulated as  
\begin{equation}\label{eq:VarReg_2}
 \min_{\substack{\template \in \RecSpace \\ \diffeo_t \in \LieGroup}} 
    \int_{0}^{1} \Bigl[\DataDisc_{\ForwardOp_t, \data_t}\bigl(\diffeo_t . \template\bigr)  + \mu_2\RegFunc_2(\diffeo_t) \Bigr]\dint t + \mu_1 \RegFunc_1(\template), 
\end{equation}  
where 
\begin{equation}\label{eq:LDDMM_match_short}
\DataDisc_{\ForwardOp_t, \data_t}\bigl(\diffeo_t . \template\bigr)  :=  \DataDisc\Bigl( \ForwardOp_t\bigl(\diffeo_t . \template\bigr), \data(t,\Cdot) \Bigr)
\end{equation}
and $\mu_1$, $\mu_2$ are positive regularization parameters. 
Here the $\DataDisc \colon \DataSpace \times \DataSpace \rightarrow \Real_+$ acts as the 
data fitting functional. The $\RegFunc_1 \colon \RecSpace \to \Real_+$ is the spatial regularization for imposing 
priori information about the template image. Moreover, the $\RegFunc_2 \colon \LieGroup \to \Real_+$ is the shape regularization 
for constructing the desirable and applicable flow of deformations, which is critical in mathematical 
modelling for spatiotemporal imaging. 

In this work, we will study this problem by combining the Wasserstein distance and the flow of diffeomorphisms.

\subsection{The Wasserstein distance} \label{sec:OT_principle}

The original transportation problem can be traced back to the work of Monge in \cite{Mo81}, which is a 
civil engineering problem that parcels of materials have to be displaced from one site to another one 
with minimal transportation cost \cite{BeBr00}. A modern treatment of this problem has been initiated 
by Kantorovich in \cite{Kan48}, leading to the Monge--Kantorovich problem which has gained extensive 
interests, and also had a broad range of applications in the recent years \cite{BeBr00,Vil03,Am03}. 

The original problem can be stated as follows: given two density distributions 
$\signal_0$ and $\signal_1$  with equal masses of a given material 
(corresponding for instance to an embankment and an excavation), find a transportation 
map $\diffeo: \Real^n \rightarrow \Real^n$ which carries the first distribution 
into the second and minimizes the transportation cost
\begin{equation}\label{eq:costfunc1}
d_p(\signal_0, \signal_1) :=\inf_{\diffeo} \left(\int_{\Real^n} |x-\diffeo(x)|^p\signal_0(x)\dint x\right)^{1/p}, 
\end{equation}
where the condition that the first distribution of mass is carried into the second can be written as
\begin{equation}\label{eq:masspreserv}
\int_{\diffeo^{-1}(B)} \signal_0(x)\dint x = \int_B \signal_1(y)\dint y \quad \text{for $\forall B\subset \mathcal{B}(\Real^n)$},
\end{equation}
or, by the change of variables formula, as
\begin{equation}\label{eq:masspreserv2}
\vert \Diff (\diffeo) (x) \vert \signal_1\bigl(\diffeo(x)\bigr) = \signal_0(x) \quad \text{for~$\mathcal{L}^n$\text{-a.e.}~$x\in B$},
\end{equation}
namely, the \textit{Jacobian equation} (similarly, see the mass-preserving deformation 
in \cref{eq:MassPreservedDeform}), if $\diffeo$ is one to one and sufficiently regular. 
Here $d_p(\signal_0, \signal_1)$ is defined as the so-called \textit{$\LpSpace^p$ Wasserstein (or Kantorovich) distance} 
between $\signal_0$ and $\signal_1$ with fixed $p \ge 1$, 
in which $|\cdot|$ denotes the Euclidean norm in $\Real^n$, and the infimum is taken among all 
map $\diffeo$ transporting $\signal_0$ to $\signal_1$. 
Note that $\mathcal{B}(\Real^n)$ is Borel $\sigma$-algebra of $\Real^n$, and $\mathcal{L}^n$ is Lebesgue 
measure in $\Real^{d}$. If the infimum is achieved by some map $\diffeo^{\ast}$, we say that $\diffeo^{\ast}$ is an optimal 
transfer and solves the \textit{$\LpSpace^p$ Monge--Kantorovich problem (MKP)} \cite{BeBr00,Am03}. 

For the case $p=2$, the above optimal transportation problem can be further reformulated in 
the way inspired by fluid mechanics, which will be useful for the study of spatiotemporal imaging.

\begin{theorem}[\cite{BeBr00}]\label{thm:BB_formula}
Assume that the time-dependent density $\signal(t, x) \ge 0$ and velocity field $\velocityfield(t, x) \in \Real^n$ are 
appropriately smooth, and $\signal_0$ and $\signal_1$ are compactly supported. 
The square of the $\LpSpace^2$ Wasserstein distance equals to 
\begin{equation}\label{eq:Wasserstein_term}
\inf_{\signal \ge 0, \velocityfield}\int_0^1 \int_{\Real^n}\signal(t, x)|\velocityfield(t, x)|^2\dint x\dint t,
\end{equation}
such that 
\begin{equation}\label{eq:pde_constraint}
\begin{cases} 
    \partial_t \signal(t, x) +  \grad\cdot\bigl(\signal(t, x)\, \velocityfield(t, x) \bigr)= 0 & \\[0.5em]
    \signal(0, x) = \signal_0(x), \quad \signal(1, x) = \signal_1(x) &  
   \end{cases} 
   \quad\text{for $x\in \Real^n$ and $0 \leq t \leq 1$.}    
\end{equation}
\end{theorem}

The formula in \cref{thm:BB_formula} is also called \textit{Benamou--Brenier formula}, which is referred to \cite[Theorem 8.1]{Vil03}. 
Remarkably, the minimizer of $\LpSpace^2$ MKP is the solution at $t = 1$ to the following \ac{ODE} 
\begin{equation}\label{eq:FlowEq_0}
 \begin{cases} 
    \partial_t \diffeo_t(x) = \velocityfield\bigl( t, \diffeo_t(x) \bigr) & \\[0.5em]
    \diffeo_0(x)=x &  
   \end{cases} 
   \quad\text{for $x\in \Real^n$ and $0 \leq t \leq 1$.}    
\end{equation}

\subsection{The flow of diffeomorphisms} \label{sec:flow_diffeomorphisms}

To begin with, let $\signal_0$ and $\signal_1$ be two density functions compactly supported on $\domain$. 
Here we review a way to produce a flow of diffeomorphisms through a velocity field, 
which has been successfully used in \ac{LDDMM} \cite{BeMiTrYo05,Yo10}. 
Specifically, fixed an appropriate velocity field $\velocityfield \colon [0, 1] \times \domain \to \Real^n$, 
a flow of diffeomorphisms $\diffeo_t$ is produced by the \ac{ODE} below. 
\begin{equation}\label{eq:FlowEq}
 \begin{cases} 
    \partial_t \diffeo_t(x) = \velocityfield\bigl( t, \diffeo_t(x) \bigr) & \\[0.5em]
    \diffeo_0(x)=x &  
   \end{cases} 
   \quad\text{for $x\in \domain$ and $0 \leq t \leq 1$.}    
\end{equation} 
Subsequently,  the required regularity condition will be given for the velocity field.  
To proceed, we first give the concept of admissible space. 

\begin{definition}[\cite{Yo10}]\label{def:Admissible}
If a Hilbert space $\Vspace$ is (canonically) embedded in 
$\Smooth^1_0(\domain, \Real^n)$ with norm $\Vert \cdot \Vert_{1,\infty}$, namely, existing a constant $C>0$ such that
\[    \Vert \vfield \Vert_{1,\infty}  \leq C \Vert \vfield \Vert_{\Vspace}
     \quad\text{for any $\vfield \in \Vspace$,} \]
then the $\Vspace$ is called admissible. Here $\Vert \vfield \Vert_{1,\infty} := \Vert \vfield \Vert_{\infty} + \Vert \Diff\vfield \Vert_{\infty}$ for 
$\vfield \in \Smooth^1_0(\domain, \Real^n)$. 
\end{definition} 

Using the definition of \cref{def:Admissible}, a space of velocity fields is defined as 
\begin{equation}\label{eq:FlowSpace}
  \LpSpace^p([0,1], \Vspace) := 
     \Bigl\{ 
       \text{$\velocityfield  : \velocityfield(t,\cdot) \in \Vspace$ 
       and $\Vert \velocityfield \Vert_{ \LpSpace^p([0,1], \Vspace)}  < \infty$ for $1 \leq p \leq \infty$}.
     \Bigr\}
\end{equation} 
Then the norm is given as 
\[ \Vert \velocityfield \Vert_{\LpSpace^p([0,1], \Vspace)} :=  
      \biggl( \int_0^1 \bigl\Vert \velocityfield(t,\cdot) \bigr\Vert^p_{\Vspace}\dint t \biggr)^{1/p}.
\]  
For simplicity, denote $\LpSpace^p([0,1], \Vspace)$ by $\Xspace{p}$. Particularly, 
for $p = 2$, the $\Xspace{2}$ is a Hilbert space, the inner product of which is defined as  
\[
\langle \velocityfield, \velocityfieldother\rangle_{\Xspace{2}} = \int_0^1 \bigl\langle\velocityfield(t,\cdot), \velocityfieldother(t,\Cdot) \bigr\rangle_{\Vspace}\dint t \quad\text{for $\velocityfield, \velocityfieldother \in \Xspace{2}$}.
\]

More importantly, if the given velocity field is included in $\Xspace{2}$, a flow of diffeomorphisms can be obtained immediately.  
This result is given by the following theorem. 

\begin{theorem}[\cite{Yo10,BrHo15}]\label{thm:FlowDiffeo}
Let $\Vspace$ be an admissible Hilbert space, and the given $\velocityfield \in \Xspace{2}$ be a velocity field. 
Then the \ac{ODE} \cref{eq:FlowEq} has a unique solution 
$\diffeo^{\velocityfield} \in \Smooth^1([0,1] \times \domain, \domain)$, and 
for $t\in [0,1]$, the mapping $\diffeo_t^{\velocityfield}\colon \domain \rightarrow \domain$ is 
a $\Smooth^1$-diffeomorphism on $\domain$.
\end{theorem}

\section{Spatiotemporal imaging with diffeomorphic optimal transportation}
\label{sec:OT_recon_model}

This section proposes a joint variational model based on the thought here termed diffeomorphic optimal transportation.   

Although the minimizer for $\LpSpace^2$ MKP is expected to be one to one and sufficiently regular, the formula of interest  
cannot restrict it into the space with such regularity. 
As we know in \cref{sec:flow_diffeomorphisms}, the \cref{def:Admissible} and \cref{thm:FlowDiffeo} give us the inspiration on 
the regularity (i.e., admissible Hilbert space) that is required for the velocity field to generate a flow of 
diffeomorphisms by the \ac{ODE} \cref{eq:FlowEq}. One method to ensure a Hilbert space being admissible is to construct the space via a 
differential operator $\diffoperator$ (denoting its adjoint as $\diffoperator^{\dag}$) given by 
\begin{equation}\label{eq:diff_operator}
\langle \vfield, \vfieldother\rangle_{\Vspace} := \langle\diffoperator\vfield, \diffoperator\vfieldother\rangle_{ \LpSpace^2} = \langle \diffoperator^{\dag}\diffoperator\vfield, \vfieldother\rangle_{ \LpSpace^2},
\end{equation}
where $\LpSpace^2$ denotes the usual inner product space with square 
integrable vector fields defined on $\domain$.  

The way of choice for $\diffoperator$ is referred to \cite[Example 18]{BrHo15}. For $m > n/2 + 1$, 
the space $H^m(\domain)$ is an admissible space. However, the $\diffoperator$ is usually complex that 
difficultly be used in practice. 

\begin{remark}\label{re:remark1}
A kind of admissible Hilbert spaces is \ac{RKHS}, which is affiliated with a symmetric and positive-definite 
reproducing kernel \cite{Ar50,Yo10}. Assume that the $\Vspace$ is an \ac{RKHS} with a reproducing 
kernel $\kernel \colon \domain \times \domain \to \Matrix_{+}^{d \times d}$. And then a compactly self-adjoint 
operator $\Koperator \colon \LpSpace^2(\domain, \Real^n) \rightarrow \Vspace$ is uniquely defined by 
\[
\langle \vfield, \vfieldother\rangle_{ \LpSpace^2} = \langle \Koperator(\vfield), \vfieldother\rangle_{\Vspace},
\]
where $\Koperator(\vfield) = \int_{\Omega} \kernel(\Cdot, y)\vfield(y) \dint y$. 
Combined with \cref{eq:diff_operator}, the fact for $\vfield \in \Vspace$ is that 
\begin{equation}\label{eq:Koperator_relation}
\Koperator(\diffoperator^{\dag}\diffoperator)\vfield = \vfield.
\end{equation} 
\end{remark}

Consequently, if the \ac{RKHS} is taken into account, the reproducing kernel would be used 
rather than the $\diffoperator$ needs to be given explicitly. 
In what follows the space of vector fields is selected as the \ac{RKHS} with Gaussian reproducing kernel for 
the advantages of sufficient smoothness and fast computability \cite{ChGrOz19}.

\subsection{The proposed model}\label{sec:Wass_based_model} 

Here we propose the variational model for joint image reconstruction and motion estimation in spatiotemporal imaging. 

To guarantee the flow of deformations being diffeomorphic, we restrict the velocity field $\velocityfield$ into 
the admissible Hilbert space $\Xspace{2}$, namely, $\velocityfield \in \Xspace{2}$. By \cref{thm:FlowDiffeo}, the unique solution, denoted 
by $\diffeo_t^{\velocityfield}$ for $0 \leq t \leq 1$, to the \ac{ODE} \cref{eq:FlowEq} is 
determined by the given velocity field $\velocityfield$, which is a flow of $\Smooth^1$-diffeomorphisms 
with $\diffeo_0^{\velocityfield} = \Id$. For ease of description, we adopt the convention 
\begin{equation}\label{eq:FlowRelation}
\gelement{s,t}{\velocityfield} := \diffeo_t^{\velocityfield} \circ (\diffeo_s^{\velocityfield})^{-1} \quad\text{for $0 \leq t,s \leq 1$}. 
\end{equation} 
Using \cref{eq:FlowRelation}, we have  
\begin{equation}\label{eq:diff_trans}
\diffeo_t^{\velocityfield} = \diffeo_{0,t}^{\velocityfield}, \quad (\diffeo_t^{\velocityfield})^{-1} = \diffeo_{t,0}^{\velocityfield}.
\end{equation}  
Subsequently, we state a significant result by the following theorem. 
\begin{theorem}\label{thm:EquivalenceWD_flow} 
Assume that the time-dependent density function $\signal(t, x) \ge 0$ is  
appropriately smooth, the $\Vspace$ is an admissible Hilbert space, 
the velocity field $\velocityfield \in \Xspace{2}$, and $\signal_0$, $\signal_1$ are 
compactly supported on $\domain$. The Benamou--Brenier formula in \cref{thm:BB_formula} is equivalent to 
\begin{equation}\label{eq:high_order_Wasserstein_distance_ODE}
\begin{split}
 &\inf_{\substack{\velocityfield \in \Xspace{2}\\ \diffeo_{0,1}^{\velocityfield} . \template = \signal_1, \template \geq 0}} \int_0^1 \int_{\Omega}\diffeo_{0,t}^{\velocityfield} . \template(x)\vert\velocityfield(t, x)\vert^2\dint x\dint t \\
 & \quad\quad\text{s.t.  $\diffeo_{0,t}^{\velocityfield}$ solves \ac{ODE} \cref{eq:FlowEq},}
  \end{split}
\end{equation}  
where $\template = \signal_0$.
\end{theorem}
\begin{proof}
Suppose that $\signal$ and $\velocityfield$ solve \cref{eq:Wasserstein_term}. 
Define the diffeomorphism $\diffeoother(t,\cdot)$ that solves the \ac{ODE} \cref{eq:FlowEq} with the  
given $\velocityfield$. Since $\signal$ satisfies the \ac{PDE} constraint 
in \cref{eq:pde_constraint}, 
considering $t \mapsto \bigl\vert \Diff\diffeoother(t,\Cdot)\bigr\vert \signal\bigl(t, \diffeoother(t,\Cdot)\bigr)$, we have 
\begin{multline}\label{eq:rightterm}
   \dfrac{\dint}{\dint t}\Bigl( \bigl\vert \Diff\diffeoother(t,\Cdot)\bigr\vert \signal\bigl(t, \diffeoother(t,\Cdot)\bigr) \Bigr) \\ 
   = \bigl\vert \Diff\diffeoother(t,\Cdot) \bigr\vert \Bigl[\dfrac{d}{dt}\signal\bigl(t,\Cdot\bigr) +  \Div\Bigl(\signal\bigl(t,\Cdot\bigr) \velocityfield\bigl(t,\Cdot\bigr)\Bigr)\Bigr]\bigl(\diffeoother(t,\Cdot)\bigr) = 0.
\end{multline}
Hence,  $t \mapsto \bigl\vert \Diff\diffeoother(t,\Cdot)\bigr\vert \signal\bigl(t, \diffeoother(t,\Cdot)\bigr)$ 
is constant so in particular we have  
\[  \bigl\vert\Diff\diffeoother(t,\Cdot)\bigr\vert \signal\bigl(t, \diffeoother(t,\Cdot)\bigr) \equiv \signal(0, \Cdot) = \template. \]
Let $\diffeoother_t$ be $\gelement{0,t}{\velocityfield}$. 
Then $\signal\bigl(t, \Cdot\bigr) = \diffeo_{0,t}^{\velocityfield} . \template$.
Hence a solution to \cref{eq:Wasserstein_term} produces a 
solution to \cref{eq:high_order_Wasserstein_distance_ODE}.
It is simple to verify that a solution to \cref{eq:high_order_Wasserstein_distance_ODE} also produces a 
solution to \cref{eq:Wasserstein_term}. 
\end{proof}

Actually, the \cref{eq:FlowEq} is the characteristic \ac{ODE} of the \ac{PDE} \cref{eq:pde_constraint}. 
Note that the \cref{eq:high_order_Wasserstein_distance_ODE} implies that the distance 
between $\signal_0$ and $\signal_1$ can be seen as the transportation cost from $\signal_0$ to $\signal_1$ 
that characterized by the velocity field $\velocityfield(t, \Cdot)$ for $0 \leq t \leq 1$. And the velocity 
field $\velocityfield(t, \Cdot)$ with $t$ from $0$ to $1$ generates the flow of 
diffeomorphisms $\diffeo_{0,t}^{\velocityfield}$, combined with $\template \geq 0$, 
which leads to $\signal(t, \Cdot) = \diffeo_{0,t}^{\velocityfield} . \template$ also with compact support $\domain$ and nonnegativity. 

Inspired by the formulation in \cref{eq:high_order_Wasserstein_distance_ODE} and the strategy in \cite{ChGrOz19}, we construct  
the shape regularization $\RegFunc_2$ for the temporal deformation $\diffeo_t^{\velocityfield}$ by 
\begin{equation}\label{eq:RegTerm_tt}
\RegFunc_2(\diffeo_t^{\velocityfield}) := \int_0^t \int_{\Omega}\diffeo_{0,\tau}^{\velocityfield} . \template(x)\vert\velocityfield(\tau, x)\vert^2\dint x\dint \tau, 
\end{equation}
where the available $\diffeo_{0,\tau}^{\velocityfield} . \template$ is incorporated into the formula above as a weight function. 
Hence the proposed model with \ac{ODE} constraint under the framework in \cref{eq:VarReg_2} is formulated as  
\begin{equation}\label{eq:VarReg_LDDMM_2}
\begin{split}
 &\min_{\substack{\template \in \RecSpace \\ \velocityfield \in \Xspace{2}}} \int_{0}^{1} \left[\DataDisc_{\ForwardOp_t, \data_t}\bigl(\diffeo_{0,t}^{\velocityfield} . \template\bigr)  + 
   \mu_2 \int_0^t \int_{\Omega}\diffeo_{0,\tau}^{\velocityfield} . \template(x)\vert\velocityfield(\tau, x)\vert^2\dint x\dint \tau \right] \dint t   + \mu_1 \RegFunc_1(\template)  \\
 & \quad\,\, \text{s.t.  $\diffeo_{0,t}^{\velocityfield}$ solves \ac{ODE} \cref{eq:FlowEq},}
  \end{split}
\end{equation}   
where $\RecSpace$ denotes a certain space of real-valued functions with appropriate smoothness and nonnegativity. 

The model \cref{eq:VarReg_LDDMM_2} is termed \emph{time-continuous version} of 
the proposed model with \ac{ODE} constraint. Furthermore, this model can be restated as a \ac{PDE}-constrained 
optimal control formulation, which is given by the following theorem. 
\begin{theorem}\label{thm:EquivalencePDEOptProb_2}
Assume that $\RecSpace$ is a space of nonnegative real-valued functions with 
appropriate smoothness and compactly supported on $\domain$. 
Let $\template \in \RecSpace$ and $\signal(t,\Cdot) := \diffeo_{0,t}^{\velocityfield} . \template$ for $0 \leq t \leq 1$, 
where $\diffeo_{0,t}^{\velocityfield}$ solves \ac{ODE} \cref{eq:FlowEq}. 
Then \cref{eq:VarReg_LDDMM_2} equals to  
\begin{equation}\label{eq:LDDMMPDEConstrained_mp_2}
\begin{split}
 &\min_{\substack{\signal(0,\Cdot) \in \RecSpace \\ \velocityfield \in \Xspace{2}}} \int_{0}^{1} \left[\DataDisc_{\ForwardOp_t, \data_t}\bigl(\signal(t,\Cdot)\bigr) 
  +  \mu_2 \int_0^t \int_{\Omega}\signal(\tau,x)\vert\velocityfield(\tau, x)\vert^2\dint x\dint \tau \right] \dint t + \mu_1 \RegFunc_1\bigl(\signal(0,\Cdot)\bigr) \\
 & \quad\,\, \text{s.t. $\partial_t \signal(t, \Cdot) +  \grad\cdot\bigl(\signal(t, \Cdot)\, \velocityfield(t, \Cdot) \bigr)= 0$.}
  \end{split}
\end{equation}  
\end{theorem}
\begin{proof}
The proof can be readily obtained by following those of \cref{thm:EquivalenceWD_flow} and \cite[Theorem 3.5]{ChGrOz19}. 
\end{proof}

\begin{remark}\label{rem:mp_choose}
As described in \cref{sec:OT_principle}, the mass-preserving deformation is required under the principle 
of optimal transportation. Hence, we merely consider such deformation in \cref{thm:EquivalencePDEOptProb_2}.  
\end{remark} 

Therefore, the investigation on spatiotemporal imaging can be motivated by  
the perspective of \ac{PDE}-constrained optimal control. The other purpose of the \ac{PDE}-constrained formulation  
is that it can be simply used to compare against the \ac{PDE} based alternatives, 
for instance, the one based on conventional Wasserstein distance in \cite{Br10}. 
More details are provided in \cref{sec:WD_based_model_related}.

\subsection{Time discretization}
\label{sec:time_discretized_version}

The time-discretized version of the proposed model is useful to the practical applications. Without loss of generality, 
assume that the acquired data is collected at equally discretized time point through the gating method. 

We suppose that the sampling is performed on a uniform partition over $[0,1]$ with $\{t_i = i/N\}$ 
for $0 \leq i \leq N$, which acts as the gating grid. Then the time-discretized formulation of the 
general spatiotemporal imaging in \cref{eq:InvProb} becomes 
\begin{equation}\label{eq:InvProb_discrete}  
     \data(t_i,\Cdot) = \ForwardOp_{t_i}\bigl(\signal(t_i,\Cdot)\bigr) + \noisedata(t_i,\Cdot). 
\end{equation}

Hence, one of the time discretization of \cref{eq:VarReg_LDDMM_2} is formulated as 
\begin{equation}\label{eq:time_discretized_VarReg_LDDMM_2}
\begin{split}
 &\min_{\substack{\template \in \RecSpace \\ \velocityfield \in \Xspace{2}}} \frac{1}{N}\sum_{i=1}^{N} \left[\DataDisc_{\ForwardOp_{t_i}, \data_{t_i}}\bigl(\diffeo_{0,t_i}^{\velocityfield} . \template \bigr)  + 
   \mu_2 \int_0^{t_i} \int_{\Omega}\diffeo_{0,\tau}^{\velocityfield} . \template(x)\vert\velocityfield(\tau, x)\vert^2\dint x\dint \tau \right]  + \mu_1 \RegFunc_1(\template)  \\
 & \quad\,\, \text{s.t.  $\diffeo_{0,t}^{\velocityfield}$ solves \ac{ODE} \cref{eq:FlowEq}.}
  \end{split}
\end{equation}  
Actually, using the form \cref{eq:time_discretized_VarReg_LDDMM_2} means no projection data is acquired at $t=0$. 
The unknown template acts as a ``virtual" image, and merely its warped version $\diffeo_{0,t_i}^{\velocityfield} . \template$ 
relates. The ``virtual" image was also considered in \cite{BoTh16} for gated \ac{PET}/\ac{CT} imaging. 

If the projection data is assumed to acquire at $t=0$, then the \cref{eq:VarReg_LDDMM_2} can be time-discretized as 
\begin{equation}\label{eq:time_discretized_VarReg_LDDMM_another}
\begin{split}
 &\min_{\substack{\template \in \RecSpace \\ \velocityfield \in \Xspace{2}}} \frac{1}{N+1}\sum_{i=0}^{N} \left[\DataDisc_{\ForwardOp_{t_i}, \data_{t_i}}\bigl(\diffeo_{0,t_i}^{\velocityfield} . \template \bigr)  + 
   \mu_2 \int_0^{t_i} \int_{\Omega}\diffeo_{0,\tau}^{\velocityfield} . \template(x)\vert\velocityfield(\tau, x)\vert^2\dint x\dint \tau \right]  + \mu_1 \RegFunc_1(\template)  \\
 & \quad\,\, \text{s.t.  $\diffeo_{0,t}^{\velocityfield}$ solves \ac{ODE} \cref{eq:FlowEq},}
  \end{split}
\end{equation}  
where the unknown image at the initial gate acts as the template. 

To get satisfying result, the template in \cref{eq:time_discretized_VarReg_LDDMM_another} should be 
reconstructed as accurate as possible. However, for \cref{eq:time_discretized_VarReg_LDDMM_2}, 
even if the reconstructed ``virtual" template is not so accurate, the velocity field can be used to make correction 
to obtain the desirable sequential images $\diffeo_{0,t_i}^{\velocityfield} . \template$ ($1 \leq i \leq N$) to some extent. 
In this article, the time-discretized scheme \cref{eq:time_discretized_VarReg_LDDMM_2} is adopted  
to perform numerical implementation.

\subsection{Comparison with several existing alternatives}
\label{sec:WD_based_model_related}

In this section, the analytical comparison will be conducted between the proposed model and 
the existing alternatives.  

\subsubsection{The joint variational model based on Wasserstein distance} 

A variational model based on optimal transportation was proposed for joint motion estimation and 
image reconstruction in \cite{Br10}, which is also based on the $\LpSpace^2$ Wasserstein distance but in different framework.  
This model is formulated as the following \ac{PDE}-constrained optimal control problem, which is written as  
\begin{equation}\label{eq:Wasserstein_distancePDEConstrained}
\begin{split}
 &\min_{\substack{\signal \ge 0, \velocityfield}} \int_{0}^{1} \left[\DataDisc_{\ForwardOp_t, \data_t}\bigl(\signal(t,\Cdot)\bigr)   
  + \mu_1 \RegFunc_1\bigl(\signal(t,\Cdot)\bigr) +  \mu_2 \int_{\Omega}\signal(t, x)\vert\velocityfield(t, x)\vert^2\dint x \right]  \dint t \\
 &\quad \text{s.t. $\partial_t \signal(t, \Cdot) +  \grad\cdot\bigl(\signal(t, \Cdot)\, \velocityfield(t, \Cdot) \bigr)= 0$.}
  \end{split}
\end{equation}  

It is easy to figure out that the \ac{PDE} constraints in \cref{eq:LDDMMPDEConstrained_mp_2} 
and \cref{eq:Wasserstein_distancePDEConstrained} are the same. Through analysis, the obvious distinctions  
between them relate to the constrained space with respect to the velocity field $\velocityfield$ and the selection of 
the shape regularization term $\RegFunc_2$. 

As given in \cref{eq:Wasserstein_distancePDEConstrained}, the constrained space of velocity fields is  
$\LpSpace^2([0, 1], \domain)$, which leads to the velocity field lacking sufficient smoothness on $\domain$. 
In contrast to \cref{eq:Wasserstein_distancePDEConstrained}, 
the space of velocity fields in \cref{eq:LDDMMPDEConstrained_mp_2} is restricted in $\Xspace{2}$, which means  
the velocity field at every time point located in an admissible Hilbert space, namely, a sufficiently smooth vector-valued 
function distributed on $\Omega$. 
This guarantees the non-rigid diffeomorphic deformations, to some extent, which fulfils the physical 
mechanism \cite{GiRuBu12,BuMoRu13}. 
Moreover,  the regularization on $\velocityfield$ in \cref{eq:Wasserstein_distancePDEConstrained} is merely  
achieved by the $\RegFunc_2(\diffeo_t)$ at $t=1$. However, 
the model \cref{eq:LDDMMPDEConstrained_mp_2} make use of $\RegFunc_2(\diffeo_t)$ for all of the time. 

Besides those above, both approaches also differ in the selection of regularization 
term $\RegFunc_1$. The \cref{eq:LDDMMPDEConstrained_mp_2} only poses regularization on 
the initial image $\signal(0, \Cdot)$, whereas in \cref{eq:Wasserstein_distancePDEConstrained} the whole 
time trajectory $t \mapsto \signal(t, \Cdot)$ is regularized. 
Bear in mind that the $\signal(t, \Cdot)$ is determined by $\signal(0, \Cdot)$ and $\velocityfield$. 
Hence, \cref{eq:LDDMMPDEConstrained_mp_2} 
not only has a simpler formulation, but also gets the advantage for numerical implementation. 

For further comparison, an equivalent result is given by the following theorem if 
we change the constrained space of velocity fields into $\Xspace{2}$ for \cref{eq:Wasserstein_distancePDEConstrained}. 

\begin{theorem}\label{thm:EquivalencePDEOpt_ODEOpt}
Suppose that the assumptions in \cref{thm:EquivalenceWD_flow} hold. Let $\RecSpace$ be a space of nonnegative real-valued functions with 
appropriate smoothness and compactly supported on $\domain$, and $\signal(0, \Cdot) \in \RecSpace$, which is denoted by 
template $\template$. Then \cref{eq:Wasserstein_distancePDEConstrained} is equivalent to 
\begin{equation}\label{eq:Wasserstein_distancePDEConstrained_ode}
\begin{split}
 &\min_{\substack{\template \in \RecSpace \\ \velocityfield \in \Xspace{2}}} \int_{0}^{1} \left[\DataDisc_{\ForwardOp_t, \data_t}\bigl(\diffeo_{0,t}^{\velocityfield} . \template\bigr)   + \mu_1 \RegFunc_1\bigl(\diffeo_{0,t}^{\velocityfield} . \template\bigr) 
  +  \mu_2 \int_{\Omega}\diffeo_{0,t}^{\velocityfield} . \template(x)\vert\velocityfield(t, x)\vert^2\dint x \right] \dint t \\
 &\quad\,\, \text{s.t.  $\diffeo_{0,t}^{\velocityfield}$ solves \ac{ODE} \cref{eq:FlowEq},}
  \end{split}
\end{equation}
where $\diffeo_{0,t}^{\velocityfield} . \template$ is defined as \cref{eq:MassPreservedDeform}. 
\end{theorem}
\begin{proof}
The proof can be readily obtained by following those of \cref{thm:EquivalenceWD_flow} and \cref{thm:EquivalencePDEOptProb_2}. 
\end{proof} 

Let $\GoalFunctionalV_{W}(\template, \velocityfield)$ be the objective function 
in \cref{eq:Wasserstein_distancePDEConstrained_ode}, and $\Vspace$ be an \ac{RKHS}. Following the derivations of the 
proof in \cref{thm:energy_functional_derivative_2}, 
we get the $\Xspace{2}$-gradient with regard to the velocity field $\velocityfield$ as 
\begin{multline}\label{eq:compare_Wasserstein_velocityfield}
\grad^{\,\Vspace}_{\velocityfield}\GoalFunctionalV_{W}(\template, \velocityfield)(t,\Cdot)  \\
    = \Koperator\biggl(\diffeo_{0,t}^{\velocityfield} . \template \int_{t}^{1} \grad\Bigl(\Bigl[\partial\DataDisc_{\ForwardOp_{\tau}, g_{\tau}} \bigl(\diffeo_{0,\tau}^{\velocityfield} . \template\bigr) + \mu_1 \partial \RegFunc_1\bigl(\diffeo_{0,\tau}^{\velocityfield} . \template\bigr)  + \mu_2 \vert\velocityfield(\tau, \cdot)\vert^2\Bigr]\bigl( \gelement{t,\tau}{\velocityfield}\bigr)\Bigr) \dint \tau  \\ 
    + 2\mu_2\diffeo_{0,t}^{\velocityfield} . \template\, \velocityfield(t,\Cdot) \biggr)
\end{multline}
for $0\leq t \leq 1$ and the gradient (i.e., $\LpSpace^2$-gradient if not indicated) in terms of the template $\template$ as 
\begin{equation}\label{eq:compare_Wasserstein_template}
\grad_{\template}\GoalFunctionalV_{W}(\template, \velocityfield) 
    = \int_{0}^{1} \Bigl[\partial\DataDisc_{\ForwardOp_t, g_t} \bigl(\diffeo_{0,t}^{\velocityfield} . \template\bigr) + \mu_1 \partial \RegFunc_1\bigl(\diffeo_{0,t}^{\velocityfield} . \template\bigr)  
    + \mu_2 \vert\velocityfield(t, \cdot)\vert^2\Bigr]\bigl( \gelement{0,t}{\velocityfield}\bigr) \dint t. 
\end{equation}

Assume that $(\template^{\ast}, \velocityfield^{\ast})$ is the solution to the problem \cref{eq:Wasserstein_distancePDEConstrained_ode}, 
which should satisfy the following optimality conditions 
 \begin{equation}\label{eq:optimality_conditions_Wasserstein}
 \begin{cases} 
    \grad^{\,\Vspace}_{\velocityfield}\GoalFunctionalV_{W}(\template^{\ast}, \velocityfield^{\ast}) = 0, & \\[0.5em]
    \grad_{\template}\GoalFunctionalV_{W}(\template^{\ast}, \velocityfield^{\ast}) - \lambda^{\ast} = 0, & \\[0.5em]
    \lambda^{\ast} \geq 0, \quad \template^{\ast} \geq 0, \quad \lambda^{\ast}\template^{\ast} = 0, 
   \end{cases} 
\end{equation}
where $\lambda^{\ast}$ is the function of Lagrange multiplier.

In particular, we consider the optimal velocity field at the end points $t = 0$ and $t = 1$. If $t = 0$ and $1$, 
using \cref{eq:compare_Wasserstein_velocityfield}--\cref{eq:optimality_conditions_Wasserstein}, 
and considering the symmetric and positive-definite reproducing kernel, then we have 
\[
\template^{\ast}\velocityfield^{\ast}(0,\Cdot) = 0 \quad \text{and} \quad \diffeo_{0,1}^{\velocityfield^{\ast}} . \template^{\ast}\velocityfield^{\ast}(1,\Cdot) = 0.
\]
Hence $\velocityfield^{\ast}(0,\Cdot)$ and $\velocityfield^{\ast}(1,\Cdot)$ are vanishing on the supports 
of $\template^{\ast}$ and $\diffeo_{0,1}^{\velocityfield^{\ast}} . \template^{\ast}$ respectively. 
Namely, the optimal velocity field that minimizes \cref{eq:Wasserstein_distancePDEConstrained_ode} 
is vanishing on the supports of the associated images to be reconstructed at the end time points. 

On the other hand, suppose that $(\bar{\template}, \bar{\velocityfield})$ is the solution to the 
problem \cref{eq:VarReg_LDDMM_2}, combined with \cref{thm:energy_functional_derivative_2}, 
which fulfils the following optimality conditions 
 \begin{equation}\label{eq:optimality_conditions_proposed}
 \begin{cases} 
    \grad^{\,\Vspace}_{\velocityfield}\GoalFunctionalV_{C}(\bar{\template}, \bar{\velocityfield}) = 0, & \\[0.5em]
    \grad_{\template}\GoalFunctionalV_{C}(\bar{\template}, \bar{\velocityfield}) - \bar{\lambda} = 0, & \\[0.5em]
    \bar{\lambda} \geq 0, \quad \bar{\template} \geq 0, \quad \bar{\lambda}\bar{\template} = 0, 
   \end{cases} 
\end{equation}
where $\bar{\lambda}$ is the associated Lagrange multiplier function. Using \cref{eq:optimality_conditions_proposed} and 
\cref{thm:energy_functional_derivative_2}, we immediately observe that the optimal velocity field that 
minimizes \cref{eq:VarReg_LDDMM_2} is unnecessarily vanishing on the supports of the reconstructed images at the  
end time points. Following \cref{thm:energy_functional_time_discretized_derivative}, 
the optimal velocity field to the related time-discretized model is also unnecessarily vanishing at the same domains as the above. 
This statement has been also demonstrated numerically by the computed optimal velocity field in \cref{Test_suite_1:display_velocity_field}. 
Hence, this implies the consistency between the time-continuous model and its associated time-discretized counterpart.

\subsubsection{The joint variational model based on \ac{LDDMM}}
\label{sec:diffeomorphic_motion_model }

We then recall the \ac{LDDMM} based model for joint image reconstruction and motion estimation in \cite{ChGrOz19}, 
which implies an \ac{LDDMM} consistent growth model. Using the same notation, that model with mass-preserving deformation 
in the form of \ac{PDE}-constrained optimal control can be written as  
\begin{equation}\label{eq:LDDMMPDEConstrained_mp_19}
\begin{split}
 &\min_{\substack{\signal(0,\Cdot) \in \RecSpace \\ \velocityfield \in \Xspace{2}}} \int_{0}^{1} \left[\DataDisc_{\ForwardOp_t, \data_t}\bigl(\signal(t,\Cdot)\bigr) 
  +  \mu_2 \int_0^t \int_{\Omega}\vert\diffoperator\velocityfield(\tau, x)\vert^2\dint x\dint \tau \right] \dint t + \mu_1 \RegFunc_1\bigl(\signal(0,\Cdot)\bigr) \\
 & \quad\,\, \text{s.t. $\partial_t \signal(t, \Cdot) +  \grad\cdot\bigl(\signal(t, \Cdot)\, \velocityfield(t, \Cdot) \bigr)= 0$.}
  \end{split}
\end{equation}  

It is obvious that both \cref{eq:LDDMMPDEConstrained_mp_2} 
and \cref{eq:LDDMMPDEConstrained_mp_19} constrain the velocity field into the $\Xspace{2}$ space. 
However, the difference is situated on the choice of shape regularization $\RegFunc_2(\diffeo_t)$ for $0\leq t \leq 1$. 
Under the framework of shape theory, in \cref{eq:LDDMMPDEConstrained_mp_19} the squared shape distance acts as the shape regularization 
term (see \cite{ChGrOz19}). No any weight is involved or all of the involved weight values are one. 
It would make sense if the deformation, for instance the geometric deformation, merely moves the position of the 
pixel/voxel but does not change its intensity. 

As stated in \cref{sec:problem_setting}, the mass-preserving non-rigid deformation not only moves the position of 
the pixel/voxel but also changes its intensity. In contrast, the proposed model uses the squared $\LpSpace^2$ Wasserstein distance. 
Particularly, a time-dependent weight function is introduced into the new shape regularization term under the framework of 
optimal transportation, which is chosen as the unknown time-series image/density $\signal(t,\Cdot)$. 
In other words, the weight is the nonnegative density $\signal(t,\Cdot)$, which is positive on the support of $\signal(t,\Cdot)$, otherwise is zero. 
This implies the penalty is only put on the range of the objects in the image, which equivalently means the transportation 
cost is just originated from the sites having the objects to be transported. 

This also demonstrates that the proposed model using diffeomorphic optimal transport combines the 
thoughts of \ac{LDDMM} and optimal transportation.  Specially, the other alternative is further constructed 
in \cref{sec:Alternative_model}. Through analyzing these models, the relationship 
between \ac{LDDMM} and optimal transportation would be more clear.

\section{Numerical implementation}
\label{sec:computed_method}

To validate the proposed model, we consider a specific example using the time-discretized 
model \cref{eq:time_discretized_VarReg_LDDMM_2}. The detailed numerical implementation will be presented.

\subsection{A specific example}\label{sec:pratical_example}

As a specific example in \ac{CT}, the data fidelity term is usually constructed  
as the squared $\LpSpace^2$-norm, and the spatial regularization is often 
selected as the \ac{TV} functional. More clearly, 
\begin{align}
\label{eq:data_fidelity_realized}\DataDisc_{\ForwardOp_{t_i}, \data_{t_i}}\bigl(\diffeo_{0,t_i}^{\velocityfield} . \template \bigr) &:= \|\ForwardOp_{t_i}\bigl(\diffeo_{0,t_i}^{\velocityfield} . \template\bigr) - \data(t_i,\Cdot)\|_2^2, \\
\label{eq:spatial_regularization_realized}\RegFunc_1(\template) &:= \|\nabla \template\|_1, 
\end{align}
then \cref{eq:time_discretized_VarReg_LDDMM_2} is specified by 
\begin{equation}\label{eq:VarReg_LDDMM_time_discretized}
\begin{split}
 &\min_{\substack{\template \in \RecSpace \\ \velocityfield \in \Xspace{2}}} \frac{1}{N}\sum_{i=1}^{N}\left[ \|\ForwardOp_{t_i}\bigl(\diffeo_{0,t_i}^{\velocityfield} . \template\bigr) - \data(t_i,\Cdot)\|_2^2  + 
   \mu_2 \int_0^{t_i} \int_{\Omega}\diffeo_{0,\tau}^{\velocityfield} . \template(x)\vert\velocityfield(\tau, x)\vert^2\dint x\dint \tau \right]  + \mu_1 \|\nabla \template\|_1 \\
 & \quad\,\, \text{s.t.  $\diffeo_{0,t}^{\velocityfield}$ solves \ac{ODE} \cref{eq:FlowEq}.}
  \end{split}
\end{equation}  
Here the mass-preserving deformation is applied as mentioned previously, i.e., 
\[
\diffeo_{0,t}^{\velocityfield} . \template : = \bigl\vert \Diff(\diffeo_{t,0}^{\velocityfield}) \bigr\vert \template \circ \diffeo_{t,0}^{\velocityfield},
\] 
and $\RecSpace$ is assumed to be the nonnegative $BV(\Omega)$. 

Remark that the proposed model serves as a general model to deal with the motion compensated image reconstruction in various 
imaging modalities in spatiotemporal setting. One would select the required data fidelity and spatial regularization for different 
imaging modalities. The following algorithm also can be presented in the general scheme. 

We apply the alternating minimization algorithm to solve the 
model \cref{eq:VarReg_LDDMM_time_discretized} for the involved variables being mutually coupled.  
More specifically, fixed the velocity field $\velocityfield$, the flow of 
diffeomorphisms $\diffeo_{0,t}^{\velocityfield}$ is generated by the \ac{ODE} \cref{eq:FlowEq}. 
Then, the original problem \cref{eq:VarReg_LDDMM_time_discretized} boils down  
to the following modified static image reconstruction problem 
\begin{equation}\label{eq:VarReg_LDDMM_template_time_discrete2}
 \min_{\template \in \RecSpace}\frac{1}{N}\sum_{i=1}^{N} \left[\|\ForwardOp_{t_i}\bigl(\diffeo_{0,t_i}^{\velocityfield} . \template\bigr) - \data(t_i,\Cdot)\|_2^2  + 
   \mu_2 \int_0^{t_i} \int_{\Omega}\diffeo_{0,\tau}^{\velocityfield} . \template(x)\vert\velocityfield(\tau, x)\vert^2\dint x\dint \tau \right]  + \mu_1 \|\nabla \template\|_1.
 \end{equation} 
In contrast, given the template $\template$, then original problem \cref{eq:VarReg_LDDMM_time_discretized} 
reduces to a sequentially indirect image registration, where we estimate the velocity 
field $\velocityfield$ from the time-series data that are indirect observations of the target by 
\begin{equation}\label{eq:VarReg_LDDMM_deformation_time_discrete2}
\begin{split}
 & \min_{\velocityfield \in \Xspace{2}}\frac{1}{N}\sum_{i=1}^{N} \left[\|\ForwardOp_{t_i}\bigl(\diffeo_{0,t_i}^{\velocityfield} . \template\bigr) - \data(t_i,\Cdot)\|_2^2  + 
   \mu_2 \int_0^{t_i} \int_{\Omega}\diffeo_{0,\tau}^{\velocityfield} . \template(x)\vert\velocityfield(\tau, x)\vert^2\dint x\dint \tau \right]  \\
 &\quad\,\,  \text{s.t.  $\diffeo_{0,t}^{\velocityfield}$ solves \ac{ODE} \cref{eq:FlowEq}.}
  \end{split}
  \end{equation}
Let $\GoalFunctionalV_{\velocityfield} \colon \RecSpace \to \Real$ and 
$\GoalFunctionalV_ {\template} \colon \Xspace{2} \to \Real$ be the objective 
functionals in \cref{eq:VarReg_LDDMM_template_time_discrete2} and \cref{eq:VarReg_LDDMM_deformation_time_discrete2}, respectively. 
We figure out \cref{eq:VarReg_LDDMM_time_discretized} by solving 
for \cref{eq:VarReg_LDDMM_template_time_discrete2} and \cref{eq:VarReg_LDDMM_deformation_time_discrete2} alternately, i.e.,
\begin{equation}\label{eq:Alternating}
\begin{cases}
  \template^{k+1} := \text{the solution to \cref{eq:VarReg_LDDMM_template_time_discrete2} with fixed $\velocityfield = \velocityfield^k$,} \\[0.5em]
  \velocityfield^{k+1} := \text{the solution to \cref{eq:VarReg_LDDMM_deformation_time_discrete2} with fixed $\template = \template^{k+1}$,} 
\end{cases}
\end{equation}
or by the scheme via changing the updating order in \cref{eq:Alternating}.

\subsection{Template reconstruction}
\label{sec:Template_reconstruction}

In what follows we construct the algorithm for template reconstruction by solving the static 
image reconstruction problem in \cref{eq:VarReg_LDDMM_template_time_discrete2}. 

The subproblem in \cref{eq:VarReg_LDDMM_template_time_discrete2} is a 
nonsmooth minimization. We modify the nonsmooth \ac{TV} term into the smooth one as 
\begin{equation}\label{eq:template_recon_CT} 
\|\nabla \template\|_1 \approx \int_{\domain}\vert \nabla \template(x) \vert_{2, \epsilon}\dint x := \int_{\domain} \Bigl(\sum_i\bigl(\partial_i\template(x)\bigr)^2 + \epsilon\Bigr)^{1/2}\dint x, 
\end{equation}
where $\epsilon > 0$ is sufficiently small, e.g., $\epsilon = 10^{-12}$. This also implies that 
we reconsider a smoothed model of \cref{eq:VarReg_LDDMM_time_discretized} by the modification above, 
which is a often used smoothed strategy for \ac{TV} regularization in image reconstruction. 

Then using \cref{thm:energy_functional_time_discretized_derivative}, 
the smoothed version of \cref{eq:VarReg_LDDMM_template_time_discrete2} can be solved by the following 
projected gradient descent scheme: 
\begin{equation}\label{eq:soc4}
\template^{k+1} = \textsf{Proj}_{\geq 0}\biggl\{\template^k - \alpha^k\grad\GoalFunctionalV_{\velocityfield}(\template^k)\biggr\}, 
\end{equation}
where  \[
 \grad\GoalFunctionalV_{\velocityfield}(\template)
    =  \frac{1}{N}\sum_{i=1}^{N}\Bigl(h_{0, t_i}^{\template, \velocityfield}  + \mu_2 \eta_{0, t_i}^{\velocityfield} \Bigr) + \mu_1\grad^{\,\ast}\biggl(\frac{\grad \template}{\vert \nabla \template \vert_{2, \epsilon}}\biggr). 
  \]
Here $\textsf{Proj}_{\geq 0}$ means the projection operator onto the space with nonnegativity, and $\alpha^k$ the stepsize for the $k$-th iteration. Furthermore, by \cref{eq:data_fidelity_realized}, and \cref{eq:h_t_ti_general}, \cref{eq:eta_t_ti_general}, for $t_i \geq t$, we have 
\begin{align}\label{eq:h_t_ti_concrete}
h_{t, t_i}^{\template, \velocityfield} &= 2\ForwardOpAdjoint_{t_i}\bigl(\ForwardOp_{t_i}(\diffeo_{0, t_i}^{\velocityfield}.\template) - \data(t_i,\Cdot) \bigr)(\diffeo_{t,t_i}^{\velocityfield}), \\
\label{eq:eta_t_ti_concrete}
\eta_{t, t_i}^{\velocityfield} &= \int_{t}^{t_i} \vert\velocityfield(\iota, \cdot)\vert^2\bigl( \gelement{t,\iota}{\velocityfield}\bigr)\dint \iota, 
\end{align}
where $\mathcal{T}$ is assumed to be linear, and $\mathcal{T}^{\ast}$ denotes its adjoint operator. 
The numerical implementation for the scheme \cref{eq:soc4} is 
given in \cref{algo:GDSB_4DCT}.

For solving the nonsmooth problem above, the convex optimization techniques can be applied but
need to introduce more auxiliary variables and parameters than the above algorithm. 
As did in \cite{ChGrOz19}, to optimize the whole problem \cref{eq:VarReg_LDDMM_time_discretized} 
efficiently, we still employ the iterative scheme \cref{eq:soc4} to solve the subproblem.

\subsection{Velocity field estimation}
\label{sec:GradientDescent_velocithfield}

Here we present an algorithm for solving the sequentially indirect image 
registration \cref{eq:VarReg_LDDMM_deformation_time_discrete2}. To guarantee the velocity field 
constrained in $\Xspace{2}$, and then resulting a flow of diffeomorphisms through \ac{ODE} \cref{eq:FlowEq}, we use the 
gradient descent scheme based on $\Xspace{2}$-gradient. By \cref{thm:energy_functional_time_discretized_derivative}, 
the scheme is written as  
\begin{equation}\label{eq:gradientflow}
  \velocityfield^{k+1} =  \velocityfield^k - \beta^k\grad^{\,\Vspace}\!\!\GoalFunctionalV_{\template}(\velocityfield^k),  
\end{equation}
where  
\begin{equation*}
 \grad^{\,\Vspace}\!\!\GoalFunctionalV_{\template}(\velocityfield)(t,\Cdot) 
    =  \frac{1}{N}\sum_{\{i\geq 1 : t_i \geq t\}} \Koperator\biggl( \gelement{0, t}{\velocityfield} . \template \Bigl[\grad \bigl(h_{t, t_i}^{\template,\velocityfield}  + \mu_2 \eta_{t, t_i}^{\velocityfield}\bigr)  + 2\mu_2 \velocityfield_{t, t_i}\Bigr]\biggr), 
  \end{equation*}
and $\beta^k$ is the stepsize in the $k$-th iteration, by \cref{eq:Middle_func_deriv_2_general}, for $t_i \geq t$,
\begin{equation}
\label{eq:Middle_func_deriv_2_concrete}
\velocityfield_{t, t_i} = \velocityfield(t,\Cdot). 
\end{equation}
Here $h_{t, t_i}^{\template,\velocityfield}$ and $\eta_{t, t_i}^{\velocityfield}$ are 
defined by \cref{eq:h_t_ti_concrete} and \cref{eq:eta_t_ti_concrete}, respectively. 
The detailed implementation for the scheme \cref{eq:gradientflow} is 
given in \cref{algo:GradientDescentAlgorithmForFiniteFunctional_0}.

As indicated in \cref{re:remark1}, here we use the \ac{RKHS} with a symmetric and positive-definite 
Gaussian kernel $\kernel \colon \domain \times \domain \to \Matrix_{+}^{n \times n}$, for instance 
defined by \cref{eq:KernelEq} for $n=2$, and then the operator $\Koperator \colon \LpSpace^2(\domain, \Real^n) \rightarrow \Vspace$ is uniquely defined by 
\[
\langle \vfield, \vfieldother\rangle_{ \LpSpace^2} = \langle \Koperator(\vfield), \vfieldother\rangle_{\Vspace},
\]
where $\Koperator(\vfield) = \int_{\Omega} \kernel(\Cdot, y)\vfield(y) \dint y$. 

As a result, the sequence $\{\velocityfield^k\}$ generated by \cref{eq:gradientflow} locate  
in $\Xspace{2}$ if the initial value of velocity field is selected in the same space, say the given zero velocity field, which ultimately 
leads to the cluster points located in $\Xspace{2}$, 
and further produces the flow of diffeomorphisms through \ac{ODE} \cref{eq:FlowEq}. 
In contrast, if one uses the $\LpSpace^2$-gradient descent scheme in \cref{eq:gradientflow}, 
the resulting sequence and its cluster points would locate in $\LpSpace^2([0,1], \domain)$ instead. 
This cannot generate a flow of diffeomorphisms for lacking sufficient smoothness. 
Since the proposed model is nonconvex with regard to the velocity field, the different solving scheme would lead to the  
different local minimum (see the numerical comparison in test suite 1 of \cref{sec:test_suite_1}).

\subsection{Numerical discretization}
\label{sec:Discretization}

The period $[0, 1]$ is discretized uniformly 
into $MN$ partitions. Then a discretized time grid is formulated   
as $\{\tau_j = j/(MN)\}$ for $j = 0, 1, \ldots, MN$. Hence, $\tau_{iM} = t_i$ 
for $i = 0, 1, \ldots, N$. In other words, each partition $[t_i, t_{i+1}]$ is subdivided  
into $M$ even segments. Evidently, we have $\tau_i = t_i$ when $M=1$, which means the discretized time grid 
is consistent with the gating grid. The $M$ is called the factor of discretized 
time degree, which determines the fineness of the grid along the temporal axis. 

Solving the \ac{ODE} \cref{eq:FlowEq} numerically, the deformations $\gelement{\tau_i, \tau_{i-1}}{\velocityfield}$ 
and $\gelement{\tau_i, \tau_{i+1}}{\velocityfield}$ can be computed by
\begin{equation}\label{eq:deformation_approx2}
\gelement{\tau_j, \tau_{j-1}}{\velocityfield} \approx  \Id - \frac{1}{MN}\velocityfield(\tau_j,\Cdot),
\end{equation}
and 
\begin{equation}\label{eq:deformation_approx1}
\gelement{\tau_j, \tau_{j+1}}{\velocityfield} \approx  \Id + \frac{1}{MN}\velocityfield(\tau_j,\Cdot).
\end{equation}

By \cref{eq:FlowRelation} and \cref{eq:deformation_approx2}, we have  
\begin{equation}\label{eq:deformation_approx3}
\gelement{\tau_j,0}{\velocityfield} \approx  \gelement{\tau_{j-1},0}{\velocityfield} \circ \Bigl(\Id - \frac{1}{MN}\velocityfield(\tau_j,\Cdot)\Bigr)
\end{equation} 
for $j = 1, 2, \ldots, MN$. Similarly, combining \cref{eq:FlowRelation} and \cref{eq:deformation_approx1}, we obtain the following formula  
\begin{equation}\label{eq:deformation_approx4}
\gelement{\tau_j, t_i}{\velocityfield} \approx  \gelement{\tau_{j+1}, t_i}{\velocityfield} \circ \Bigl(\Id + \frac{1}{MN}\velocityfield(\tau_j,\Cdot)\Bigr) 
\end{equation}
for $j = iM-1, iM-2, \ldots, 0$, where $\gelement{t_i,t_i}{\velocityfield} = \Id$. 

Then using \cref{eq:deformation_approx3}, the Jacobian determinant can be calculated by
\begin{equation}\label{eq:deformation_approx5}
\bigl\vert \Diff(\gelement{\tau_j, 0}{\velocityfield})\bigr\vert \approx \Bigl(1 - \frac{1}{MN} \Div\velocityfield(\tau_j,\Cdot)\Bigr) 
 \bigl\vert  \Diff(\gelement{\tau_{j-1},0}{\velocityfield})\bigr\vert  \circ \Bigl(\Id - \frac{1}{MN}\velocityfield(\tau_j,\Cdot) \Bigr) 
 \end{equation}
for $j = 1, 2, \ldots, MN$. Here $\gelement{0,0}{\velocityfield} = \Id$ 
and $\bigl\vert \Diff(\gelement{0,0}{\velocityfield})\bigr\vert = 1$.

As given in \cref{eq:soc4}, updating the template requires to compute the mass-preserving 
deformations like $\diffeo_{0,{t_i}}^{\velocityfield} . \template = \bigl\vert\Diff\bigl(\diffeo_{t_i,0}^{\velocityfield}\bigr)\bigr\vert \template \circ\gelement{t_i,0}{\velocityfield}$. 

By \cref{eq:deformation_approx3}, we have the following estimate 
\begin{equation}\label{eq:deformation_approx7}
\template \circ \gelement{\tau_j,0}{\velocityfield} \approx  \bigl(\template \circ \gelement{\tau_{j-1},0}{\velocityfield}\bigr) \circ \Bigl(\Id - \frac{1}{MN}\velocityfield(\tau_j,\Cdot)\Bigr)
\end{equation}
for $j = 1, 2, \ldots, MN$, and $\template \circ \gelement{0,0}{\velocityfield} = \template$. 
Multiplying \cref{eq:deformation_approx7} with \cref{eq:deformation_approx5}, we have the significant update 
\begin{equation}\label{eq:deformation_approx_template}
\gelement{0, \tau_j}{\velocityfield} . \template  \approx  \Bigl(1 - \frac{1}{MN} \Div\velocityfield(\tau_j,\Cdot)\Bigr) 
 \bigl(\gelement{0, \tau_{j-1}}{\velocityfield} . \template\bigr)  \circ \Bigl(\Id - \frac{1}{MN}\velocityfield(\tau_j,\Cdot) \Bigr). 
\end{equation}

Then for \cref{eq:h_t_ti_concrete}, the \cref{eq:deformation_approx4} also yields the following approximation 
\begin{equation}\label{eq:deformation_approx6}
h_{\tau_j, t_i}^{\template, \velocityfield} \approx  h_{\tau_{j+1}, t_i}^{\template, \velocityfield}   \circ \Bigl(\Id + \frac{1}{MN}\velocityfield(\tau_j,\Cdot) \Bigr) 
\end{equation}
for $j = iM-1, iM-2, \ldots, 0$, where, by \cref{eq:h_t_ti_concrete}, 
\[h_{t_i, t_i}^{\template, \velocityfield} = 2\ForwardOpAdjoint_{t_i}\bigl(\ForwardOp_{t_i}(\diffeo_{0, t_i}^{\velocityfield}.\template) - \data(t_i,\Cdot) \bigr).\]

As observed from \cref{eq:soc4} and \cref{eq:gradientflow}, we need to discrete 
$\eta_{\tau_j, t_i}^{\velocityfield}$ for $j = iM-1, iM-2, \ldots, 0$ and $i = 1, \ldots, N$. By \cref{eq:eta_t_ti_concrete} we know the fact  
$\eta_{t_i, t_i}^{\velocityfield} = 0$ and  
\begin{equation}\label{eq:eta_tauj_ti_concrete}
\eta_{\tau_j, t_i}^{\velocityfield} = \int_{\tau_j}^{t_i} \vert\velocityfield(\iota, \cdot)\vert^2\bigl( \gelement{\tau_j,\iota}{\velocityfield}\bigr)\dint \iota. 
\end{equation}
We discretize the right-hand side of \cref{eq:eta_tauj_ti_concrete} by 
\begin{equation}\label{eq:eta_tauj_ti_concrete_discretize}
\eta_{\tau_j, t_i}^{\velocityfield} \approx \frac{1}{iM - j}\sum_{l=j+1}^{iM}\vert\velocityfield(\tau_l, \cdot)\vert^2\bigl( \gelement{\tau_j,\tau_l}{\velocityfield}\bigr). 
\end{equation}
\begin{remark}
The following scheme 
\begin{equation}\label{eq:eta_tauj_ti_concrete_discretize_alternative}
\eta_{\tau_j, t_i}^{\velocityfield} \approx \frac{1}{iM - j + 1}\sum_{l=j}^{iM}\vert\velocityfield(\tau_l, \cdot)\vert^2\bigl( \gelement{\tau_j,\tau_l}{\velocityfield}\bigr) 
\end{equation}
is an alternative for discretizing \cref{eq:eta_tauj_ti_concrete}. But it has no remarkable improvement to the ultimate result. 
\end{remark}
Similarly, using \cref{eq:deformation_approx4}, we implement the deformation $\gelement{\tau_j,\tau_l}{\velocityfield}$ 
in \cref{eq:eta_tauj_ti_concrete_discretize} by 
  \begin{equation}\label{eq:deformation_approx4_generalize}
\vert\velocityfield(\tau_l, \cdot)\vert^2 \circ \gelement{\tau_s, \tau_l}{\velocityfield} \approx  \vert\velocityfield(\tau_l, \cdot)\vert^2 \circ \gelement{\tau_{s+1}, \tau_l}{\velocityfield} \circ \Bigl(\Id + \frac{1}{MN}\velocityfield(\tau_s,\Cdot)\Bigr) 
\end{equation}
for $s = l-1, l-2, \ldots, j$.

\subsection{Algorithms} 
\label{sec:Algorithm}

As analyzed in \cref{sec:pratical_example}, we need to solve \cref{eq:VarReg_LDDMM_time_discretized} by the 
alternating iterative scheme in \cref{eq:Alternating}. At each iteration two subproblems will be solved, namely, 
updating $\template$ with given $\velocityfield$ (\cref{algo:GDSB_4DCT}) 
and updating $\velocityfield$ with given 
$\template$ (\cref{algo:GradientDescentAlgorithmForFiniteFunctional_0}).

\subsubsection{Algorithm for template reconstruction}
\label{sec:Gradient_descent_scheme1}

Using the discretization in \cref{sec:Discretization}, we give detailed implementation of the projected gradient 
descent scheme in \cref{algo:GDSB_4DCT} for minimizing the smoothed version 
of \cref{eq:VarReg_LDDMM_template_time_discrete2} as described in \cref{sec:Template_reconstruction}.

\begin{algorithm}[htbp]
\caption{Projected gradient descent \cref{eq:soc4} for template reconstruction}
\label{algo:GDSB_4DCT}
\begin{algorithmic}[1]
\State \emph{Initialize}: Let $t_i \gets \frac{i}{N}$ for $i = 0, \ldots, N$, $\tau_j \gets \frac{j}{MN}$ 
for $j = 0, \ldots, MN$. Given initial template $\template^0$, velocity field $\velocityfield$, regularization parameters $\mu_1 > 0$ and $\mu_2 > 0$, 
error tolerance $\epsilon_{\template} > 0$, stepsize $\alpha > 0$, and iteration number $K_{\template} > 0$. Let $k \gets 0$. 
\State \emph{Loop}:
\State \quad Compute $\gelement{0, \tau_j}{\velocityfield} . \template^k$ by \cref{eq:deformation_approx_template}: 
\[
\gelement{0, \tau_j}{\velocityfield} . \template^k  \gets  \Bigl(1 - \frac{1}{MN} \Div\velocityfield(\tau_j,\Cdot)\Bigr) 
 \bigl(\gelement{0, \tau_{j-1}}{\velocityfield} . \template^k\bigr)  \circ \Bigl(\Id - \frac{1}{MN}\velocityfield(\tau_j,\Cdot) \Bigr) 
\]
\quad  for $1 \leq j \leq MN$, where $\gelement{0, 0}{\velocityfield} . \template^k = \template^k$. 
\State \quad Estimate $h_{0, t_i}^{\template^k, \velocityfield}$ for $1 \leq i \leq N$ by \cref{eq:deformation_approx6}: 
\[
h_{\tau_j, t_i}^{\template^k, \velocityfield} \gets  h_{\tau_{j+1}, t_i}^{\template^k, \velocityfield}   \circ \Bigl(\Id + \frac{1}{MN}\velocityfield(\tau_j,\Cdot) \Bigr) 
 \]
\quad for $j = iM-1, iM-2, \ldots, 0$, and update $h_{t_i, t_i}^{\template^k, \velocityfield}$ by
\[
h_{t_i, t_i}^{\template^k, \velocityfield} \gets 2\ForwardOpAdjoint_{t_i}\bigl(\ForwardOp_{t_i}(\diffeo_{0, t_i}^{\velocityfield}.\template^k) - \data(t_i,\Cdot) \bigr). 
\] 
\State \quad Compute $\eta_{0, t_i}^{\velocityfield}$ for $1 \leq i \leq N$ by \cref{eq:eta_tauj_ti_concrete_discretize}: 
\[
\eta_{0, t_i}^{\velocityfield} \gets \frac{1}{iM}\sum_{l=1}^{iM}\vert\velocityfield(\tau_l, \cdot)\vert^2\bigl( \gelement{0,\tau_l}{\velocityfield}\bigr), 
\]
\quad where for $1 \leq l \leq iM$, by \cref{eq:deformation_approx4_generalize}, 
\[
\vert\velocityfield(\tau_l, \cdot)\vert^2 \circ \gelement{\tau_s, \tau_l}{\velocityfield} \gets  \vert\velocityfield(\tau_l, \cdot)\vert^2 \circ \gelement{\tau_{s+1}, \tau_l}{\velocityfield} \circ \Bigl(\Id + \frac{1}{MN}\velocityfield(\tau_s,\Cdot)\Bigr) 
 \]
\quad for $s = l-1, l-2, \ldots, 0$.
\State \quad Update $\grad\GoalFunctionalV_{\velocityfield}(\template^k)$ by 
 \[
 \grad\GoalFunctionalV_{\velocityfield}(\template^k)
    \gets  \frac{1}{N}\sum_{i=1}^{N}\Bigl(h_{0, t_i}^{\template^k, \velocityfield}  + \mu_2 \eta_{0, t_i}^{\velocityfield} \Bigr) + \mu_1\grad^{\,\ast}\biggl(\frac{\grad \template^k}{\vert \nabla \template^k \vert_{2, \epsilon}}\biggr). 
  \]
\State \quad Evaluate $\template^{k+1}$ by
\[
\template^{k+1} \gets \textsf{Proj}_{\geq 0}\Bigl\{\template^k - \alpha\grad\GoalFunctionalV_{\velocityfield}(\template^k)\Bigr\}. 
\]
\State \quad \textbf{If} $\bigr\vert \template^{k+1} - \template^k \bigr\vert > \epsilon_{\template}$ and $k<K_{\template}$, then $k \gets k+1$, \textbf{goto} \emph{Loop}.
\State \textbf{Output} $\template^{k+1}$.
\end{algorithmic}
\end{algorithm}

\subsubsection{Algorithm for velocity field estimation}
\label{sec:Gradient_descent_scheme2}

Here we list the numerical implementation of gradient descent scheme for 
velocity field estimation in \cref{sec:GradientDescent_velocithfield}.
The following \cref{algo:GradientDescentAlgorithmForFiniteFunctional_0} outlines the procedure for computing the 
scheme \cref{eq:gradientflow} that makes use of the discretization in \cref{sec:Discretization}.

\begin{algorithm}[htbp]
\caption{Gradient descent \cref{eq:gradientflow} for velocity field estimation}
\label{algo:GradientDescentAlgorithmForFiniteFunctional_0}
\begin{algorithmic}[1]
\State \emph{Initialize}: Let $t_i \gets \frac{i}{N}$ for $i = 0, \ldots, N$, $\tau_j \gets \frac{j}{MN}$ 
for $j = 0, \ldots, MN$. Given initial velocity field $\velocityfield^0(\tau_j)$, template $\template$, 
regularization parameter $\mu_2 > 0$, error tolerance $\epsilon_{\velocityfield} > 0$, 
stepsize $\beta > 0$, and iteration number $K_{\velocityfield} > 0$. 
Fixed kernel function $\kernel(\Cdot,\Cdot)$. Let $k \gets 0$.
\State \emph{Loop}:
\State  \quad Compute $\gelement{0, \tau_j}{\velocityfield^k} . \template$ by \cref{eq:deformation_approx_template}: 
\[
\gelement{0, \tau_j}{\velocityfield^k} . \template  \gets  \Bigl(1 - \frac{1}{MN} \Div\velocityfield^k(\tau_j,\Cdot)\Bigr) 
 \bigl(\gelement{0, \tau_{j-1}}{\velocityfield^k} . \template\bigr)  \circ \Bigl(\Id - \frac{1}{MN}\velocityfield^k(\tau_j,\Cdot) \Bigr) 
\]
\quad  for $1 \leq j \leq MN$, where $\gelement{0, 0}{\velocityfield^k} . \template = \template$. 
\State \quad Update $h_{\tau_j, t_i}^{\template, \velocityfield^k}$ for $1 \leq i \leq N$ by \cref{eq:deformation_approx6}: 
\[
h_{\tau_j, t_i}^{\template, \velocityfield^k}    
  \gets  h_{\tau_{j+1}, t_i}^{\template, \velocityfield^k} \circ \Bigl(\Id + \frac{1}{MN}\velocityfield^k(\tau_j,\Cdot) \Bigr) 
 \]
\quad for $j = iM-1, iM-2, \ldots, 0$ and compute $h_{t_i, t_i}^{\template, \velocityfield^k}$ for $1 \leq i \leq N$ by
\[
h_{t_i, t_i}^{\template, \velocityfield^k} \gets 2\ForwardOpAdjoint_{t_i}\bigl(\ForwardOp_{t_i}(\gelement{0, t_i}{\velocityfield^k} . \template) - \data(t_i,\Cdot)\bigr).
\]
\State \quad Compute $\eta_{\tau_j, t_i}^{\velocityfield^k}$ for $1 \leq i \leq N$ by \cref{eq:eta_tauj_ti_concrete_discretize}: 
\[
\eta_{\tau_j, t_i}^{\velocityfield^k} \gets \frac{1}{iM - j}\sum_{l=j + 1}^{iM}\vert\velocityfield^k(\tau_l, \cdot)\vert^2\bigl( \gelement{\tau_j,\tau_l}{\velocityfield^k}\bigr), 
\]
\quad where for $j < l \leq iM$, by \cref{eq:deformation_approx4_generalize}, 
\[
\vert\velocityfield^k(\tau_l, \cdot)\vert^2 \circ \gelement{\tau_s, \tau_l}{\velocityfield^k} \gets  \vert\velocityfield^k(\tau_l, \cdot)\vert^2 \circ \gelement{\tau_{s+1}, \tau_l}{\velocityfield^k} \circ \Bigl(\Id + \frac{1}{MN}\velocityfield^k(\tau_s,\Cdot)\Bigr) 
 \]
\quad for $s = l-1, l-2, \ldots, j$.
\State \quad Evaluate $\grad^{\,\Vspace}\!\!\GoalFunctionalV_{\template}(\velocityfield^k)(\tau_j,\Cdot)$ (using \ac{FFT} to compute the convolution) by 
 \begin{equation*}
 \grad^{\,\Vspace}\!\!\GoalFunctionalV_{\template}(\velocityfield^k)(\tau_j,\Cdot) 
    \gets  \frac{1}{N}\sum_{\{i\geq 1 : t_i \geq \tau_j\}} \Koperator\biggl( \gelement{0, \tau_j}{\velocityfield^k} . \template \Bigl[ \grad \bigl(h_{\tau_j, t_i}^{\template,\velocityfield^k}  + \mu_2 \eta_{\tau_j, t_i}^{\velocityfield^k}\bigr) 
    + 2\mu_2 \velocityfield^k_{\tau_j, t_i} \Bigr] \biggr) 
  \end{equation*}
\quad  for $0 \leq j \leq MN$. 
\State \quad Update $\velocityfield^k(\tau_j, \Cdot)$ for $0 \leq j \leq MN$ by:
\[
 \velocityfield^{k+1}(\tau_j,\Cdot)  \gets \velocityfield^k(\tau_j,\Cdot) -  \beta \grad^{\,\Vspace}\!\!\GoalFunctionalV_{\template}(\velocityfield^k)(\tau_j,\Cdot).
\]
\State \quad \textbf{If} $\bigr\vert \velocityfield^{k+1} - \velocityfield^k \bigr\vert > \epsilon_{\velocityfield}$ and $k<K_{\velocityfield}$, then $k \gets k+1$, \textbf{goto} \emph{Loop}.
\State \textbf{Output} $\velocityfield^{k+1}$.
\end{algorithmic}
\end{algorithm}

\subsubsection{Alternating minimization algorithm}
\label{sec:Alternating_minimization_method}

Ultimately, the alternating minimization algorithm for recovering the template and velocity field is presented in the 
following \cref{algo:Alternating_reconstruction}. The iteration number for solving each inner subproblem is restricted to be one. 
The complexity of the algorithm is comparable to the counterpart in \cite{ChGrOz19}.  

\begin{algorithm}[htbp]
\caption{Alternating minimization algorithm}
\label{algo:Alternating_reconstruction}
\begin{algorithmic}[1]
\State \emph{Initialize}: Given $M, N$. Let $t_i \gets \frac{i}{N}$ for $i = 0, \ldots, N$, $\tau_j \gets \frac{j}{MN}$ 
for $j = 0, \ldots, MN$. Given initial velocity field $\velocityfield^0$ and template $\template^0$, 
regularization parameters $\mu_1 > 0$ and $\mu_2 > 0$, error tolerances $\epsilon_{\template} > 0$ and $\epsilon_{\velocityfield} > 0$, 
fixed stepsizes $\alpha^k = \alpha > 0$ and $\beta^k = \beta > 0$, and maximum iteration number $K > 0$. 
Fixed kernel function $\kernel(\Cdot,\Cdot)$. Let $k \gets 0$. 
\State \emph{Loop}:
\State  \quad Set $\velocityfield = \velocityfield^k$. Run steps 3-7 of \cref{algo:GDSB_4DCT}. Update $\template^{k+1}$.
\State  \quad Set $\template = \template^{k+1}$. Run steps 3-7 of \cref{algo:GradientDescentAlgorithmForFiniteFunctional_0}. Update $\velocityfield^{k+1}$.
\\
\quad \textbf{Or} 
\State  \quad Set $\template = \template^k$. Run Lines 3-7 of \cref{algo:GradientDescentAlgorithmForFiniteFunctional_0}. Update $\velocityfield^{k+1}$.
\State  \quad Set $\velocityfield = \velocityfield^{k+1}$. Run Lines 3-7 of \cref{algo:GDSB_4DCT}. Update $\template^{k+1}$. 
\\
\State \quad \textbf{If} $\bigr\vert \velocityfield^{k+1} - \velocityfield^k \bigr\vert > \epsilon_{\velocityfield}$ or $\bigr\vert \template^{k+1} - \template^k \bigr\vert > \epsilon_{\template}$, and $k<K$,  \\
\quad then $k \gets k+1$, \textbf{goto} \emph{Loop}.
\State \textbf{Output} $\template^{k+1}$, $\velocityfield^{k+1}$.
\end{algorithmic}
\end{algorithm}

\section{Numerical experiments}
\label{sec:numerical_experiments}

To evaluate the proposed method, we adopt the very sparse and/or highly noisy data sets simulated in 2D spatial and 
temporal tomography, which are measured from the mass-preserving sequential images by parallel beam scanning.  
The implemented algorithms were programmed in Python. The routines were operated on ThinkStation Xeon E5-2620 v4 2.10 GHz CPU, 64GB ROM, 
TITAN Xp GPU. The GPU was merely used to accelerate the forward 
and backward projections. The test section illustrates the performance of the proposed method even 
though this is not a full evaluation. The implementation 
is partially supported by Operator Discretization Library (\href{http://github.com/odlgroup/odl}{http://github.com/odlgroup/odl}).

The forward operator $\ForwardOp_t \colon \RecSpace \to \DataSpace$ is specified by 
Radon transform in $\Real^2$, namely,
\begin{equation*}
R(f)(\theta, x) = \int_{\Real} f(x + s\theta)\dint s  \quad \text{for $\theta \in S^1$ and $x \in \theta^{\bot}$},
\end{equation*}
where $R$ represents the Radon transform, $S^1$ is the unit circle, and $(\theta, x)$ 
determines a line through $x$ in $\Real^2$ with direction $\theta$.  
Additionally, the $\Vspace$ denotes the space of vector fields that is specified by an \ac{RKHS} with the following Gaussian kernel function  
\begin{equation}\label{eq:KernelEq}
  \kernel(x,y) := 
      \exp\Bigl(-\dfrac{1}{2 \sigma^2} \Vert x-y \Vert_2^2 \Bigr)
      \begin{pmatrix} 
          1  & 0 \\
          0  & 1
      \end{pmatrix} 
\quad\text{for $x,y \in \Real^2$,}
\end{equation}
where the $\sigma >0$ determines the kernel width. 

The mass-preserving images of all gates are defined on $\domain$. For the image at each gate, 
the noise-free data for per view is obtained by the 2D parallel beam projection, 
which is then added the Gaussian white noise at a certain level resulting in the noisy data. The noise level is quantified 
by \ac{SNR} in logarithmic decibel (dB).

\subsection{Test suites and results}\label{sec:results}

The test suites are dedicated to assessing the performance against the overview evaluation, different noise levels, and the 
sensitivity against various selections of regularization parameters $\mu_1$, $\mu_2$, and kernel width $\sigma$. 
We also compare the proposed method numerically to the methods by using \ac{TV}-based 
reconstruction, and $\LpSpace^2$-gradient descent scheme.

\subsubsection{Test suite 1: Overview evaluation}
\label{sec:test_suite_1}

Here we prepare a test for evaluating the overview performance with regard to numerical convergence, reconstructed accuracy, 
and mass-preserving property. This test uses a multi-object phantom with five gates (i.e., $N = 5$). 
The masses of the sequential images are the same. The ground truth at each gate is shown 
in the last row of \cref{Test_suite_1:multi_object_phantom}, 
which is adapted from \cite{ChOz18}. The image at each gate is consisting of six separately star-like 
objects with grey-values ranging on $[0, 1]$, which is digitized by using $438\!\times\!438$ pixels. 
The images of all gates are supported on a fixed rectangular domain $[-16, 16]\!\times\![-16, 16]$. 

To show the performance of the proposed method, we use the noise-free measurements. 
For the image at each gate, the noise-free data per view is measured by 2D parallel beam scanning geometry 
with evenly 620 bins, which is supported on the range of $[-24, 24]$.  
For the gate $i\,(1 \le i \le N)$, the scanning views are distributed on $[(i-1)\pi /36, \pi + (i-1) \pi/36]$ uniformly, 
and the total view number is only six. 

It is well-known that when the gradient of the image is sparse, tomographic reconstruction by \ac{TV}-based 
reconstruction method outperforms other methods, such as \ac{FBP}, the iterative methods without 
considering priori knowledge. This is especially notable when the data is undersampled. 
In this test, the used phantom has sparse gradient, and the sampling is quite sparse (six views per gate). 
However, assume that we neglect the dynamic motions among the gates (i.e., disregard any temporal evolution), 
and just treat the spatiotemporal problem as a static one. Then the whole tomographic data set is equivalently 
sampled from 30 projection views. We conduct image reconstruction with the \ac{TV}-based method. 
The reconstructed image carries severe motion artifacts as illustrated in \cref{Test_suite_1:TV_static_recon}, 
the distributed intensity of which is also disordered for the mass-preservation among these sequential images. 
In contrast, the proposed method exclusively focuses on dealing with such spatiotemporal imaging problem, 
and is applied to reconstruct the dynamically sequential images with mass-preservation. 
\begin{figure}[htbp]
\centering
    \includegraphics[width=0.8\textwidth]{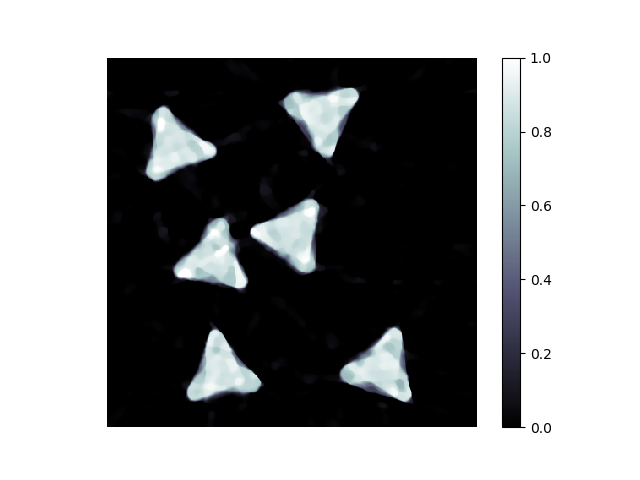}
    \vspace{-5mm}
   \caption{Test suite 1. Reconstructed image by the \ac{TV}-based method if the spatiotemporal problem is treated as a static one.}
\label{Test_suite_1:TV_static_recon}
\end{figure}

In the proposed model, the regularization parameters ($\mu_1, \mu_2$) are selected as $(0.01, 10^{-7})$ for 
the noise-free measurements. The factor  $M$ of discretized time degree is set to be $2$. 
The kernel width $\sigma$ is selected to be $2$. 
The gradient descent stepsizes are fixed as $\alpha = 0.01$ and $\beta = 0.05$, respectively. 
Firstly we apply \cref{algo:GDSB_4DCT} to obtain an initial template image after 50 
iterations by using all of the gated data with given zero velocity field. 
This is actually using static \ac{TV}-based method to perform 50 iterations. 
Then we use \cref{algo:Alternating_reconstruction} to solve the proposed model 
by the obtained initial template and zero initialized velocity field. 
Note that the above iteration number is flexible, which just needs enough to gain an appropriately 
initial template for \cref{algo:Alternating_reconstruction}. 

To validate the numerical convergence of the proposed algorithm, we set the maximum iteration number 
to be sufficiently large, for instance, $2000$. The descent curve of the objective functional is plotted 
in \cref{Test_suite_1:objective_func_descent}, which shows the stable convergence. Additionally, 
the reconstructed results are shown in the third row of \cref{Test_suite_1:multi_object_phantom}. 
It is clear that the reconstructed image at each gate are almost the same as the corresponding ground 
truth from visual observation.  
\begin{figure}[htbp]
\centering
    \includegraphics[width=\textwidth]{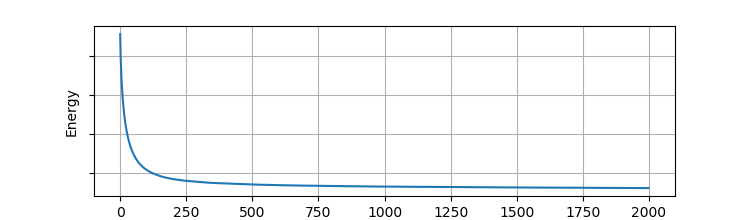}
    \vskip-0.25\baselineskip
   \caption{Test suite 1. Descent curve of the objective functional of the proposed model as the iteration grew. }
\label{Test_suite_1:objective_func_descent}
\end{figure}

We also compare the proposed method against some other approaches. In this test, the used phantom (ground truth image) 
has sparse gradient, and the sampling is sparse (six views per gate), so it is fairly comparing against \ac{TV}-based 
reconstruction method. We use the \ac{TV}-based method to perform reconstruction for the same projection data at each gate. 
The regularization parameter and the stepsize are chosen as $\alpha = 0.01$ and $\lambda = 0.01$ respectively, 
which are the same as the proposed method. After sufficiently the same 2000 iterations for each gate, the reconstructed results 
are shown in the first row of \cref{Test_suite_1:multi_object_phantom}. It is observed that these reconstructed images 
have severely stair-like blocks even though the shape structures of the objects are similar to those counterparts in the ground truth. 

Moreover, to obtain the diffeomorphic deformations, the velocity field at each time point is restricted into the \ac{RKHS} $\Vspace$ in 
the proposed model. Hence we use the $\Xspace{2}$-gradient descent scheme to solve the optimal velocity field. To validate 
its necessity, we substitute it by the more general $\LpSpace^2$ space, and then apply the $\LpSpace^2$-gradient descent 
scheme to solve the same problem in \cref{sec:GradientDescent_velocithfield}. To make fair comparison, 
we set the two regularization parameters as the same as the proposed method. 
Because the iterated velocity field lacks smoothness, the algorithm is not convergent 
under the same stepsizes as before. So the associated stepsizes are shortened as $\alpha = 0.001$ and $\beta = 0.005$. 
As a result, a convergent result is obtained as shown in the second row in \cref{Test_suite_1:multi_object_phantom} after 
sufficiently the 2000 iterations. As we can see, the reconstructed images carry severe artifacts on the objects other 
than similar shapes as the ground truth. Furthermore, we show their computed optimal velocity fields at the 
end points in \cref{Test_suite_1:display_velocity_field}. Clearly, the computed optimal velocity field 
by the $\LpSpace^2$-gradient descent scheme is nonsmooth, but that by $\Xspace{2}$-gradient descent 
scheme is quit smooth as expected. 
\begin{figure}[htbp]
\centering
\begin{minipage}[t]{0.2\textwidth}%
     \centering
     \includegraphics[trim=75 25 60 40, clip, width=\textwidth]{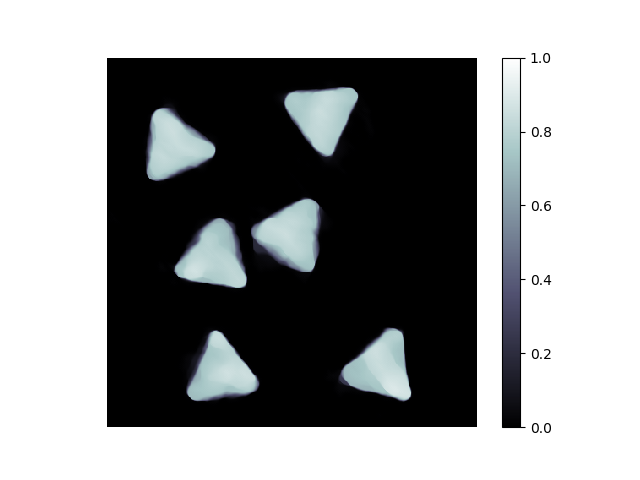}
     \vskip-0.25\baselineskip
   \end{minipage}%
   \hfill
   \begin{minipage}[t]{0.2\textwidth}%
     \centering
    \includegraphics[trim=75 25 60 40, clip, width=\textwidth]{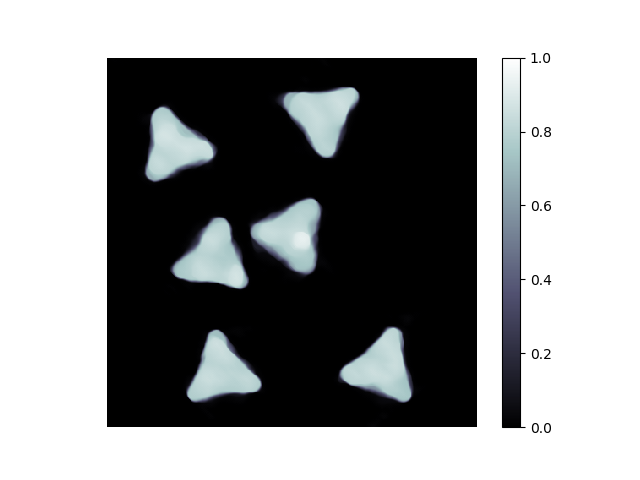}
     \vskip-0.25\baselineskip
   \end{minipage}%
   \hfill
   \begin{minipage}[t]{0.2\textwidth}%
     \centering
     \includegraphics[trim=75 25 60 40, clip, width=\textwidth]{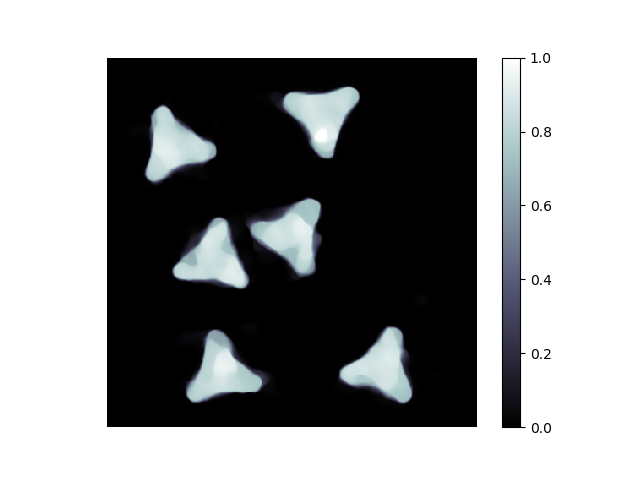}
     \vskip-0.25\baselineskip
   \end{minipage}%
      \hfill
   \begin{minipage}[t]{0.2\textwidth}%
     \centering
     \includegraphics[trim=75 25 60 40, clip, width=\textwidth]{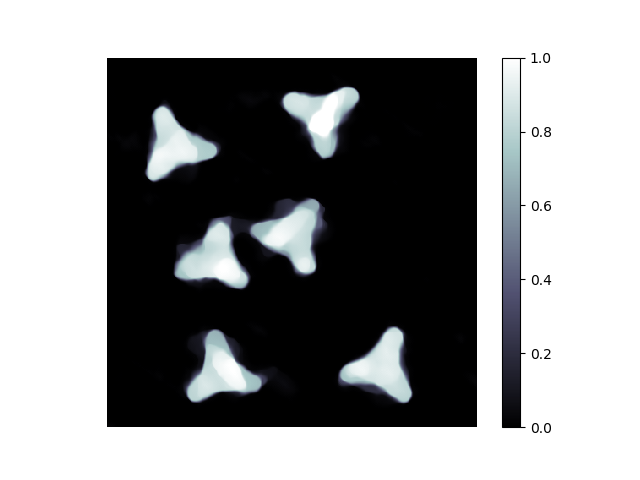}     
     \vskip-0.25\baselineskip
   \end{minipage}%
   \hfill
   \begin{minipage}[t]{0.2\textwidth}%
     \centering
    \includegraphics[trim=75 25 60 40, clip, width=\textwidth]{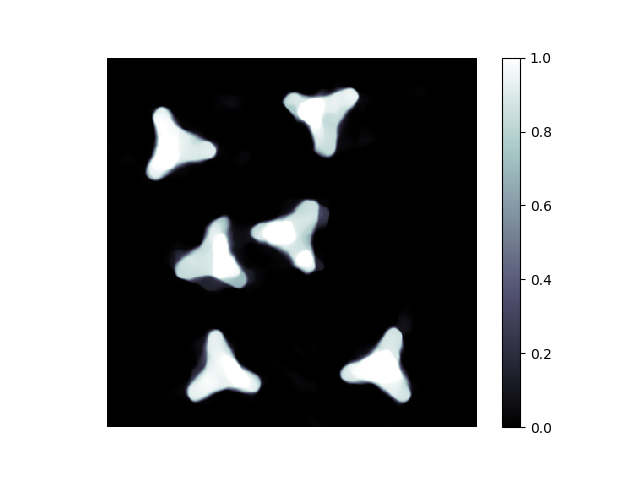}
     \vskip-0.25\baselineskip
   \end{minipage}%
\par\medskip      
\begin{minipage}[t]{0.2\textwidth}%
     \centering
     \includegraphics[trim=75 25 60 40, clip, width=\textwidth]{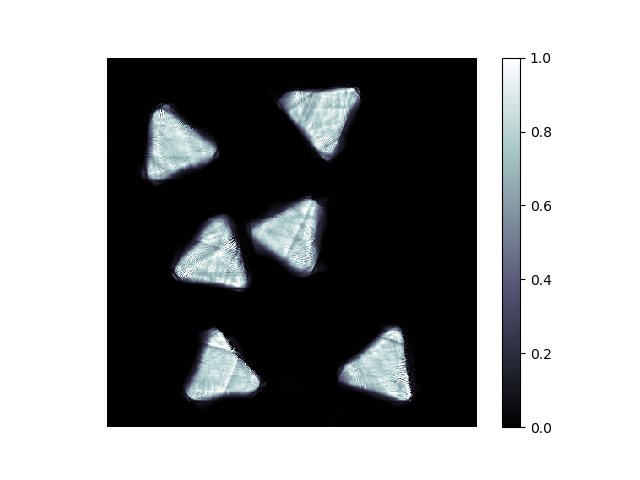}
     \vskip-0.25\baselineskip
   \end{minipage}%
   \hfill
   \begin{minipage}[t]{0.2\textwidth}%
     \centering
    \includegraphics[trim=75 25 60 40, clip, width=\textwidth]{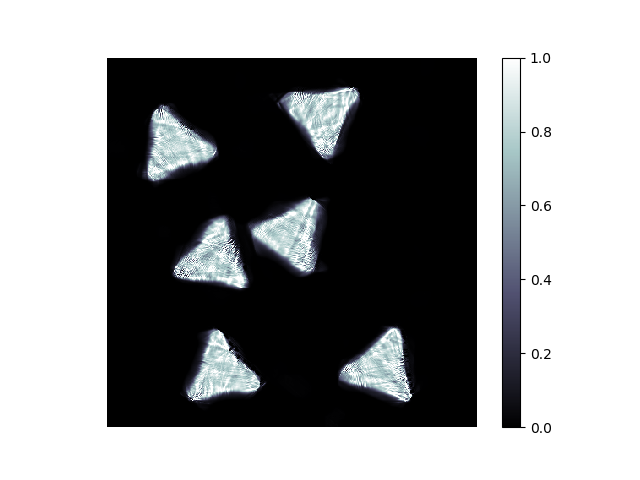}
     \vskip-0.25\baselineskip
   \end{minipage}%
   \hfill
   \begin{minipage}[t]{0.2\textwidth}%
     \centering
     \includegraphics[trim=75 25 60 40, clip, width=\textwidth]{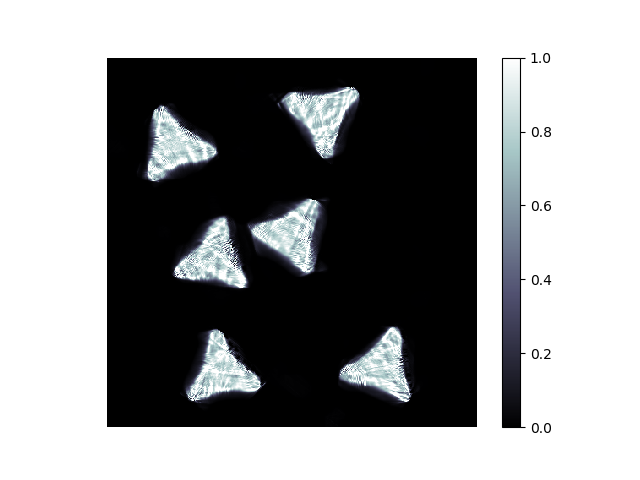}
     \vskip-0.25\baselineskip
   \end{minipage}%
      \hfill
   \begin{minipage}[t]{0.2\textwidth}%
     \centering
     \includegraphics[trim=75 25 60 40, clip, width=\textwidth]{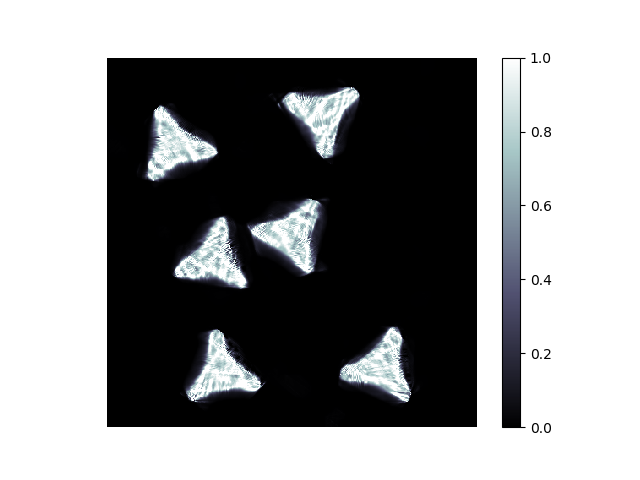}     
     \vskip-0.25\baselineskip
   \end{minipage}%
   \hfill
   \begin{minipage}[t]{0.2\textwidth}%
     \centering
    \includegraphics[trim=75 25 60 40, clip, width=\textwidth]{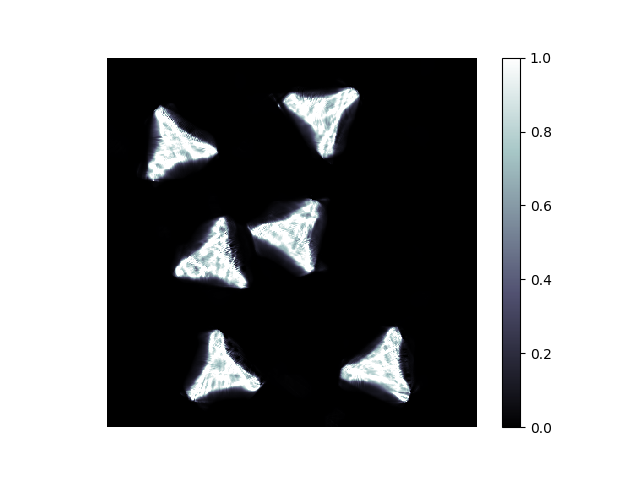}
     \vskip-0.25\baselineskip
   \end{minipage}%
\par\medskip      
\begin{minipage}[t]{0.2\textwidth}%
     \centering
     \includegraphics[trim=75 25 60 40, clip, width=\textwidth]{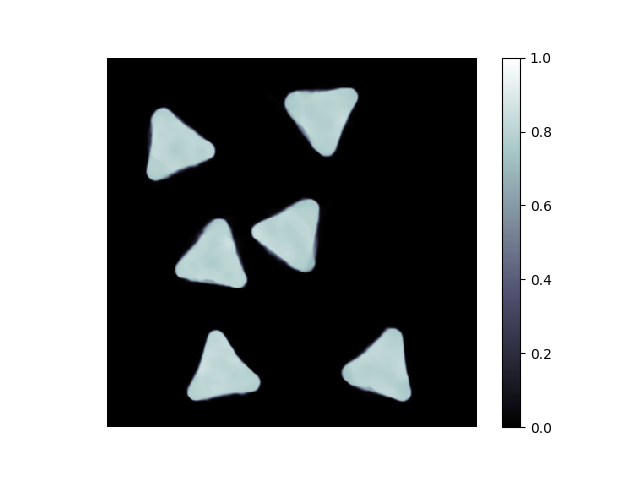}
     \vskip-0.25\baselineskip
   \end{minipage}%
   \hfill
   \begin{minipage}[t]{0.2\textwidth}%
     \centering
    \includegraphics[trim=75 25 60 40, clip, width=\textwidth]{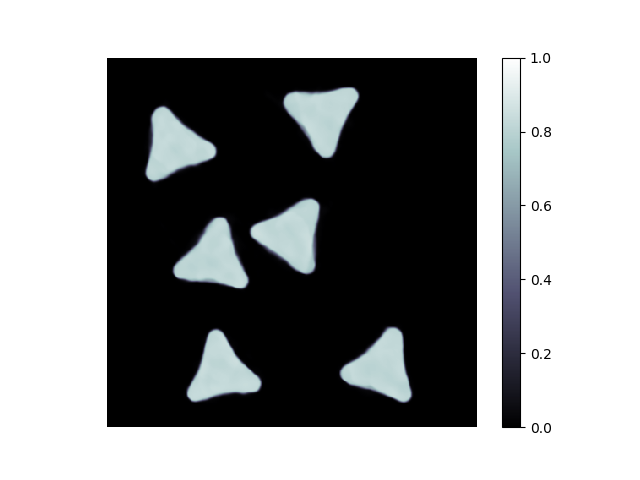}
     \vskip-0.25\baselineskip
   \end{minipage}%
   \hfill
   \begin{minipage}[t]{0.2\textwidth}%
     \centering
     \includegraphics[trim=75 25 60 40, clip, width=\textwidth]{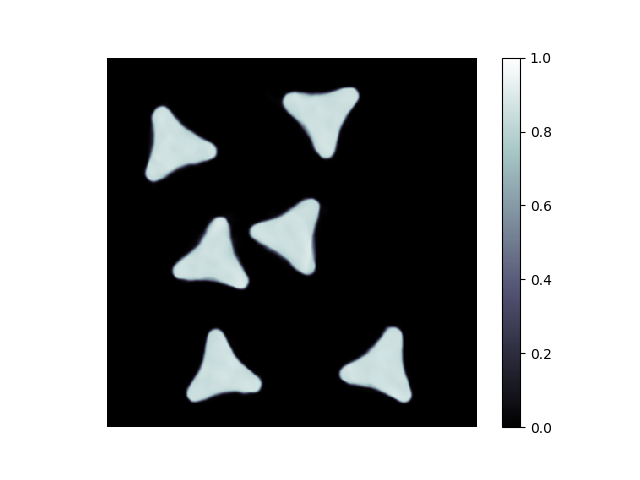}
     \vskip-0.25\baselineskip
   \end{minipage}%
      \hfill
   \begin{minipage}[t]{0.2\textwidth}%
     \centering
     \includegraphics[trim=75 25 60 40, clip, width=\textwidth]{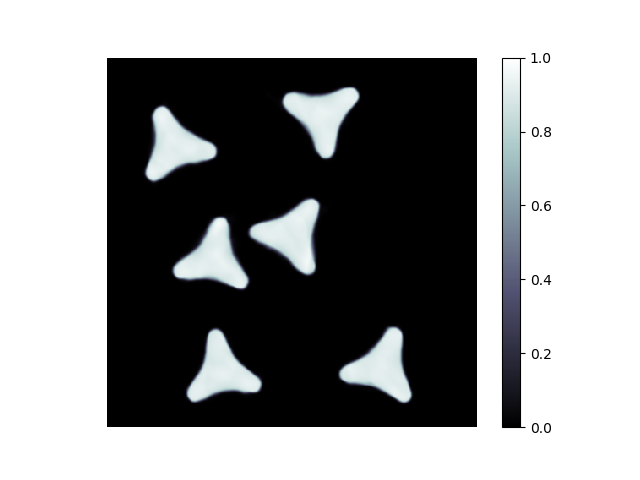}     
     \vskip-0.25\baselineskip
   \end{minipage}%
   \hfill
   \begin{minipage}[t]{0.2\textwidth}%
     \centering
    \includegraphics[trim=75 25 60 40, clip, width=\textwidth]{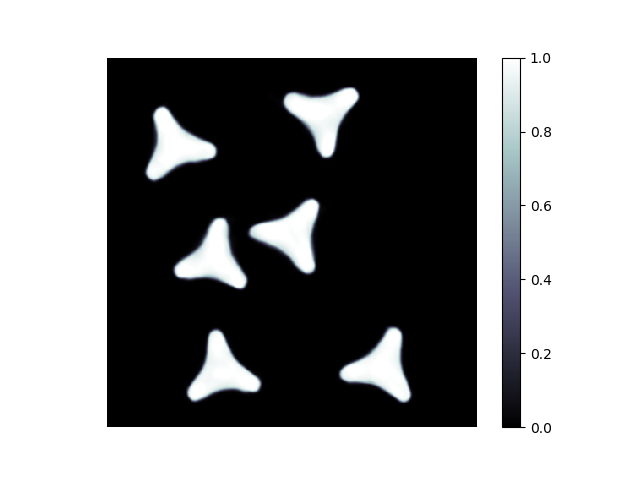}
     \vskip-0.25\baselineskip
   \end{minipage}%
\par\medskip      
   \begin{minipage}[t]{0.2\textwidth}%
     \centering
     \includegraphics[trim=75 25 60 40, clip, width=\textwidth]{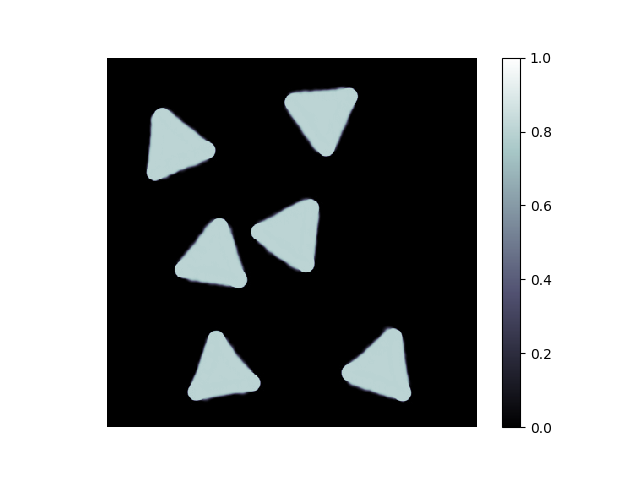}
     \vskip-0.25\baselineskip
     Gate 1
   \end{minipage}%
\hfill
   \begin{minipage}[t]{0.2\textwidth}%
     \centering
     \includegraphics[trim=75 25 60 40, clip, width=\textwidth]{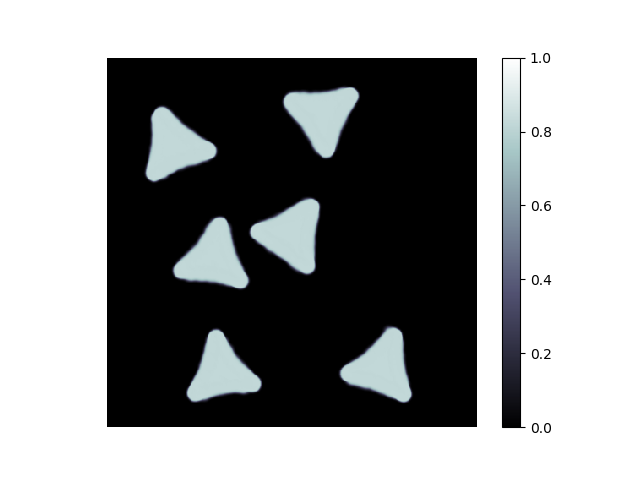}
     \vskip-0.25\baselineskip
     Gate 2
   \end{minipage}%
   \hfill
   \begin{minipage}[t]{0.2\textwidth}%
     \centering
    \includegraphics[trim=75 25 60 40, clip, width=\textwidth]{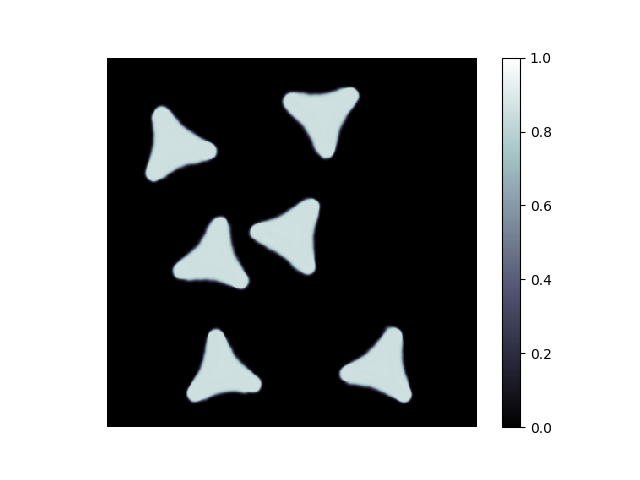}
     \vskip-0.25\baselineskip
          Gate 3
   \end{minipage}%
   \hfill
   \begin{minipage}[t]{0.2\textwidth}%
     \centering
     \includegraphics[trim=75 25 60 40, clip, width=\textwidth]{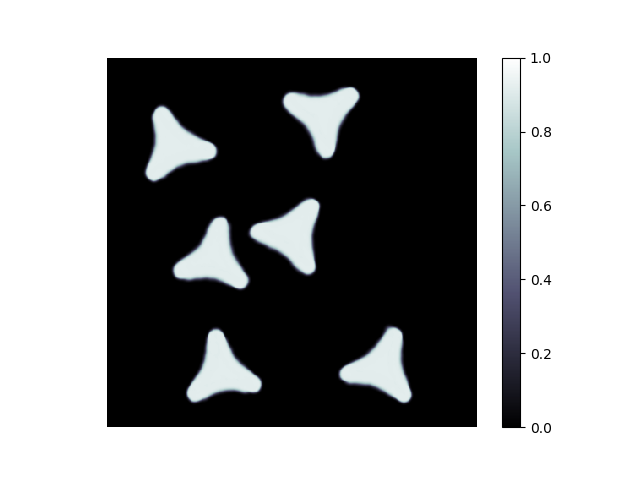}
     \vskip-0.25\baselineskip
     Gate 4
   \end{minipage}%
   \hfill
   \begin{minipage}[t]{0.2\textwidth}%
     \centering
     \includegraphics[trim=75 25 60 40, clip, width=\textwidth]{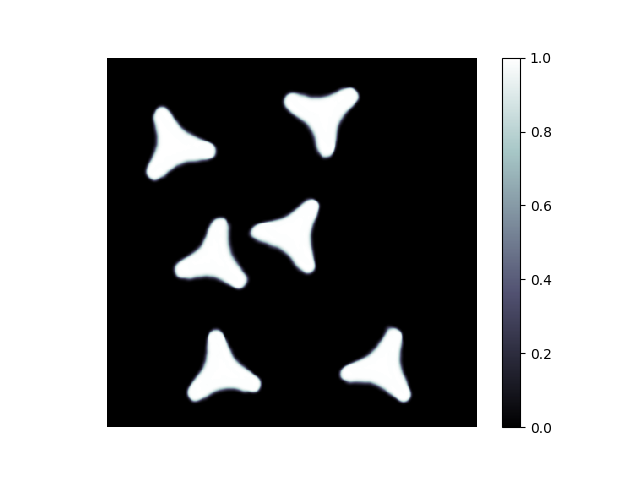}
     \vskip-0.25\baselineskip
     Gate 5
   \end{minipage}%
\caption{Test suite 1. Reconstructed images of the multi-object phantom. The columns represent the different gates. For the noise-free data, the first three rows are the reconstructed spatiotemporal images by \ac{TV}-based reconstruction method (row 1), the algorithm using $\LpSpace^2$-gradient descent method (row 2), and the proposed method (row 3). The last row (row 4) shows the ground truth for each gate.}
\label{Test_suite_1:multi_object_phantom}
\end{figure}

\begin{figure}[htbp]
\centering
\begin{minipage}[t]{0.5\textwidth}%
     \centering
     \includegraphics[width=\textwidth]{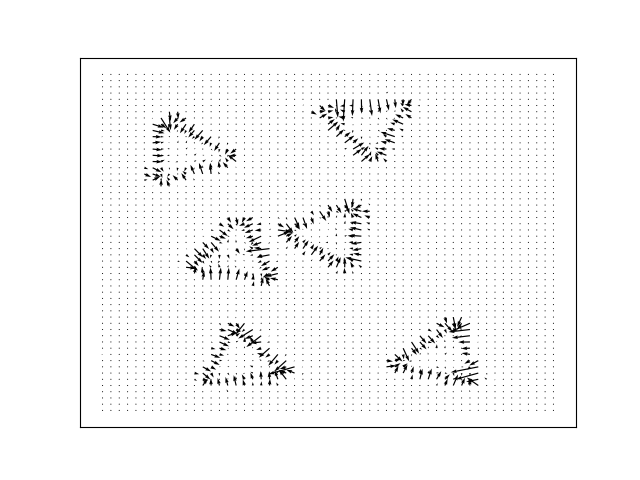}
     \vskip-0.25\baselineskip
   \end{minipage}%
   \hfill
\begin{minipage}[t]{0.5\textwidth}%
     \centering
    \includegraphics[width=\textwidth]{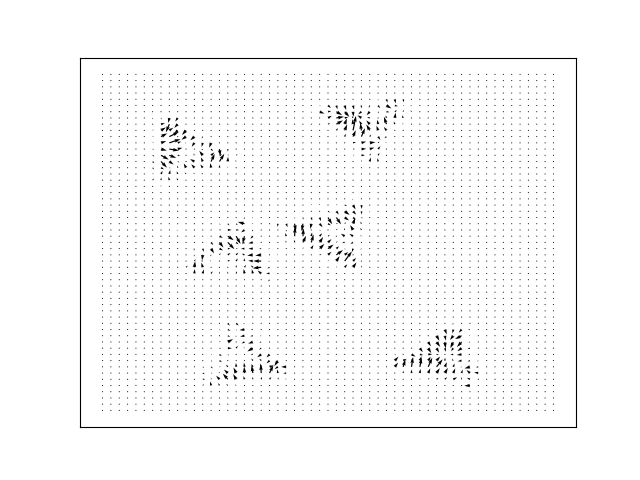}
     \vskip-0.25\baselineskip
   \end{minipage}%
 \\       
 \begin{minipage}[t]{0.5\textwidth}%
     \centering
     \includegraphics[width=\textwidth]{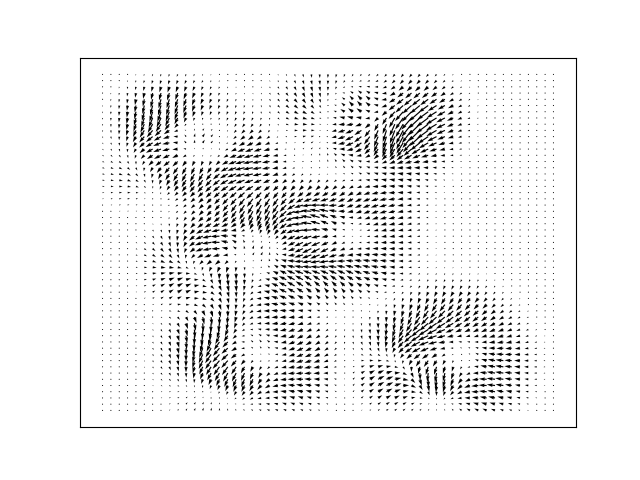}
     \vskip-0.25\baselineskip
     $t = 0$
   \end{minipage}%
   \hfill
\begin{minipage}[t]{0.5\textwidth}%
     \centering
    \includegraphics[width=\textwidth]{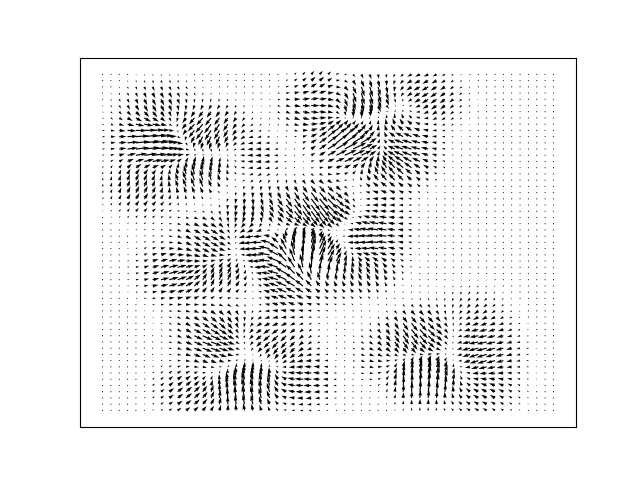}
     \vskip-0.25\baselineskip
     $t = 1$
   \end{minipage}%
\caption{Test suite 1. The computed optimal velocity field at the end time points $t=0$ (left) and $t=1$ (right) 
by the $\LpSpace^2$-gradient descent scheme (top) and the proposed method (bottom) in \cref{Test_suite_1:multi_object_phantom}, respectively. }
\label{Test_suite_1:display_velocity_field}
\end{figure}

Apart from the visual contrast, the reconstruction is quantitatively compared 
using \ac{SSIM}, \ac{PSNR} and \ac{NRMSE}, which is frequently used to evaluate the 
image quality of reconstruction \cite{WaBoShSi04}. The larger of the first two indexes implies the better image quality. 
But the larger of the last index means the worse image quality. For the reconstruction results of 
different methods with the same noise-free data, the values of \ac{SSIM}, \ac{PSNR} and \ac{NRMSE} of 
the reconstructed spatiotemporal images compared to the related ground truths 
are tabulated in \cref{Test_suite_1:noise_free_quantitative}. 
\begin{table}[htbp]
\centering
\begin{tabular}{c | r r r r r}
&  \multicolumn{1}{c}{Gate 1} &  \multicolumn{1}{c}{Gate 2} & \multicolumn{1}{c}{Gate 3} 
 &  \multicolumn{1}{c}{Gate 4}  &  \multicolumn{1}{c}{Gate 5} \\ 
\hline                                   
 \multirow{3}{*}{\ac{TV}}  &0.9571     &0.9609   &0.9416   &0.9279  &0.9350  \\
     								 &26.70       &28.15      &26.58    &25.31    &27.05 \\
     								 &0.1283      &0.1102    &0.1355  &0.1628  &0.1397  \\
\hline                                   
 \multirow{3}{*}{$\LpSpace^2$ gradient}  &0.8749       &0.8686   &  0.8650  &0.8644  &0.8677   \\
                                          & 22.36         &20.73    &  20.37    &20.65  &21.28  \\
                                          & 0.2115       &0.2591   & 0.2770  &0.2784  &0.2716\\
\hline
 \multirow{3}{*}{Proposed}    &0.9819   &0.9879   &0.9893   &0.9892   &0.9874   \\  
     								   &31.60    &36.20      &38.21      &38.10      &35.83   \\
     								   &0.0729  &0.0437   &0.0355   &0.0373    &0.0509  \\
     \hline
\end{tabular}
\caption{Test suite 1. The values of \ac{SSIM}, \ac{PSNR} and \ac{NRMSE} of the reconstructed spatiotemporal images compared to the related ground truths for the noise-free measurements, see \cref{Test_suite_1:multi_object_phantom} for reconstructed images. Each table entry has three values that the upper is the value of \ac{SSIM}, the middle is the value of \ac{PSNR}, and the bottom is the value of \ac{NRMSE}, which corresponds to the image at the counterpart position of rows 1--3 in \cref{Test_suite_1:multi_object_phantom}.}
\label{Test_suite_1:noise_free_quantitative}
\end{table}

As compared these values with each other, the values of \ac{SSIM} by the proposed method is bigger than 
those by \ac{TV}-based method and using $\LpSpace^2$-gradient descent method. Additionally, the values 
of \ac{PSNR} by the proposed method are much bigger than the those by the other two methods. 
And the values of \ac{NRMSE} by the proposed method are much smaller than those by the other two methods. 
The statements are also consistent with the visual observation in \cref{Test_suite_1:multi_object_phantom}. 

Hence both visual and quantitative comparisons demonstrate that the reconstructed images by the proposed method is much  
more approximated to the corresponding ground truths. In other words, the proposed method largely improved 
the quality of the reconstructed images. 

Finally, the masses of the reconstructed images are hopefully to be preserved. In order to inspect this characteristic, 
we obtain all of the masses of the images in \cref{Test_suite_1:multi_object_phantom}.  
As listed in \cref{Test_suite_1:mass_preserved_property}, the values of the mass of ground truths are all $111.75$ for 
that we uses the originally sequential images with the same masses. It is clear in \cref{Test_suite_1:mass_preserved_property} that 
the mass of the reconstructed images is preserved very well during the numerical implementation of 
the proposed method, which is better than the $\LpSpace^2$-gradient descent method. 
Moreover, we found that the mass of the result at each gate by the proposed method is almost the 
same as \ac{TV}-based method. Since the \ac{TV}-based reconstruction method is implemented gate by gate, 
the mass of the result at each gate should be the same essentially. Even though the masses 
have a little bit errors compared with the ground truths, that is reasonable because these images are reconstructed only from six-angle 
projection data. Hence, the proposed method has desirable performance on the mass-preserving property. 

\begin{table}[htbp]
\centering
\begin{tabular}{c | r r r r r}
&  \multicolumn{1}{c}{Gate 1} &  \multicolumn{1}{c}{Gate 2} & \multicolumn{1}{c}{Gate 3} 
 &  \multicolumn{1}{c}{Gate 4}  &  \multicolumn{1}{c}{Gate 5} \\ 
\hline                                   
 \multirow{1}{*}{\ac{TV}}  &112.20     &112.21   &112.18   &112.19  &112.21  \\
                                  
 \multirow{1}{*}{$\LpSpace^2$ gradient}  &113.83    &112.83   &111.58  &110.92  &110.54   \\
 
 \multirow{1}{*}{Proposed}    &112.02   &112.15   &112.22   &112.27   &112.24   \\  
 
 \multirow{1}{*}{Ground truth}  &111.75    &111.75  &111.75  &111.75  &111.75  \\
     \hline
\end{tabular}
\caption{Test suite 1. The values of masses of the reconstructed images, which corresponds to the image at the counterpart position in \cref{Test_suite_1:multi_object_phantom}.}
\label{Test_suite_1:mass_preserved_property}
\end{table}

\subsubsection{Test suite 2: Robustness against different noise levels}
\label{sec:Robustness}

The images are reconstructed by using the noise-free data in test suite 1. 
To test the robustness against different noise levels of the proposed method, here the same 
multi-object phantom is used, and three different levels of additive Gaussian 
white noise are added onto the above noise-free data. The resulting \ac{SNR} are about $14.6$dB, $7.69$dB, and $5.53$dB, respectively. 
To show the noise levels more clear, we profile the noise-free and noisy projection data of the first view of at 
Gate 1 in \cref{Test_suite_2:sinogram}. To some extent, such three data sets can be seen as the cases of low, 
moderate and high noise levels correspondingly. 
\begin{figure}[htbp]
\centering
\begin{minipage}[t]{0.33\textwidth}%
     \centering
     \includegraphics[trim=30 15 30 40, clip, width=\textwidth]{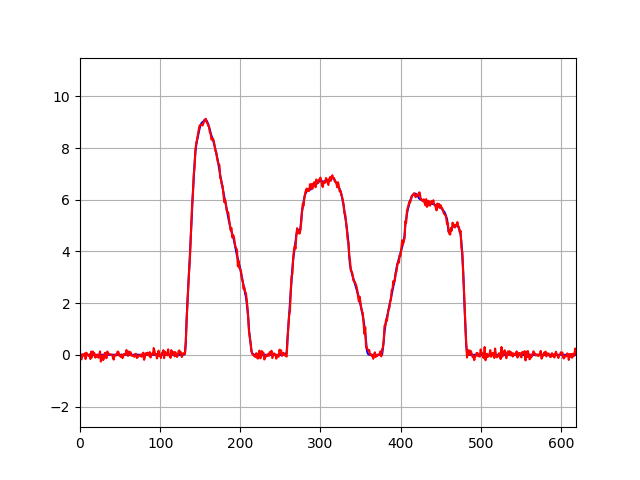}
     \vskip-0.25\baselineskip
   \end{minipage}%
   \hfill
   \begin{minipage}[t]{0.33\textwidth}%
     \centering
    \includegraphics[trim=30 15 30 40, clip, width=\textwidth]{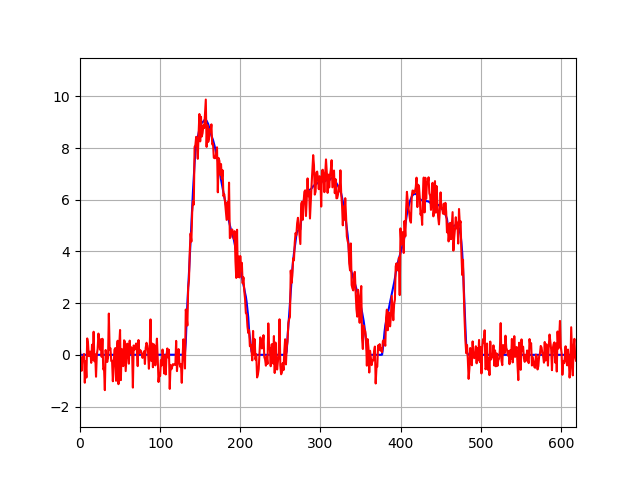}
     \vskip-0.25\baselineskip
   \end{minipage}%
   \hfill
   \begin{minipage}[t]{0.33\textwidth}%
     \centering
    \includegraphics[trim=30 15 30 40, clip, width=\textwidth]{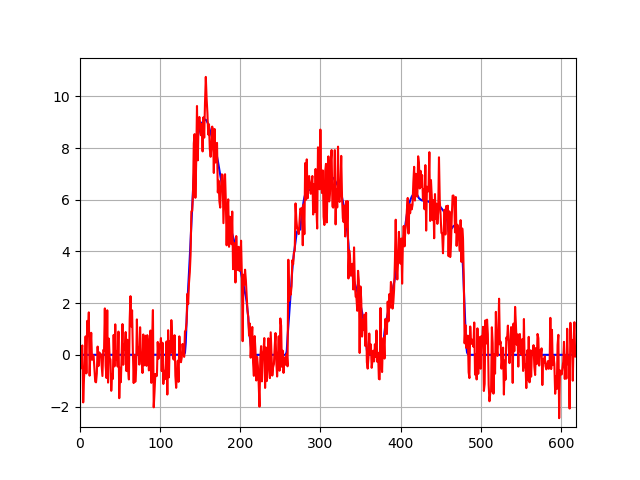}
     \vskip-0.25\baselineskip
   \end{minipage}%
   \caption{Test suite 2. Data of the first projection view at Gate 1. The left, middle, and right figures show the data of the first view with $14.6$dB, $7.69$dB, and $5.53$dB noise levels, respectively. The blue curve denotes the noise-free data, and the red jagged curve shows the noisy data.}
\label{Test_suite_2:sinogram}
\end{figure}

During numerical implementations, the regularization parameters ($\mu_1, \mu_2$) are 
selected as $(0.05, 10^{-7})$, $(0.1, 10^{-7})$ and $(0.15, 10^{-7})$ for the data with $14.6$dB, $7.69$dB 
and $5.53$dB noise levels, respectively. The lower \ac{SNR}, the lager value of $\mu_1$ for the spatio regularization term. 
The maximum iteration number is set to be $2000$ for sufficiently numerical convergence. The associated stepsizes 
are set as $\alpha = 0.001$ and $\beta = 0.005$. As before, the initial template is obtained by \cref{algo:GDSB_4DCT} for 50 
iterations using all of the data with given zero velocity field. Then we use \cref{algo:Alternating_reconstruction} to solve the 
proposed model with the obtained initial template and zero initialized velocity field. 
The reconstructed results are shown in \cref{Test_suite_2:different_noise_levels}. 
It is demonstrated that the reconstructed images by the proposed method (rows~1, 3 and 5) are close to the 
corresponding ground truth in \cref{Test_suite_1:multi_object_phantom}, even though the noise level of the data is higher and higher.  
Additionally, the image at each single gate is also reconstructed by the \ac{TV}-based method for numerical comparison, 
as shown in rows~2, 4 and 6 of \cref{Test_suite_2:different_noise_levels} for each the same noise level data. Obviously, 
the reconstructed results by the proposed method is much better than the \ac{TV}-based method. 
\begin{figure}[htbp]
\centering
\begin{minipage}[t]{0.2\textwidth}%
     \centering
     \includegraphics[trim=75 25 60 40, clip, width=\textwidth]{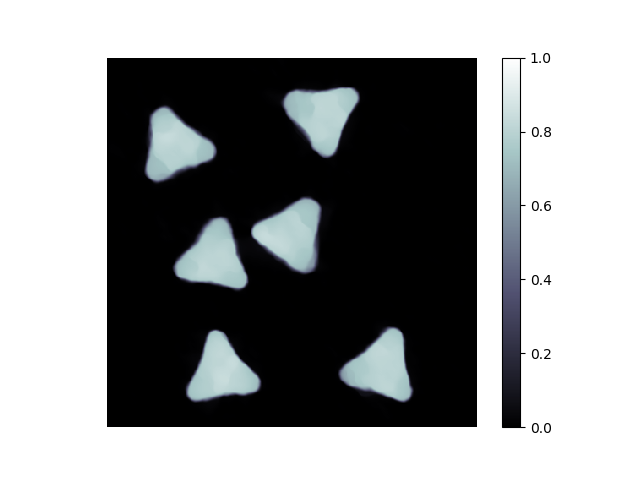}
     \vskip-0.25\baselineskip
   \end{minipage}%
   \hfill
   \begin{minipage}[t]{0.2\textwidth}%
     \centering
    \includegraphics[trim=75 25 60 40, clip, width=\textwidth]{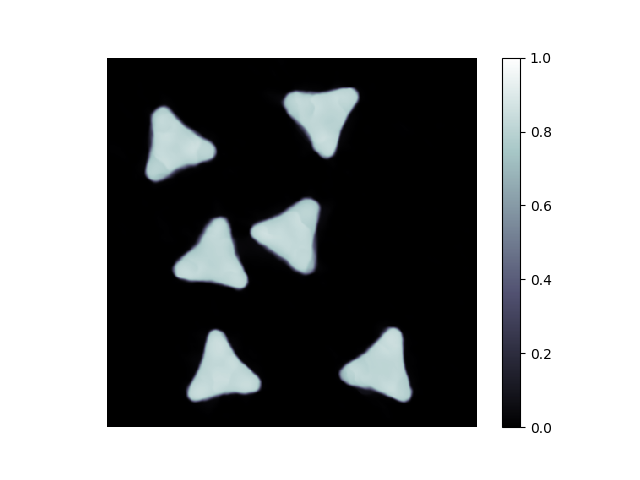}
     \vskip-0.25\baselineskip
   \end{minipage}%
   \hfill
   \begin{minipage}[t]{0.2\textwidth}%
     \centering
     \includegraphics[trim=75 25 60 40, clip, width=\textwidth]{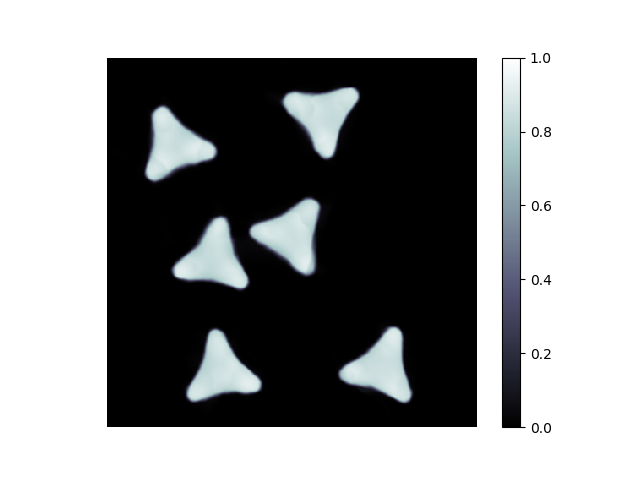}
     \vskip-0.25\baselineskip
   \end{minipage}%
      \hfill
   \begin{minipage}[t]{0.2\textwidth}%
     \centering
     \includegraphics[trim=75 25 60 40, clip, width=\textwidth]{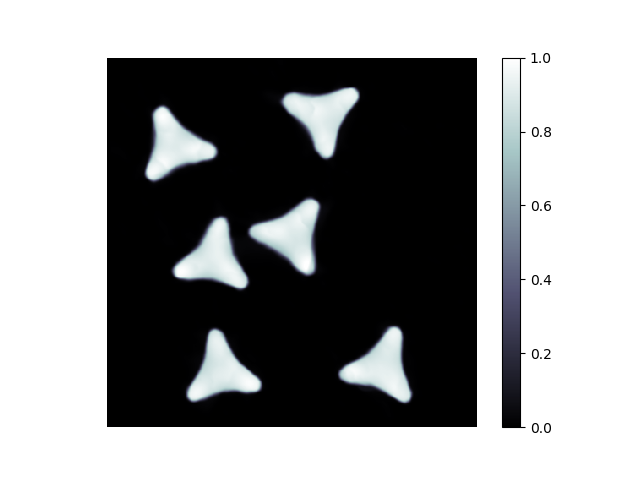}     
     \vskip-0.25\baselineskip
   \end{minipage}%
   \hfill
   \begin{minipage}[t]{0.2\textwidth}%
     \centering
    \includegraphics[trim=75 25 60 40, clip, width=\textwidth]{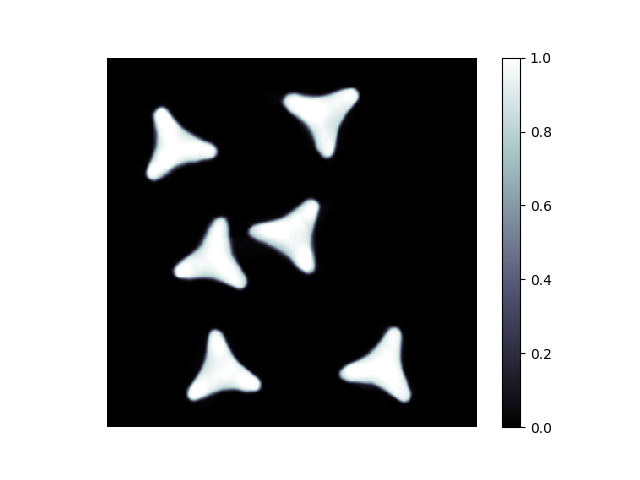}
     \vskip-0.25\baselineskip
   \end{minipage}%
   \hfill
   \begin{minipage}[t]{0.2\textwidth}%
     \centering
     \includegraphics[trim=75 25 60 40, clip, width=\textwidth]{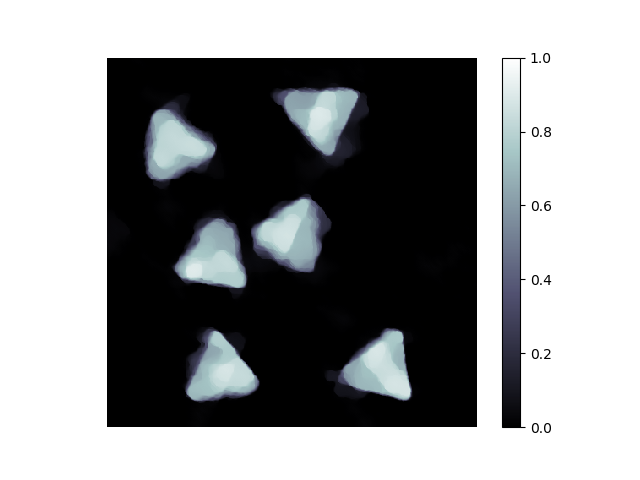}
     \vskip-0.25\baselineskip
   \end{minipage}%
   \hfill
   \begin{minipage}[t]{0.2\textwidth}%
     \centering
    \includegraphics[trim=75 25 60 40, clip, width=\textwidth]{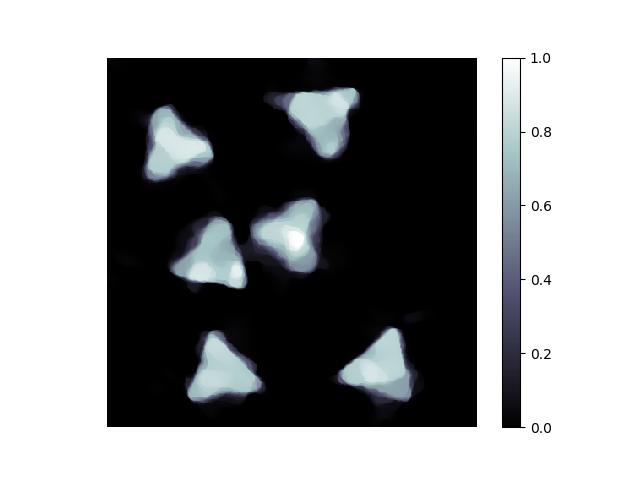}
     \vskip-0.25\baselineskip
   \end{minipage}%
   \hfill
   \begin{minipage}[t]{0.2\textwidth}%
     \centering
    \includegraphics[trim=75 25 60 40, clip, width=\textwidth]{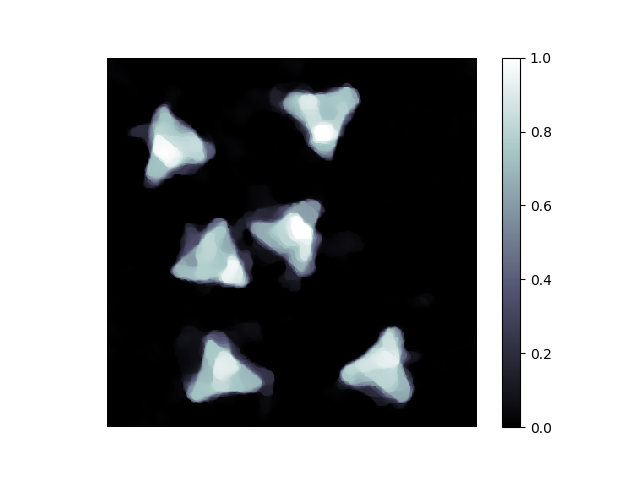}
     \vskip-0.25\baselineskip
   \end{minipage}%
   \hfill
   \begin{minipage}[t]{0.2\textwidth}%
     \centering
    \includegraphics[trim=75 25 60 40, clip, width=\textwidth]{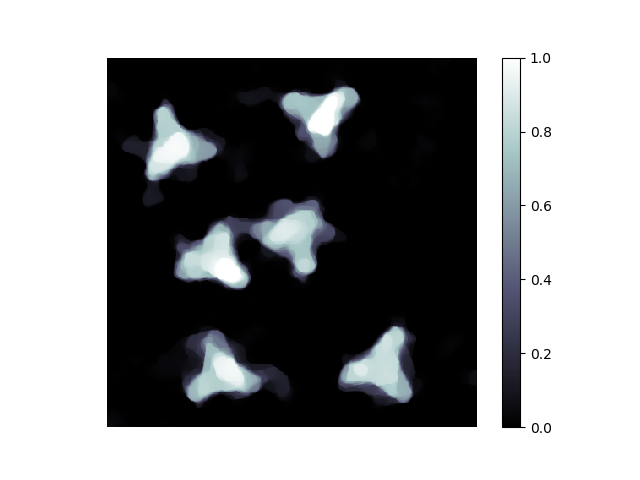}
     \vskip-0.25\baselineskip
   \end{minipage}%
   \hfill
   \begin{minipage}[t]{0.2\textwidth}%
     \centering
    \includegraphics[trim=75 25 60 40, clip, width=\textwidth]{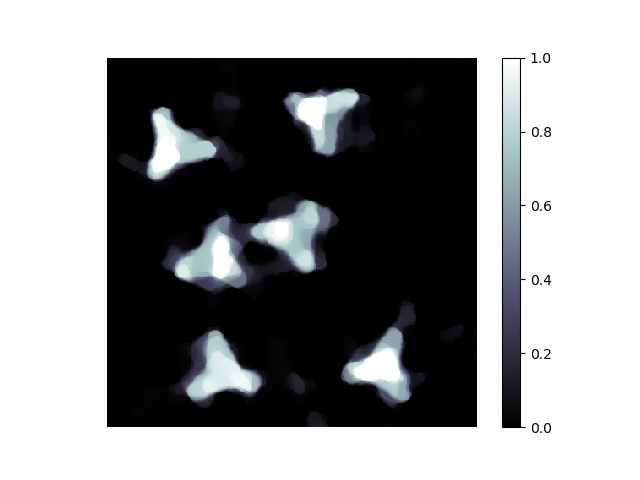}
     \vskip-0.25\baselineskip
   \end{minipage}%
   \hfill
\par\medskip      
\begin{minipage}[t]{0.2\textwidth}%
     \centering
     \includegraphics[trim=75 25 60 40, clip, width=\textwidth]{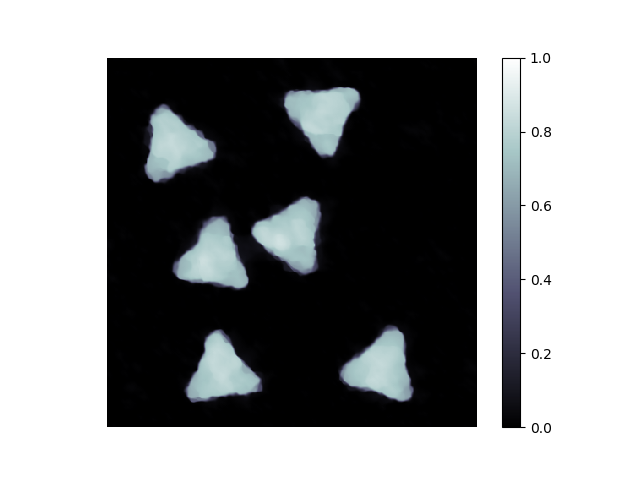}
     \vskip-0.25\baselineskip
   \end{minipage}%
   \hfill
   \begin{minipage}[t]{0.2\textwidth}%
     \centering
    \includegraphics[trim=75 25 60 40, clip, width=\textwidth]{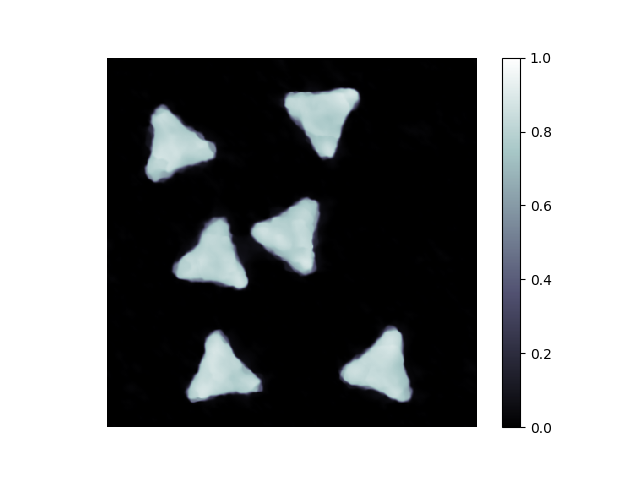}
     \vskip-0.25\baselineskip
   \end{minipage}%
   \hfill
   \begin{minipage}[t]{0.2\textwidth}%
     \centering
     \includegraphics[trim=75 25 60 40, clip, width=\textwidth]{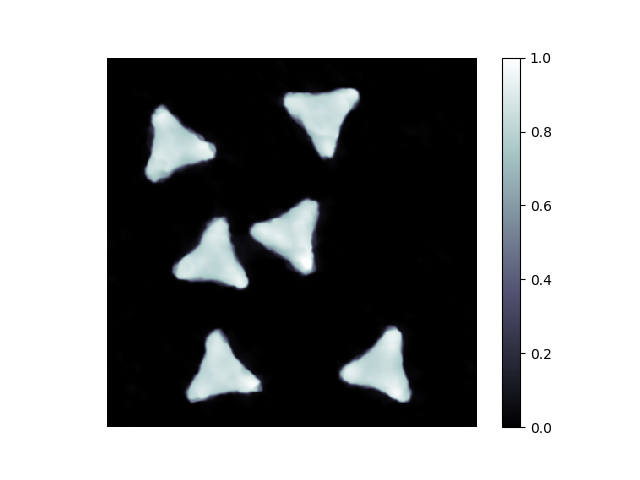}
     \vskip-0.25\baselineskip
   \end{minipage}%
      \hfill
   \begin{minipage}[t]{0.2\textwidth}%
     \centering
     \includegraphics[trim=75 25 60 40, clip, width=\textwidth]{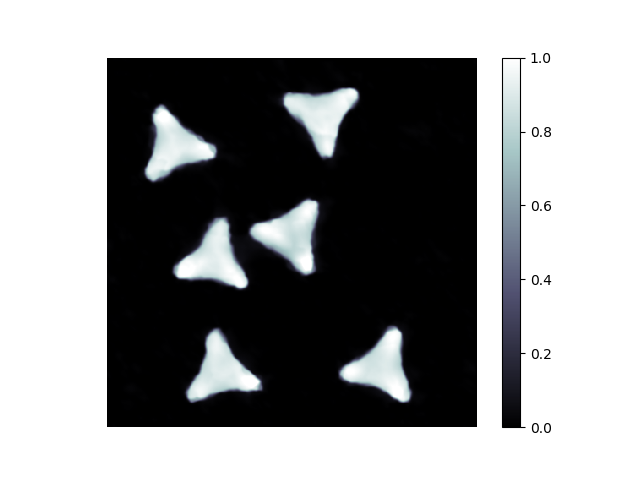}     
     \vskip-0.25\baselineskip
   \end{minipage}%
   \hfill
   \begin{minipage}[t]{0.2\textwidth}%
     \centering
    \includegraphics[trim=75 25 60 40, clip, width=\textwidth]{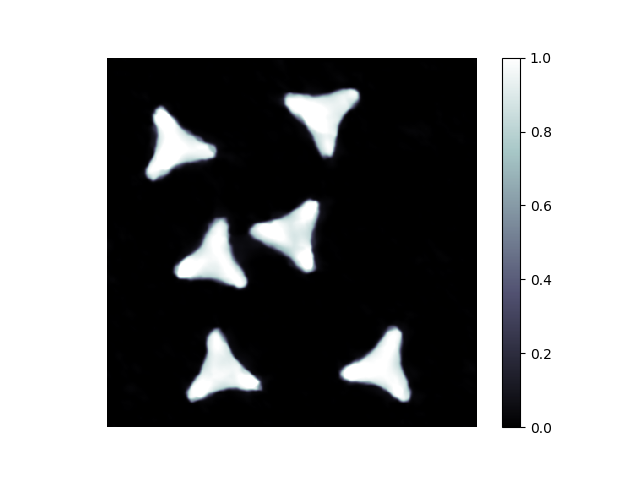}
     \vskip-0.25\baselineskip
   \end{minipage}%
         \hfill
   \begin{minipage}[t]{0.2\textwidth}%
     \centering
     \includegraphics[trim=75 25 60 40, clip, width=\textwidth]{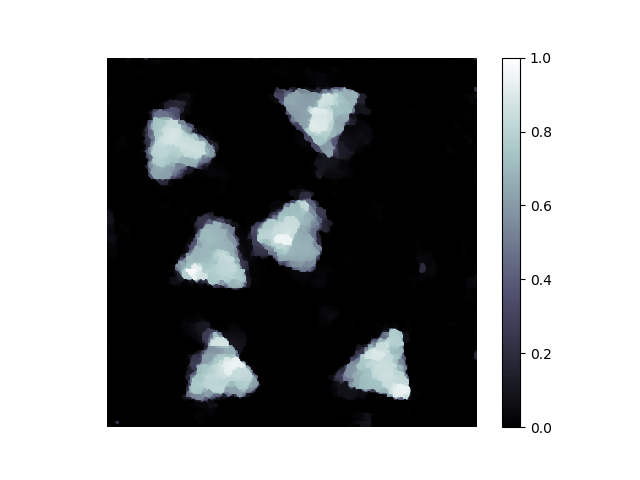}
     \vskip-0.25\baselineskip
   \end{minipage}%
   \hfill
   \begin{minipage}[t]{0.2\textwidth}%
     \centering
    \includegraphics[trim=75 25 60 40, clip, width=\textwidth]{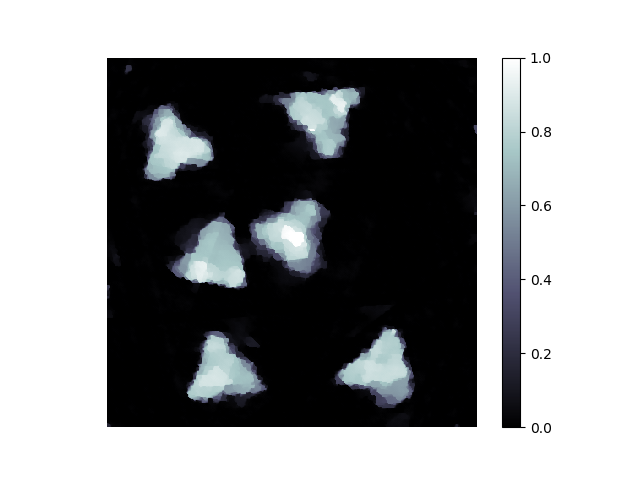}
     \vskip-0.25\baselineskip
   \end{minipage}%
   \hfill
   \begin{minipage}[t]{0.2\textwidth}%
     \centering
    \includegraphics[trim=75 25 60 40, clip, width=\textwidth]{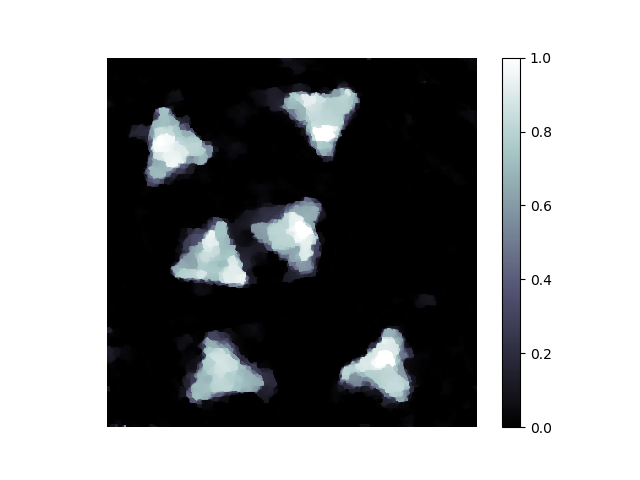}
     \vskip-0.25\baselineskip
   \end{minipage}%
   \hfill
   \begin{minipage}[t]{0.2\textwidth}%
     \centering
    \includegraphics[trim=75 25 60 40, clip, width=\textwidth]{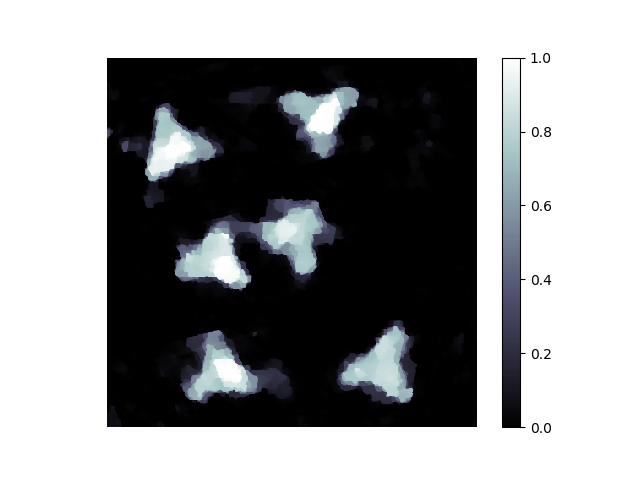}
     \vskip-0.25\baselineskip
   \end{minipage}%
   \hfill
   \begin{minipage}[t]{0.2\textwidth}%
     \centering
    \includegraphics[trim=75 25 60 40, clip, width=\textwidth]{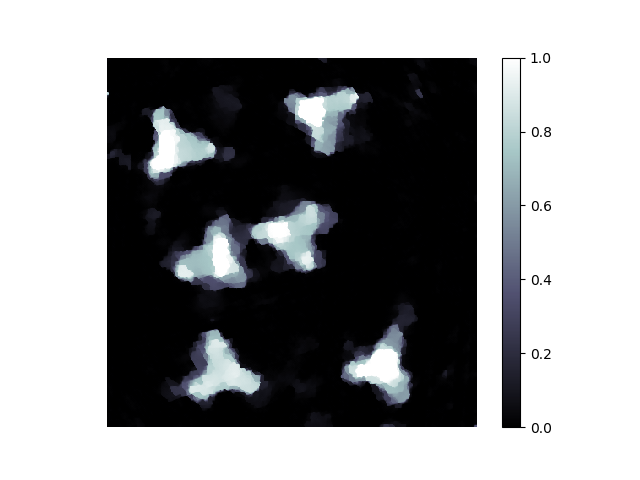}
     \vskip-0.25\baselineskip
   \end{minipage}%
   \hfill
\par\medskip      
\begin{minipage}[t]{0.2\textwidth}%
     \centering
     \includegraphics[trim=75 25 60 40, clip, width=\textwidth]{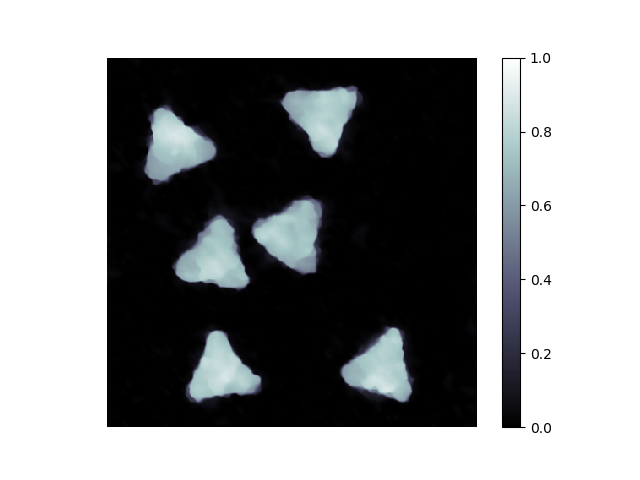}
     \vskip-0.25\baselineskip
   \end{minipage}%
   \hfill
   \begin{minipage}[t]{0.2\textwidth}%
     \centering
    \includegraphics[trim=75 25 60 40, clip, width=\textwidth]{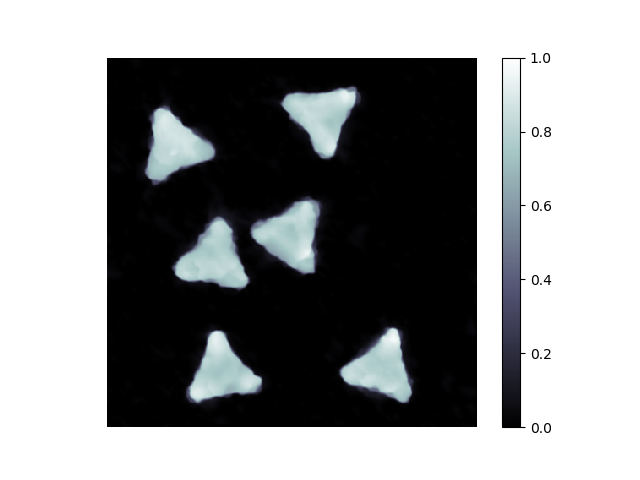}
     \vskip-0.25\baselineskip
   \end{minipage}%
   \hfill
   \begin{minipage}[t]{0.2\textwidth}%
     \centering
     \includegraphics[trim=75 25 60 40, clip, width=\textwidth]{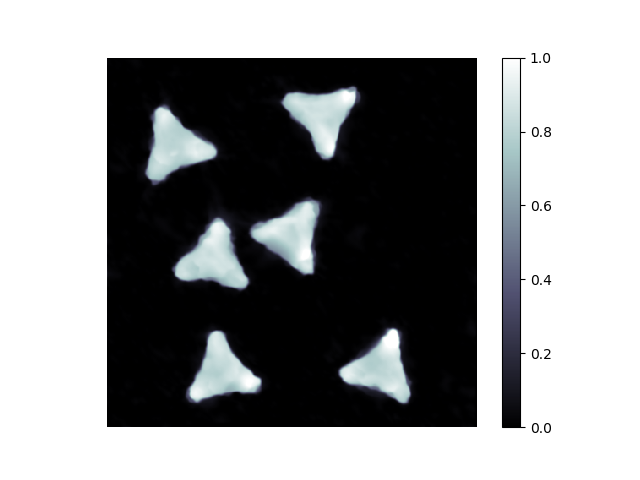}
     \vskip-0.25\baselineskip
   \end{minipage}%
      \hfill
   \begin{minipage}[t]{0.2\textwidth}%
     \centering
     \includegraphics[trim=75 25 60 40, clip, width=\textwidth]{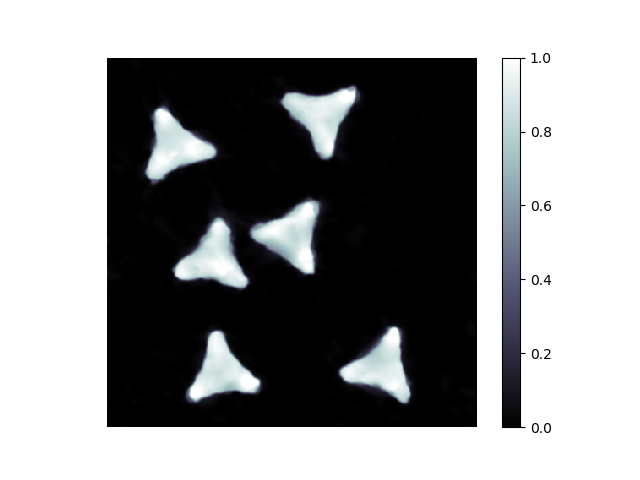}     
     \vskip-0.25\baselineskip
   \end{minipage}%
   \hfill
   \begin{minipage}[t]{0.2\textwidth}%
     \centering
    \includegraphics[trim=75 25 60 40, clip, width=\textwidth]{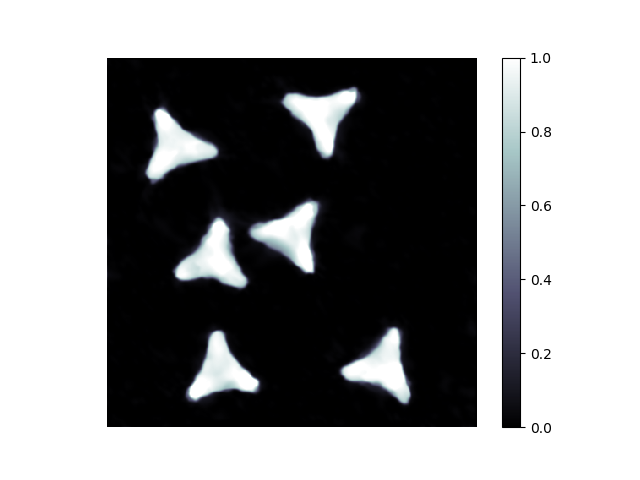}
     \vskip-0.25\baselineskip
   \end{minipage}%
\hfill     
   \begin{minipage}[t]{0.2\textwidth}%
     \centering
     \includegraphics[trim=75 25 60 40, clip, width=\textwidth]{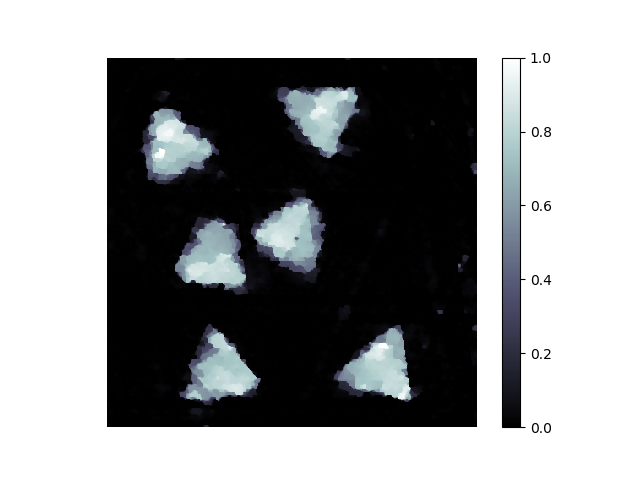}
     \vskip-0.25\baselineskip
     Gate 1
   \end{minipage}%
\hfill
   \begin{minipage}[t]{0.2\textwidth}%
     \centering
     \includegraphics[trim=75 25 60 40, clip, width=\textwidth]{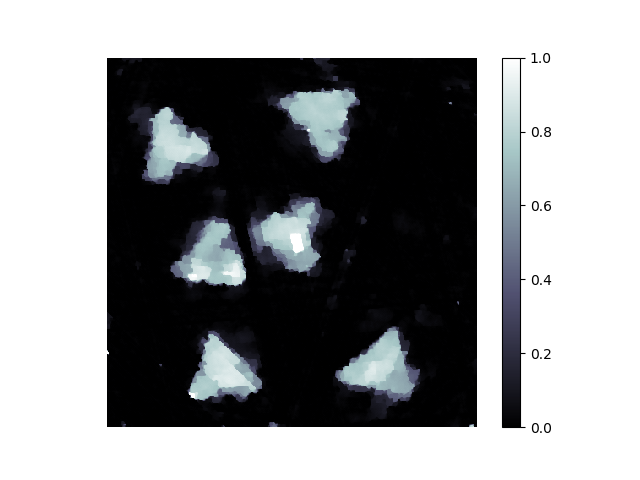}
     \vskip-0.25\baselineskip
     Gate 2
   \end{minipage}%
   \hfill
   \begin{minipage}[t]{0.2\textwidth}%
     \centering
    \includegraphics[trim=75 25 60 40, clip, width=\textwidth]{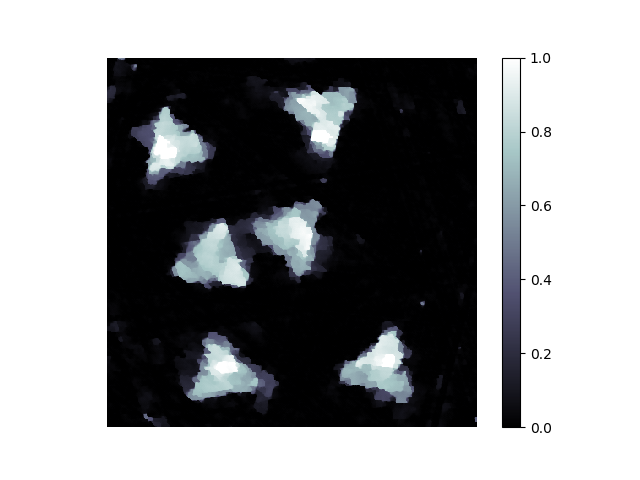}
     \vskip-0.25\baselineskip
          Gate 3
   \end{minipage}%
   \hfill
   \begin{minipage}[t]{0.2\textwidth}%
     \centering
     \includegraphics[trim=75 25 60 40, clip, width=\textwidth]{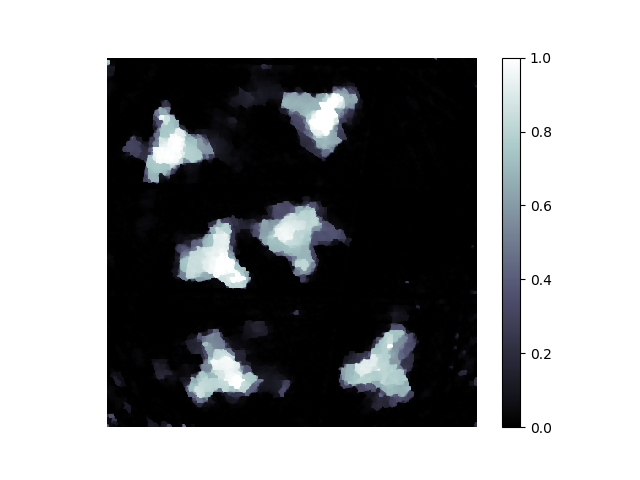}
     \vskip-0.25\baselineskip
     Gate 4
   \end{minipage}%
   \hfill
   \begin{minipage}[t]{0.2\textwidth}%
     \centering
     \includegraphics[trim=75 25 60 40, clip, width=\textwidth]{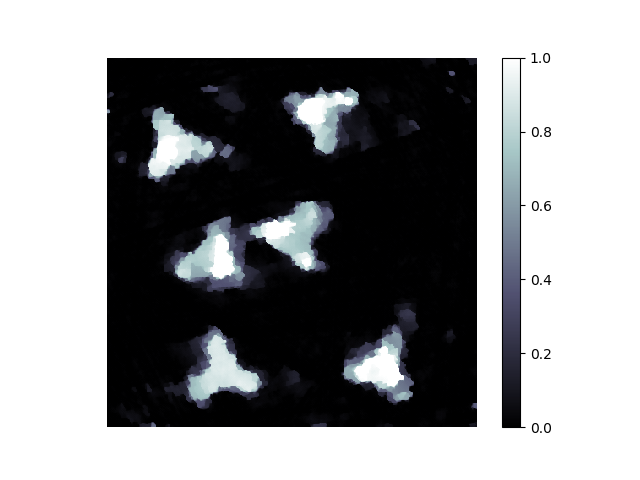}
     \vskip-0.25\baselineskip
     Gate 5
   \end{minipage}%
\caption{Test suite 2. Reconstructed spatiotemporal images from the data with different noise levels. The columns represent the different gates. The reconstructed results by the proposed method and \ac{TV}-based reconstruction method using about $14.6$dB data (rows 1 and 2), $7.69$dB data (rows 3 and 4), and $5.53$dB data (rows 5 and 6), respectively. The ground truth at each gate is displayed in the last row of \cref{Test_suite_1:multi_object_phantom}.}
\label{Test_suite_2:different_noise_levels}
\end{figure}

Moreover, as we did in test suite 1, the reconstruction results are quantitatively 
compared by using the indexes of \ac{SSIM}, \ac{PSNR} and \ac{NRMSE}. 
For the reconstruction results of various methods for different noise level data, the values 
of \ac{SSIM}, \ac{PSNR} and \ac{NRMSE} of the reconstructed spatiotemporal images compared to 
the corresponding ground truths are tabulated in \cref{Test_suite_2:noise_levels_quantitative}. 
\begin{table}[htbp]
\centering
\begin{tabular}{c | r r r r r}
&  \multicolumn{1}{c}{Gate 1} &  \multicolumn{1}{c}{Gate 2} & \multicolumn{1}{c}{Gate 3} 
 &  \multicolumn{1}{c}{Gate 4}  &  \multicolumn{1}{c}{Gate 5} \\ 
\hline                                   
 \multirow{3}{*}{Proposed}  &0.9498     &0.9660   &0.9699   &0.9697  &0.9661  \\
     								 &26.47       &30.86      &32.33    &32.43    &31.18 \\
     								 &0.1317      &0.0807    &0.0699  &0.0717  &0.0869  \\[1mm]                          
\multirow{3}{*}{\ac{TV}}  &0.8827     &0.8799   &0.8548   &0.8440  &0.8316  \\
     								 &21.68       &21.74     &20.94    &20.09    &20.43 \\
     								 &0.2286      & 0.2307    &0.2594  &0.2968  &0.2994  \\
\hline                                   
  \multirow{3}{*}{Proposed}  &0.9050     &0.9239   &0.9294   &0.9306  &0.9295  \\
     								 &24.35       &27.24      &27.99    &28.12    &28.06 \\
     								 &0.1681      &0.1224    &0.1153  &0.1177  &0.1244  \\[1mm]                          
\multirow{3}{*}{\ac{TV}}  &0.8372     &0.8431   &0.8252   &0.7930  &0.7933  \\
     								 &20.28       &20.48     &20.15    &19.39    &19.70 \\
     								 &0.2685      &0.2667    &0.2842  &0.3217  &0.3257  \\
\hline
  \multirow{3}{*}{Proposed}  &0.8564     &0.8770   &0.8838   &0.8861  &0.8856  \\
     								 &22.98       &25.17      &25.62    &26.29    &26.14 \\
     								 &0.1968      &0.1553    &0.1515  &0.1454  &0.1552  \\[1mm]                          
\multirow{3}{*}{\ac{TV}}  &0.8086     &0.7784   &0.7514   &0.7521  &0.7618  \\
     								 &19.82       &19.14     &18.78    &18.62    &19.36 \\
     								 &0.2833      &0.3112    &0.3327  &0.3518  &0.3388  \\
     \hline
\end{tabular}
\caption{Test suite 2. The values of \ac{SSIM}, \ac{PSNR} and \ac{NRMSE} of the reconstructed spatiotemporal images compared to the related ground truths for the different noise level measurements, see \cref{Test_suite_2:different_noise_levels} for detailed images. The upper and bottom of each row denote the results obtained by the proposed method and \ac{TV}-based reconstruction method respectively. Each table entry has three values that the upper is the value of \ac{SSIM}, the middle is the value of \ac{PSNR}, and the bottom is the value of \ac{NRMSE}, which corresponds to the image at the counterpart position in \cref{Test_suite_2:different_noise_levels}.}
\label{Test_suite_2:noise_levels_quantitative}
\end{table}

As listed in \cref{Test_suite_2:noise_levels_quantitative}, the associated values of \ac{SSIM} and \ac{PSNR} obtained by  
the proposed method is much bigger than \ac{TV}-based method. And the values of \ac{NRMSE} by 
the proposed method are much smaller than those by the \ac{TV}-based method. These statements are also consistent 
with the visual observation in \cref{Test_suite_2:different_noise_levels}. 
 
Hence the visual and quantitative comparisons demonstrate that the reconstructed images by the proposed method is much  
more approximated to the corresponding ground truths. Even if the projection data is disturbed by  
different noise levels, the proposed method is able to produce desirable results robustly, which can track the motions of the objects and 
reconstruct the sequential images accurately.

\subsubsection{Test suite 3: Sensitivity against selections of regularization parameters}

There are three regularization parameters $\mu_1$, $\mu_2$ and kernel width $\sigma$ required to select in the proposed model. 
The meaning of them has been illuminated in the previous sections. Hence the sensitivity test should be concerned 
against the selection of these parameters.  

A heart-like phantom at the first gate is used in this test, which is originated from \cite{GrMi07}. 
To produce the ground truths at the other gates, we take the given mass-preserving deformations 
against the phantom above. As shown in the last row of \cref{Test_suite_3:heart_phantom}, 
the ground truth at each gate is consisting of a heart-like object with different grey-value ranges. 
These images are digitized using $120\!\times\!120$ pixels, and displayed on a fixed rectangular 
domain $[-4.5, 4.5\!\times\![-4.5, 4.5]$. For the image at each gate, the noise-free data per view 
is measured by evaluating the 2D parallel beam scanning geometry with uniformly $170$ bins, 
which is defined on the range of $[-6.4, 6.4]$. Then the additive Gaussian white noise is added onto the noise-free data. 
The resulting \ac{SNR} is about $13$dB. For gate $i\,(1 \le i \le N)$, the scanning views are distributed 
on $[(i-1)\pi /5, \pi + (i-1) \pi/5]$ evenly, which totally has five views. The factor of discretized time degree is $M = 8$. 
The gradient stepsizes are set as $\alpha = 0.01$ and $\beta = 0.05$, respectively.

Having a good initial template is important to the final result. Using the same method as the previous test suites to get the 
initial template does not work here. That is because the degree of motions involved in this test is much lager than the previous ones. 
To this case, it is hard to obtain an applicable initial template by the \cref{algo:GDSB_4DCT} with given zero velocity 
field using all of gated data. So we first employ \cref{algo:GDSB_4DCT} to gain an initial template just by the 
projection data at the fist gate by 2000 iterations with the given zero velocity field, 
which is equivalent to apply the \ac{TV}-based reconstruction, 
and then apply \cref{algo:GradientDescentAlgorithmForFiniteFunctional_0} to obtain an initial velocity field 
by 500 iterations based on the initial template above. Note that the setting of the above iteration numbers is flexible, 
and the aim is to get good initial template and initial velocity field. 
Starting from these initialized values, we finally use \cref{algo:Alternating_reconstruction} to solve the proposed model. 
By selecting different values for regularization parameters and kernel width, after sufficiently 500 iterations, 
the reconstructed results are obtained, as shown in rows 2--5 of \cref{Test_suite_3:heart_phantom}. 
The detailed selections of varying parameter values can be referred to the caption. 
For comparison, we also present the reconstructed image at each single gate 
using \ac{TV}-based regularization method, as displayed in the first row of \cref{Test_suite_3:heart_phantom}. 
As shown in \cref{Test_suite_3:heart_phantom}, even through we choose different values for these regularization parameters, 
the corresponding reconstructed results by the proposed method are almost the same, and all close to the counterpart ground truths. 
However, the reconstructed result by \ac{TV}-regularization is severely degraded. 
\begin{figure}[htbp]
\centering
% First row
\begin{minipage}[t]{0.2\textwidth}%
     \centering
     \includegraphics[trim=75 25 60 40, clip, width=\textwidth]{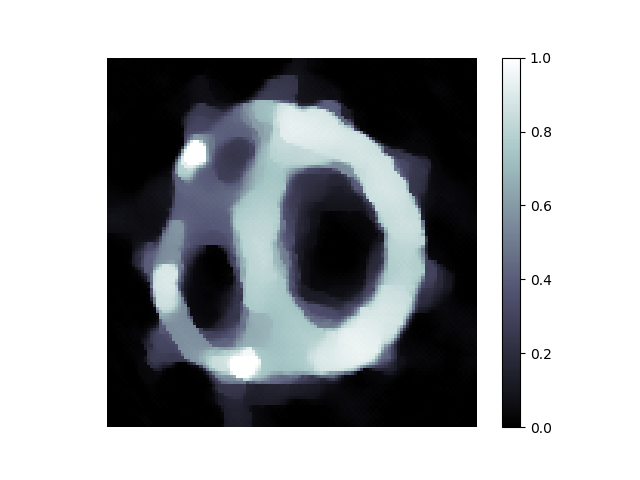}
     \vskip-0.25\baselineskip
   \end{minipage}%
   \hfill
   \begin{minipage}[t]{0.2\textwidth}%
     \centering
    \includegraphics[trim=75 25 60 40, clip, width=\textwidth]{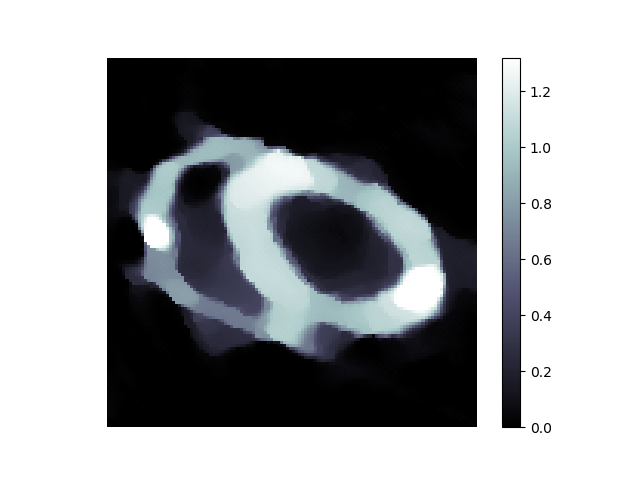}
     \vskip-0.25\baselineskip
   \end{minipage}%
   \hfill
   \begin{minipage}[t]{0.2\textwidth}%
     \centering
     \includegraphics[trim=75 25 60 40, clip, width=\textwidth]{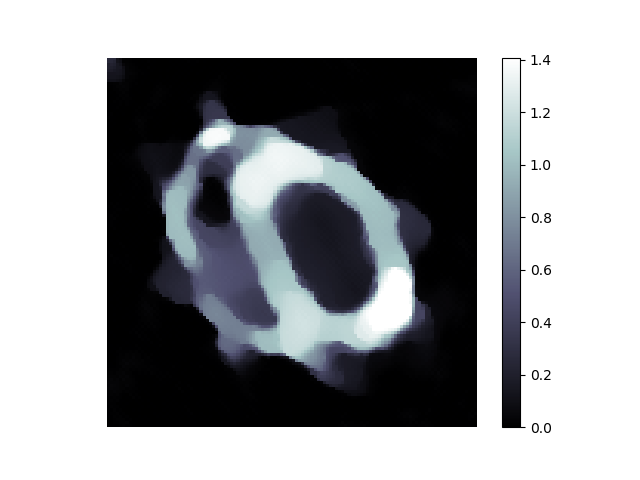}
     \vskip-0.25\baselineskip
   \end{minipage}%
      \hfill
   \begin{minipage}[t]{0.2\textwidth}%
     \centering
     \includegraphics[trim=75 25 60 40, clip, width=\textwidth]{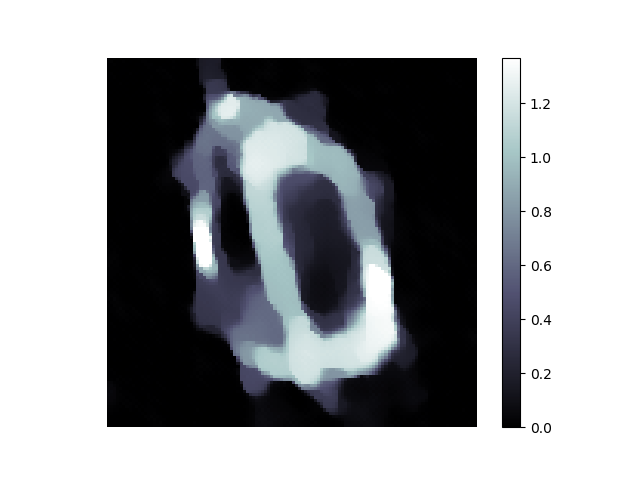}     
     \vskip-0.25\baselineskip
   \end{minipage}%
\par\medskip      
\begin{minipage}[t]{0.2\textwidth}%
     \centering
     \includegraphics[trim=75 25 60 40, clip, width=\textwidth]{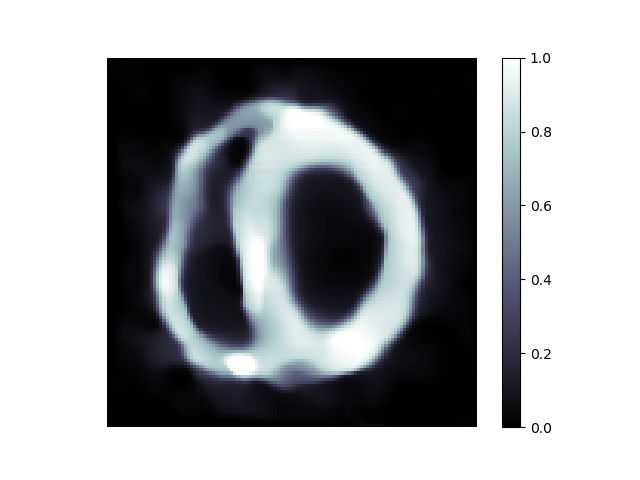}
     \vskip-0.25\baselineskip
   \end{minipage}%
   \hfill
   \begin{minipage}[t]{0.2\textwidth}%
     \centering
    \includegraphics[trim=75 25 60 40, clip, width=\textwidth]{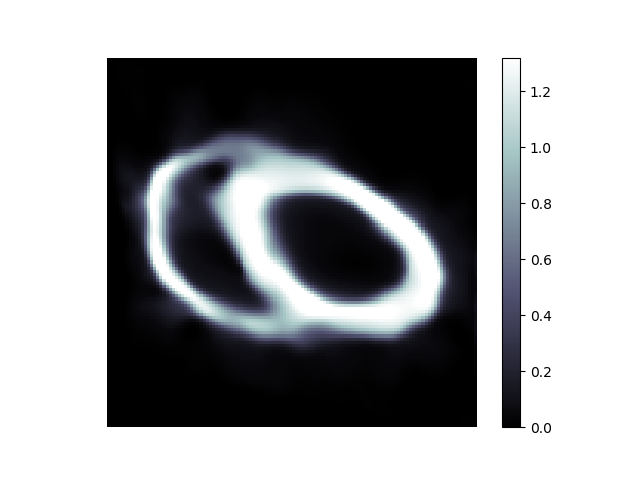}
     \vskip-0.25\baselineskip
   \end{minipage}%
   \hfill
   \begin{minipage}[t]{0.2\textwidth}%
     \centering
     \includegraphics[trim=75 25 60 40, clip, width=\textwidth]{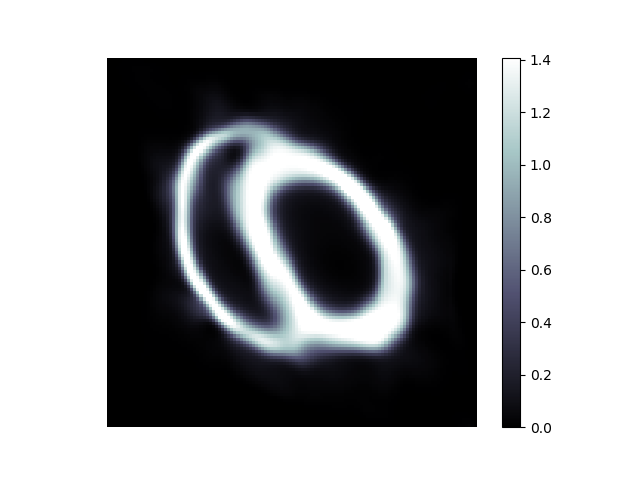}
     \vskip-0.25\baselineskip
   \end{minipage}%
      \hfill
   \begin{minipage}[t]{0.2\textwidth}%
     \centering
     \includegraphics[trim=75 25 60 40, clip, width=\textwidth]{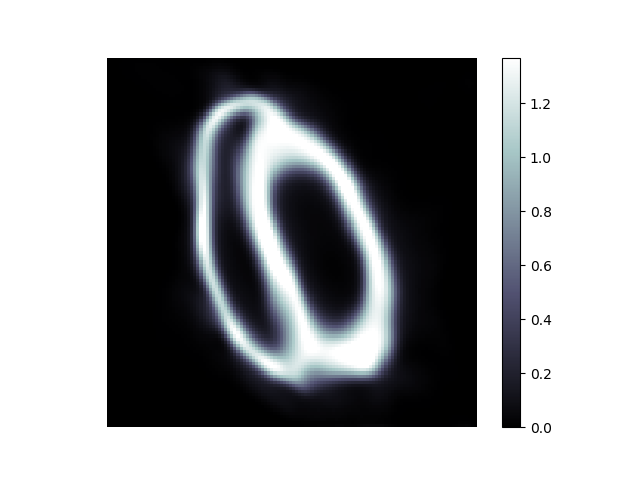}     
     \vskip-0.25\baselineskip
   \end{minipage}%
\par\medskip      
	\begin{minipage}[t]{0.2\textwidth}%
     \centering
     \includegraphics[trim=75 25 60 40, clip, width=\textwidth]{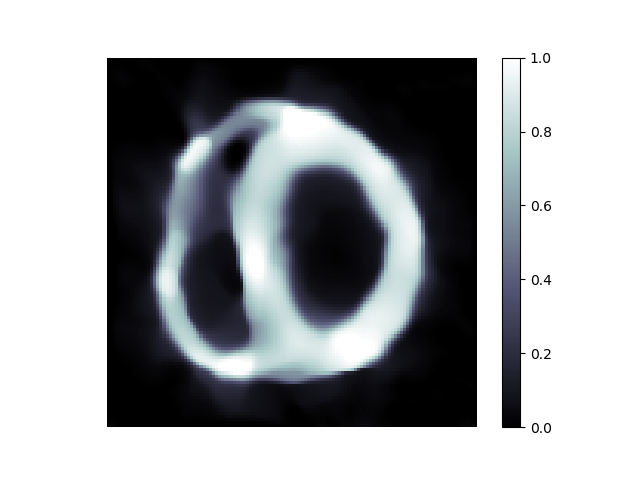}
     \vskip-0.25\baselineskip
   \end{minipage}%
   \hfill
   \begin{minipage}[t]{0.2\textwidth}%
     \centering
    \includegraphics[trim=75 25 60 40, clip, width=\textwidth]{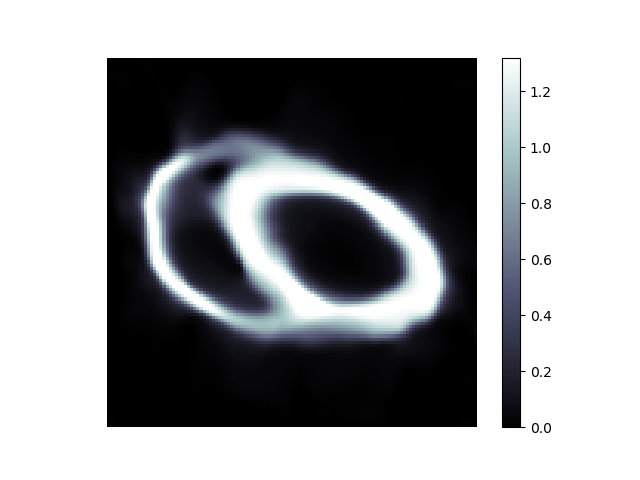}
     \vskip-0.25\baselineskip
   \end{minipage}%
   \hfill
   \begin{minipage}[t]{0.2\textwidth}%
     \centering
     \includegraphics[trim=75 25 60 40, clip, width=\textwidth]{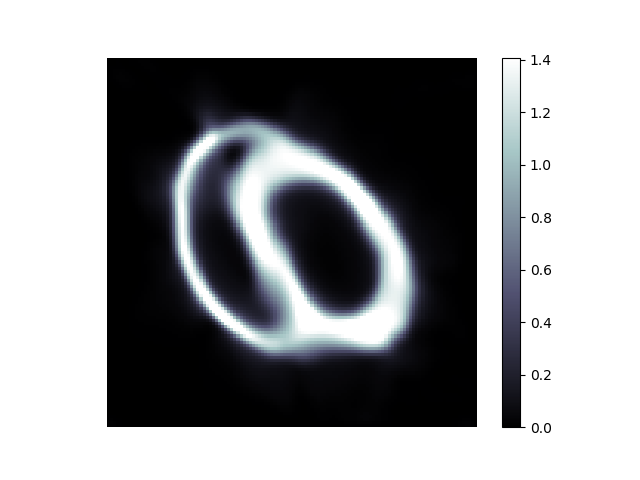}
     \vskip-0.25\baselineskip
   \end{minipage}%
      \hfill
   \begin{minipage}[t]{0.2\textwidth}%
     \centering
     \includegraphics[trim=75 25 60 40, clip, width=\textwidth]{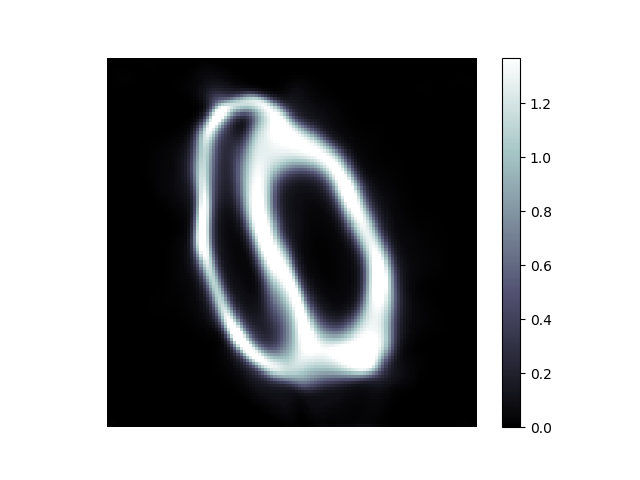}     
     \vskip-0.25\baselineskip
   \end{minipage}%
\par\medskip      
\begin{minipage}[t]{0.2\textwidth}%
     \centering
     \includegraphics[trim=75 25 60 40, clip, width=\textwidth]{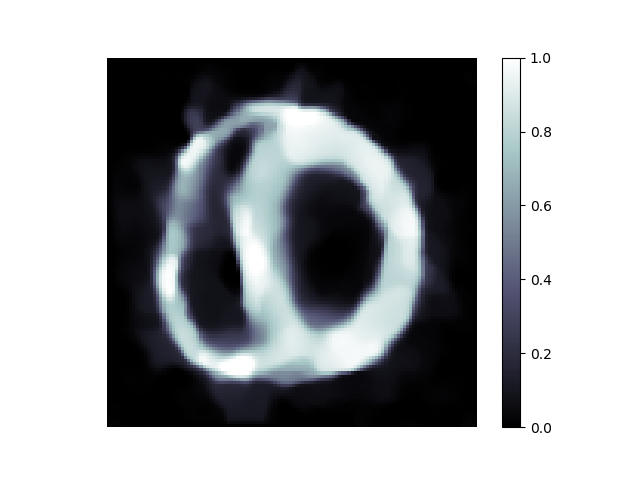}
     \vskip-0.25\baselineskip
   \end{minipage}%
   \hfill
   \begin{minipage}[t]{0.2\textwidth}%
     \centering
    \includegraphics[trim=75 25 60 40, clip, width=\textwidth]{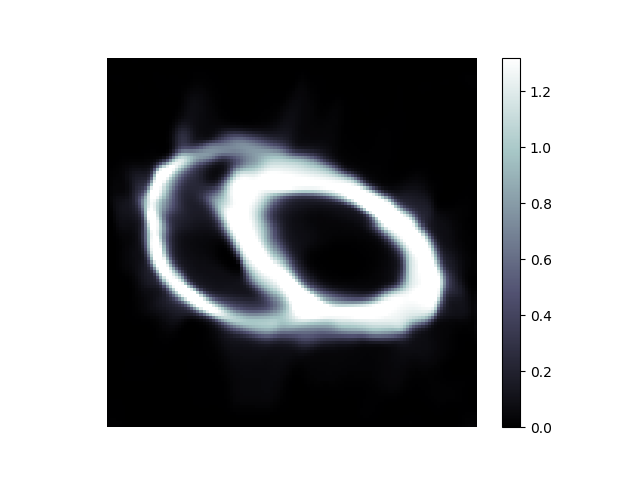}
     \vskip-0.25\baselineskip
   \end{minipage}%
   \hfill
   \begin{minipage}[t]{0.2\textwidth}%
     \centering
     \includegraphics[trim=75 25 60 40, clip, width=\textwidth]{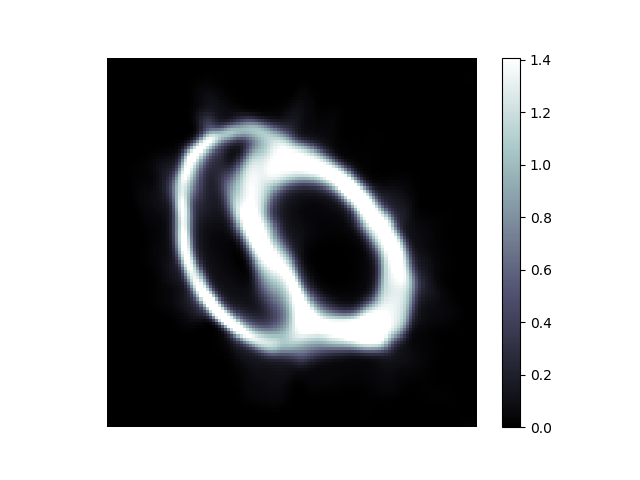}
     \vskip-0.25\baselineskip
   \end{minipage}%
      \hfill
   \begin{minipage}[t]{0.2\textwidth}%
     \centering
     \includegraphics[trim=75 25 60 40, clip, width=\textwidth]{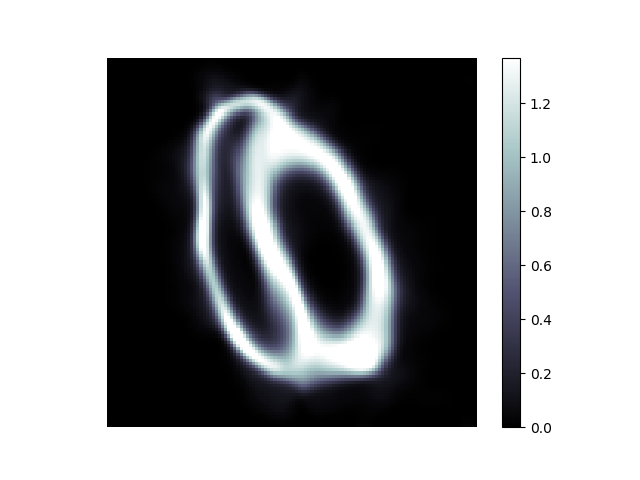}     
     \vskip-0.25\baselineskip
   \end{minipage}%
\par\medskip   
\begin{minipage}[t]{0.2\textwidth}%
     \centering
     \includegraphics[trim=75 25 60 40, clip, width=\textwidth]{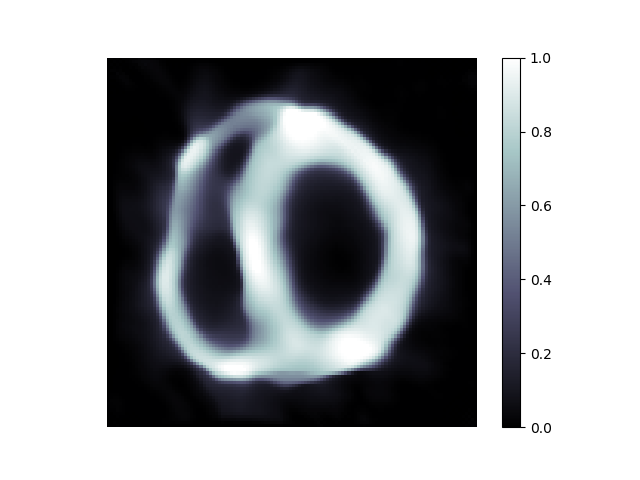}
     \vskip-0.25\baselineskip
   \end{minipage}%
   \hfill
   \begin{minipage}[t]{0.2\textwidth}%
     \centering
    \includegraphics[trim=75 25 60 40, clip, width=\textwidth]{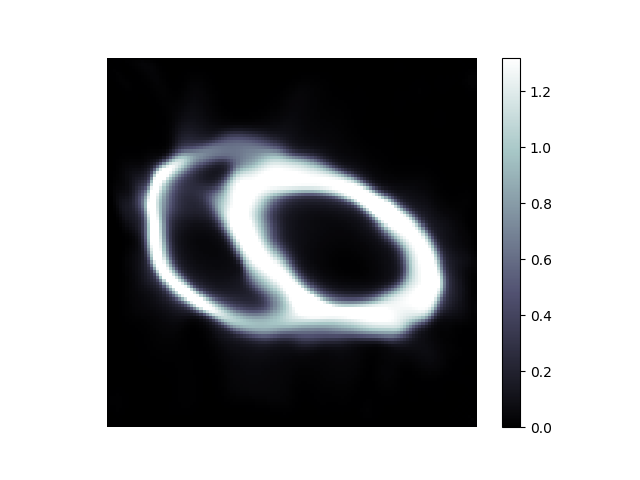}
     \vskip-0.25\baselineskip
   \end{minipage}%
   \hfill
   \begin{minipage}[t]{0.2\textwidth}%
     \centering
     \includegraphics[trim=75 25 60 40, clip, width=\textwidth]{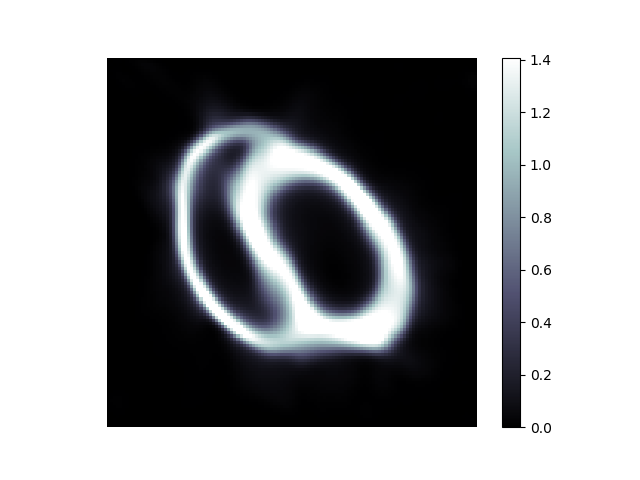}
     \vskip-0.25\baselineskip
   \end{minipage}%
      \hfill
   \begin{minipage}[t]{0.2\textwidth}%
     \centering
     \includegraphics[trim=75 25 60 40, clip, width=\textwidth]{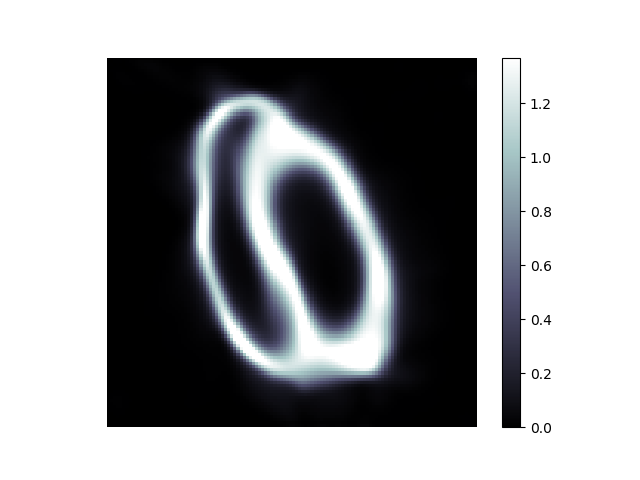}     
     \vskip-0.25\baselineskip
   \end{minipage}%
\par\medskip   
   \begin{minipage}[t]{0.2\textwidth}%
     \centering
     \includegraphics[trim=75 25 60 40, clip, width=\textwidth]{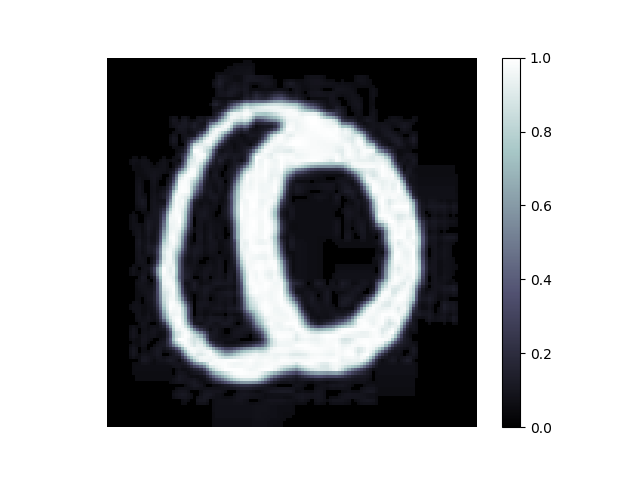}
     \vskip-0.25\baselineskip
     Gate 1
   \end{minipage}%
   \hfill
   \begin{minipage}[t]{0.2\textwidth}%
     \centering
    \includegraphics[trim=75 25 60 40, clip, width=\textwidth]{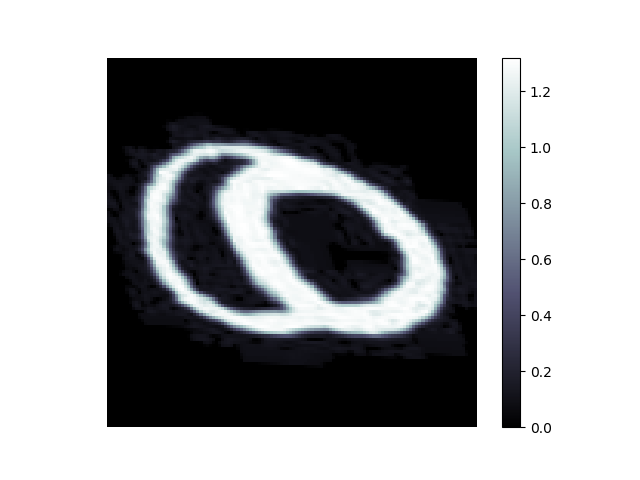}
     \vskip-0.25\baselineskip
     Gate 2
   \end{minipage}%
   \hfill
   \begin{minipage}[t]{0.2\textwidth}%
     \centering
     \includegraphics[trim=75 25 60 40, clip, width=\textwidth]{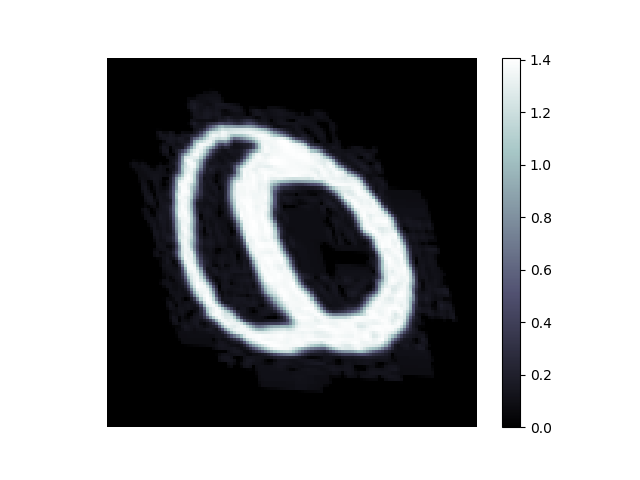}
     \vskip-0.25\baselineskip
     Gate 3
   \end{minipage}%
      \hfill
   \begin{minipage}[t]{0.2\textwidth}%
     \centering
     \includegraphics[trim=75 25 60 40, clip, width=\textwidth]{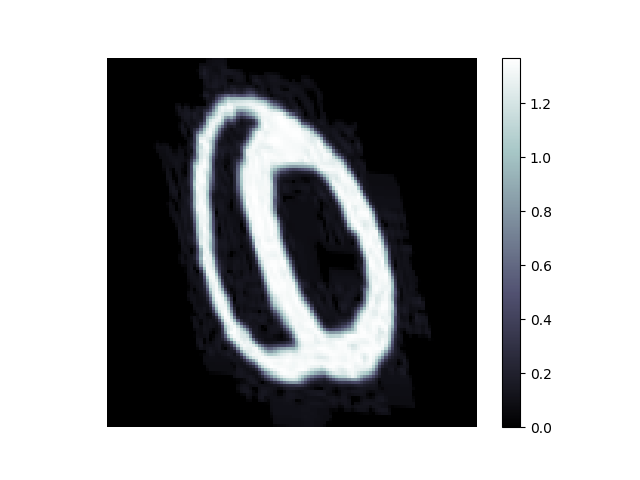}     
     \vskip-0.25\baselineskip
     Gate 4
   \end{minipage}%
\caption{Test suite 3. Reconstructed spatiotemporal images by selecting different regularization parameters. The columns represent the four gates. The first row shows  the reconstructed images by \ac{TV}-based method for each gate with $\mu_1 = 0.02$. The rows 2--5 respectively show the reconstructed spatiotemporal images by the proposed method with parameter pairs ($\mu_1$, $\mu_2$, $\sigma$) chosen as $(0.02, 10^{-6}, 1.25)$, $(0.02, 10^{-6}, 1.0)$, $(0.01, 10^{-7}, 1.0)$, $(0.03, 10^{-7}, 1.0)$. The last row shows the ground truth of each gate.}
\label{Test_suite_3:heart_phantom}
\end{figure}

Furthermore, the reconstruction results are quantitatively compared to the corresponding ground truths 
by using \ac{SSIM}, \ac{PSNR} and \ac{NRMSE}. These calculated indexes are listed in \cref{Test_suite_3:Sensitivity_table}. 
As given in the table, the corresponding \ac{SSIM} and \ac{PSNR} indexes of the proposed method are 
relatively larger than those obtained by \ac{TV}-based method, and the indexes of \ac{NRMSE} are smaller.  
Remark that the obtained indexes by the proposed method are quite similar with each other even if the  
different parameter pairs are selected. 
\begin{table}[htbp]
\centering
\begin{tabular}{c | r r r r}
&  \multicolumn{1}{c}{Gate 1} &  \multicolumn{1}{c}{Gate 2} & \multicolumn{1}{c}{Gate 3} 
 &  \multicolumn{1}{c}{Gate 4}   \\ 
\hline                                   
 \multirow{3}{*}{\ac{TV}} & 0.6403        &  0.7054     &  0.6731       &  0.6671  \\
   									& 16.81           &  18.53        &  17.65         &  16.60   \\ 
   									& 0.3136        & 0.2964      &  0.3388      & 0.3773 \\
\hline                                   
 \multirow{3}{*}{Proposed}  &  0.7603      &  0.7692   &  0.8102     &  0.8182   \\
   									 &  20.18         &  21.95     &  23.15        &  23.01   \\ 
   									 &  0.2127       &  0.2001   &  0.1798     &  0.1803 \\
\hline
\multirow{3}{*}{Proposed}  &  0.7612   &  0.7525   &  0.7941   &  0.8083    \\  
  									&   20.00    &  21.70     &  22.47    &  22.43   \\
 									&  0.2170  & 0.2058    & 0.1945    & 0.1927  \\
\hline                                  
\multirow{3}{*}{Proposed}   &  0.7485       &  0.7563   &  0.7998   &  0.8085    \\  
   									 & 19.83           &  21.89      &  22.45    &  22.20   \\
   									 & 0.2215        & 0.2014     & 0.1949    & 0.1979 \\
\hline                                   
\multirow{3}{*}{Proposed}   	&  0.7502       &  0.7560   &  0.7947   &  0.8021   \\ 
 										& 19.63          &  21.80     &  22.21      &  22.03  \\
 										& 0.2265       &  0.2034   & 0.2003    & 0.2019  \\
     \hline
\end{tabular}
\caption{Test suite 3. The values of \ac{SSIM}, \ac{PSNR} and \ac{NRMSE} of the reconstructed spatiotemporal images compared to the related ground truths for varying values of the regularization parameters $\mu_1$, $\mu_2$, and the kernel width $\sigma$, see \cref{Test_suite_3:heart_phantom} for the detailed images. Each entry has three  values, where the upper is the value of \ac{SSIM}, the middle is the value of \ac{PSNR}, and the bottom is the value of \ac{NRMSE}, which corresponds to the image on the counterpart position in \cref{Test_suite_3:heart_phantom}. Row 1: \ac{TV}-based regularization method, and rows 2--5: The proposed method with different selecting parameters. }
\label{Test_suite_3:Sensitivity_table}
\end{table}

As visual and quantitative comparisons by \cref{Test_suite_3:heart_phantom} and \cref{Test_suite_3:Sensitivity_table}, 
the proposed method is not so sensitive against the selection of the regularization parameters to some extent. However, those values are 
selected too big or too small, which would cause inappropriate regularized results.

\section{Discussion}\label{sec:Discussion} 

Here we further discuss several important issues about the model and related algorithm. 

\subsection{An alternative model}
\label{sec:Alternative_model}

As stated in the beginning of \cref{sec:OT_recon_model}, one method to ensure a Hilbert space 
being admissible is using the differential operator $\diffoperator$. 
For instance, the presented model with \ac{ODE} constraint in \cite{ChGrOz19} can be written as 
\begin{equation}\label{eq:VarReg_LDDMM_previous}
\begin{split}
 &\min_{\substack{\template \in \RecSpace \\ \velocityfield \in \Xspace{2}}} \int_{0}^{1} \left[\DataDisc_{\ForwardOp_t, \data_t}\bigl(\diffeo_{0,t}^{\velocityfield} . \template\bigr)  + 
   \mu_2 \int_0^t \int_{\Omega}\vert\diffoperator\velocityfield(\tau, x)\vert^2\dint x\dint \tau \right] \dint t   + \mu_1 \RegFunc_1(\template)  \\
 & \quad\,\, \text{s.t.  $\diffeo_{0,t}^{\velocityfield}$ solves \ac{ODE} \cref{eq:FlowEq},}
  \end{split}
\end{equation}   
where \cref{eq:LDDMMPDEConstrained_mp_19} is the equivalent \ac{PDE}-constrained optimal control formulation. 

Inspired by the proposed model \cref{eq:VarReg_LDDMM_2} and the model \cref{eq:VarReg_LDDMM_previous}, 
an alternative model using the differential operator $\diffoperator$ is formulated as 
\begin{equation}\label{eq:VarReg_LDDMM_diffoperator}
\begin{split}
 &\min_{\substack{\template \in \RecSpace \\ \velocityfield \in \Xspace{2}}} \int_{0}^{1} \left[\DataDisc_{\ForwardOp_t, \data_t}\bigl(\diffeo_{0,t}^{\velocityfield} . \template\bigr)  + 
   \mu_2 \int_0^t \int_{\Omega}\diffeo_{0,\tau}^{\velocityfield} . \template(x)\vert\diffoperator\velocityfield(\tau, x)\vert^2\dint x\dint \tau \right] \dint t   + \mu_1 \RegFunc_1(\template)  \\
 & \quad\,\, \text{s.t.  $\diffeo_{0,t}^{\velocityfield}$ solves \ac{ODE} \cref{eq:FlowEq}.}
  \end{split}
\end{equation}   
It is easy to obtain its equivalent \ac{PDE}-constrained optimal control formulation.  
Furthermore, the time-discretized versions of \cref{eq:VarReg_LDDMM_diffoperator} can be 
readily obtained following \cref{sec:time_discretized_version}. 

Compared with \cref{eq:VarReg_LDDMM_previous}, the unknown time-dependent image $\diffeo_{0,\tau}^{\velocityfield} . \template$ 
acting as the weight function is introduced into the shape regularization of \cref{eq:VarReg_LDDMM_diffoperator}.  
In contrast to \cref{eq:VarReg_LDDMM_2}, the differential operator $\diffoperator$ is explicitly used  
to construct that shape regularization. But these modifications would make the alternative model 
harder to solve. Through comparing with the models \cref{eq:VarReg_LDDMM_2}, \cref{eq:VarReg_LDDMM_previous} 
and \cref{eq:VarReg_LDDMM_diffoperator}, the relationship between \ac{LDDMM} and optimal transportation becomes more clear.

\subsection{Algorithmic initial values}
\label{sec:construct_initial_value}

Since the proposed model is nonlinear and noncovex due to the composites of the template and 
diffeomorphic deformations (generated by the velocity field), the selection of algorithmic initial value 
has important influence on the final result. 

During the implementation, it makes sense that the initial velocity field is always chosen as zero, 
and the resulting initial deformations are the identity deformation. That is because the optimal 
deformation is hopefully close to the identity deformation, which is characterized by 
the shape regularization in \cref{eq:VarReg_LDDMM_2}. Next we focus on the selection of the initial template. 
As we have tested by several examples in \cref{sec:numerical_experiments}, the selection of the initial template 
depends mainly on the degree of deformations of the ground truths at different gates. For instance in test suite 1, 
the deformation degree is relatively small, so we apply \cref{algo:GDSB_4DCT} to obtain an initial template 
after dozens of iterations by using all of the data with fixed zero velocity field. This means we treat the 
spatiotemporal reconstruction as a static one, and then use \ac{TV}-regularization method to reconstruct. 
Even though the resulting initial template is blurring, it looks like the ground truth at the first gate and can be 
act as an appropriate initial value. On the other hand, if the deformation degree is relatively large as in test suite 3, the method 
above does not work because it would result in an initial template quite dissimilar as the ground truth at the first gate. 
Hence one alternative method is to employ \cref{algo:GDSB_4DCT} to gain an initial template by the projection data 
only at the first gate by sufficient iterations, then apply \cref{algo:GradientDescentAlgorithmForFiniteFunctional_0} to 
obtain an better initial velocity field by enough iterations based on the initial template above and initially zero velocity field. 
Finally, the applicable initial template and initial velocity field are obtained for the proposed algorithm.

\subsection{The factor of discretized time degree}
\label{sec:factor}

Another issue is the setting of the factor $M$ of discretized time degree. As stated in \cref{sec:Discretization}, this factor 
determines the discretized degree of each subinterval $[t_i, t_{i+1}]$ for $0 \leq i \leq N-1$.  Setting $M=1$ means that 
the discretized time grid is coincident with the gating grid. For this case, the deformation of the images between 
adjacent gates is characterized by the linear displacement field from the view of numerical discretization. Besides that, the discretized time 
grid is finer than the gating grid by letting $M > 1$, which results in the deformation composited by multiple 
linear displacement fields (the number is $M$.) numerically. 

As we have tested in \cref{sec:numerical_experiments}, the larger deformation or motion between 
the adjacent images, the lager $M$ should be chosen. For example, we set $M = 2$ in test suites 1 and 2, and let $M = 8$ in test suite 3. 
That is because the deformation degree of the latter is larger than the former. However, we further found if the value of the factor has been 
set to be sufficiently large, using a larger one again would have no notable improvement for the ultimately results.  
In addition, the different subintervals of gating grid can be discretized adaptively according to the variability of motions.

\subsection{Extended models}
\label{sec:models_extension}

Inspired by the proposed model, we come up with several potential models also based on diffeomorphic optimal transportation. 

\paragraph{Image registration} Given the template image $I_0 \colon \domain \to \Real$ 
and the target image $I_1\colon \domain \to \Real$. Assume that 
they are both nonnegative and have the same mass. Using \cref{thm:EquivalenceWD_flow}, the variational model for image registration can be formulated as 
\begin{equation}\label{eq:VarReg_time_discrete_extension_registration}
\begin{split}
 & \min_{\velocityfield \in \Xspace{2}}\|\diffeo_{0,1}^{\velocityfield} . I_0 - I_1\|_2^2  + 
   \mu \int_0^1 \int_{\Omega}\diffeo_{0,t}^{\velocityfield} . I_0(x)\vert\velocityfield(t, x)\vert^2\dint x\dint t   \\
 &\quad\,\,  \text{s.t.  $\diffeo_{0,t}^{\velocityfield}$ solves \ac{ODE} \cref{eq:FlowEq},}
  \end{split}
  \end{equation}
  where the $\mu$ is the positive regularization parameter. 
  
\paragraph{Sequential image registration} Given the the time-series image $I_{t_i} \colon \domain \to \Real$ 
for $0 \leq t_i \leq 1$ and $0 \leq i \leq N$. Suppose that 
they are all nonnegative and have the same mass. The variational model for sequential image registration can be presented as 
\begin{equation}\label{eq:VarReg_time_discrete_extension_Sequential}
\begin{split}
 & \min_{\velocityfield \in \Xspace{2}}\frac{1}{N}\sum_{i=1}^{N} \left[\|\diffeo_{0,t_i}^{\velocityfield} . I_{t_0} - I_{t_i}\|_2^2  + 
   \mu \int_0^{t_i} \int_{\Omega}\diffeo_{0,\tau}^{\velocityfield} . I_{t_0}(x)\vert\velocityfield(\tau, x)\vert^2\dint x\dint \tau \right]  \\
 &\quad\,\,  \text{s.t.  $\diffeo_{0,t}^{\velocityfield}$ solves \ac{ODE} \cref{eq:FlowEq}.}
  \end{split}
  \end{equation}
The model \cref{eq:VarReg_time_discrete_extension_Sequential} merely gives the time-discretized version for sequential image registration. 
The time-continuous version can be obtained naturally.   
  
\paragraph{Indirect image registration}  Assume that the template image $I_0 \colon \domain \to \Real$ is given,  
and the indirect measurement $\data_1$ is obtained from the target image. Assume that 
both of the images are nonnegative and have the same mass. The variational model for indirect image registration can be formulated as 
\begin{equation}\label{eq:VarReg_time_discrete_Indirect_registration}
\begin{split}
 & \min_{\velocityfield \in \Xspace{2}}\|\ForwardOp\bigl(\diffeo_{0,1}^{\velocityfield} . I_0\bigr) - \data_1\|_2^2  + 
   \mu \int_0^{1} \int_{\Omega}\diffeo_{0,t}^{\velocityfield} . I_0(x)\vert\velocityfield(t, x)\vert^2\dint x\dint t   \\
 &\quad\,\,  \text{s.t.  $\diffeo_{0,t}^{\velocityfield}$ solves \ac{ODE} \cref{eq:FlowEq}.}
  \end{split}
  \end{equation}
Moreover, the case for sequentially indirect image registration has been already  
proposed in \cref{eq:VarReg_LDDMM_deformation_time_discrete2}. Correspondingly, the time-continuous cases can be naturally achieved. 

Note that the data fitting terms above can be modified according to the practical requirements. From the numerical point of view, 
the proposed algorithm in \cref{algo:GradientDescentAlgorithmForFiniteFunctional_0} can be simply adapted to solve the 
extended models \cref{eq:VarReg_time_discrete_extension_registration}--\cref{eq:VarReg_time_discrete_Indirect_registration}. Therefore, 
if we consider the models with \ac{ODE}-constrained formulation under appropriate conditions, the proposed algorithm provides 
a new scheme to solve the models based on $\LpSpace^2$ Wasserstein distance in Benamou--Brenier formulation.

\section{Conclusion}\label{sec:Conclusions}

In this work, the $\LpSpace^2$ Wasserstein distance in Benamou--Brenier formulation is used to characterize the optimal transport cost, 
and the unknown velocity field is restricted onto the admissible Hilbert space, which results in a diffeomorphic optimal transportation 
among the mass-preserving image flows. Along the general framework for spatiotemporal imaging that presented in \cite{ChGrOz19}, 
a joint variational model has been investigated for the spatiotemporal image reconstruction with diffeomorphic and mass-preserving property. 
Therefore, the proposed model is a production of combining the Wasserstein distance of optimal transportation and the 
flow of diffeomorphisms involved in \ac{LDDMM}, which is suitable for the scenario of spatiotemporal imaging with large diffeomorphic and 
mass-preserving deformations. 

Additionally, the equivalent \ac{PDE}-constrained optimal control formulation is obtained from the proposed model 
with \ac{ODE} constraint. Using the \ac{PDE}-constrained form, the proposed model has been theoretically compared 
against the existing joint variational model based on Wasserstein distance in \cite{Br10}, which demonstrates that the former     
can generate the sufficiently smooth velocity field, and further guarantee the flow of large non-rigid diffeomorphic deformations. 
And the optimal velocity field of the former is unnecessarily vanishing on the supports of the reconstructed images at the 
end time points, which implies the consistency between the time-continuous version and its associated time-discretized one. 
But those are not the cases for the model in \cite{Br10} under appropriate conditions. Furthermore, the comparison is also 
performed for the proposed model in \ac{ODE}-constrained form and the one 
based on the \ac{LDDMM} consistent growth model proposed in our previous work \cite{ChGrOz19}, 
which makes clear the relationship between the optimal transportation and \ac{LDDMM}. 

The time-discretized versions with/without the `virtual' template of the proposed model 
have also been presented, which are solved by the alternating minimization algorithm. 
Particularly, an alternating gradient descent algorithm was designed to solve the time-discretized 
proposed model with the `virtual' template, where the most calculations only 
involve the easy-to-implement linearized deformations. Considering the 
gained \ac{ODE}-constrained form under appropriate conditions, this algorithm provides a new idea to solve 
the other models based on $\LpSpace^2$ Wasserstein distance in Benamou--Brenier formulation. 

The performance of the proposed model and associated algorithm is finally validated by several numerical experiments in 2D space and time 
tomography with sparse-view and/or different noise level data measured from topology-preserving and mass-preserving sequential images. 
Using the noise-free and sparse-view projection data, we evaluated the overview performance of the proposed method, including 
numerical convergence, reconstructed image quality, and mass-preserving property. The numerical results have showed the desirable 
performance with respect to those aspects. In particular, the proposed method has much better reconstruction accuracy than the $\LpSpace^2$-gradient 
descent scheme and the \ac{TV}-regularization method from both the visual and quantitative perspectives. And the proposed method 
generated smooth optimal velocity field, but that is not the case of the $\LpSpace^2$-gradient descent scheme. 

Moreover, we have tested the robustness against the different noise levels for the proposed method. Even if the projection data is disturbed by  
different noise levels, the proposed method can always track the motions of the objects and reconstruct more accurate sequential 
images. Through the numerical validation, we also found that the proposed method is not so sensitive against the selection of the related 
regularization parameters. Conclusively, the proposed method can stably improve the quality of the reconstructed images in 
spatiotemporal imaging with large diffeomorphic and mass-preserving deformations. 

We further made a lot of important discussions about the proposed model and algorithms. Particularly, an alternative model was 
proposed for spatiotemporal imaging, which is also coupled the thoughts of \ac{LDDMM} and optimal transportation. Although this alternative 
is more complicated than the proposed one, it might have some potential application. Illustrating with the different numerical tests, we also 
analyzed the selections of algorithmic initial value and the factor of discretized time degree, which provides the guidelines for the 
numerical implementation of the proposed model. Inspired by the proposed model, we came up with several extended models with applications to  
more image processing and biomedical imaging. Importantly, the proposed algorithm provides 
a new scheme to solve the models based on $\LpSpace^2$ Wasserstein distance in Benamou--Brenier formulation. 

We are going to concern the more theoretical aspects of the proposed model and algorithms, 
the further extensions following the studied framework, and the applications and related theory to more spatiotemporal 
biomedical imaging and image processing.

\appendix
\section{Optimality conditions}
\label{sec:Optimality_conditions}

The optimality conditions for \cref{eq:VarReg_LDDMM_2} and \cref{eq:time_discretized_VarReg_LDDMM_2}.  Let us begin with the following lemma. 

\begin{lemma}\label{lem:DiffEvolutionOperator_tt}
Let $\velocityfield,\velocityfieldother \in \Xspace{2}$, and $\diffeo_{0,t}^{\velocityfield}$ denote the solution to 
the \ac{ODE} in \cref{eq:FlowEq} with given $\velocityfield$, and $\template \in \LpSpace^2(\domain,\Real)$ be differentiable. 
Using the mass-preserving deformation in \cref{eq:MassPreservedDeform}, then 
\begin{equation}\label{eq:DiffEvolutionOperator_tt}
  \frac{d}{d \epsilon} \bigl(\diffeo_{0,t}^{\velocityfield +  \epsilon\velocityfieldother} . \template\bigr)(x)  \Bigl\vert_{\epsilon=0}  = \bigl\vert \Diff\bigl(\diffeo_{t,0}^{\velocityfield}\bigr)(x) \bigr\vert \Div\bigl( \template\,h_{t, 0}^{\velocityfield}\bigr) \circ \diffeo_{t,0}^{\velocityfield}(x)                
\end{equation}  
for $x \in \domain$ and $0 \leq t \leq 1$, where 
\begin{equation}
   \label{eq:midfunction}
    h_{t, 0}^{\velocityfield}
    = - \int_{0}^{t} \Diff\bigl( \gelement{\tau,0}{\velocityfield} \bigr)\bigl( \gelement{0,\tau}{\velocityfield} \bigr) 
             \Bigl( \velocityfieldother\bigl(\tau, \gelement{0,\tau}{\velocityfield} \bigr) \Bigr) \dint \tau.
           \end{equation}
\end{lemma}
\begin{proof}
By the mass-preserving deformation in \cref{eq:MassPreservedDeform}, and using \cref{eq:diff_trans}, we have 
\[
\frac{d}{d \epsilon} \bigl(\diffeo_{0,t}^{\velocityfield +  \epsilon\velocityfieldother} . \template\bigr)(x)  \Bigl\vert_{\epsilon=0}   = \frac{d}{d \epsilon}  \bigl\vert \Diff\bigl(\diffeo_{t,0}^{\velocityfield +  \epsilon\velocityfieldother}\bigr)(x) \bigr\vert \template \circ \diffeo_{t,0}^{\velocityfield +  \epsilon\velocityfieldother}(x) \Bigl\vert_{\epsilon=0}. 
\]
Using the result from \cite[Lemma A.1]{ChGrOz19}, then  
\[
    \frac{d}{d \epsilon} \gelement{t, 0}{\velocityfield + \epsilon\velocityfieldother} (x) \Bigl\vert_{\epsilon=0}
    = h_{t, 0}^{\velocityfield} \circ \diffeo_{t,0}^{\velocityfield}(x). 
 \]
Following the proof of \cite[Theorem 8.3]{ChOz18}, we have   
\[
\frac{d}{d \epsilon} \bigl\vert \Diff\bigl(\diffeo_{t,0}^{\velocityfield +  \epsilon\velocityfieldother}\bigr)(x) \bigr\vert  \Bigl\vert_{\epsilon=0}  = \bigl\vert \Diff\bigl(\diffeo_{t,0}^{\velocityfield}\bigr)(x) \bigr\vert \Div\bigl(h_{t, 0}^{\velocityfield}\bigr) \circ \diffeo_{t,0}^{\velocityfield}(x).
\]
By the chain rule we obtain the result of \cref{eq:DiffEvolutionOperator_tt}.
\end{proof}

Then the following result is obtained immediately. 

\begin{lemma}\label{lem:data_matching_derivative}
Let the assumptions in \cref{lem:DiffEvolutionOperator_tt} hold and 
$\DataDisc_{\ForwardOp_t, g_t}\colon \RecSpace \to \Real$ be defined as \cref{eq:LDDMM_match_short}. 
Assuming that $\DataDisc_{\ForwardOp_t, g_t}$ is differentiable. Then
\begin{equation}\label{eq:Energy_func_deformation}
  \frac{d}{d \epsilon} \DataDisc_{\ForwardOp_t, g_t} \bigl(\diffeo_{0,t}^{\velocityfield + \epsilon\velocityfieldother} . \template \bigr)\Bigl\vert_{\epsilon=0} 
    =  \int_{0}^{t}\Bigl\langle \gelement{0, \tau}{\velocityfield} . \template  \grad\Bigl(\partial\DataDisc_{\ForwardOp_t, g_t} \bigl(\diffeo_{0,t}^{\velocityfield} . \template\bigr)\bigl( \gelement{\tau,t}{\velocityfield}\bigr)\Bigr),  \velocityfieldother(\tau, \Cdot) \Bigr\rangle_{\LpSpace^2(\domain,\Real^n)} \dint \tau,
  \end{equation}  
where $\partial\DataDisc_{\ForwardOp_t, g_t}$ represents the gradient of $\DataDisc_{\ForwardOp_t, g_t}$.
\end{lemma}

In what follows we derive the optimality conditions for \cref{eq:VarReg_LDDMM_2}.
\begin{theorem}\label{thm:energy_functional_derivative_2}
Let the assumptions in \cref{lem:data_matching_derivative} hold and 
$\GoalFunctionalV_C \colon \RecSpace \times \Xspace{2} \to \Real$ denote the objective 
functional in \cref{eq:VarReg_LDDMM_2} of time-continuous version, i.e.,
\begin{equation}\label{eq:LDDMM_vector_part2}
\GoalFunctionalV_C(\template, \velocityfield) :=  \int_{0}^{1} \left[\DataDisc_{\ForwardOp_t, \data_t}\bigl(\diffeo_{0,t}^{\velocityfield} . \template\bigr)  + 
   \mu_2 \int_0^t \int_{\Omega}\diffeo_{0,\tau}^{\velocityfield} . \template(x)\vert\velocityfield(\tau, x)\vert^2\dint x\dint \tau \right] \dint t  + \mu_1 \RegFunc_1(\template).
\end{equation}
Assume that the regularization term $\RegFunc_1$ is differentiable, and $\Vspace$ is an \ac{RKHS} with a reproducing 
kernel $\kernel \colon \domain \times \domain \to \Matrix_{+}^{d \times d}$. 
Then the gradient (i.e., $\LpSpace^2$-gradient) with regard to the velocity field $\velocityfield$ of $\GoalFunctionalV_C(\template,\Cdot)$ is 
\begin{multline}\label{eq:Energy_functional_gradient_2_l2}
  \grad_{\velocityfield}\GoalFunctionalV_C(\template, \velocityfield)(t,\Cdot)  
    = \diffeo_{0,t}^{\velocityfield} . \template \int_{t}^{1} \grad\Bigl(\Bigl[\partial\DataDisc_{\ForwardOp_{\tau}, g_{\tau}} \bigl(\diffeo_{0,\tau}^{\velocityfield} . \template\bigr) + \mu_2 (1-\tau) \vert\velocityfield(\tau, \cdot)\vert^2\Bigr]\bigl( \gelement{t,\tau}{\velocityfield}\bigr)\Bigr) \dint \tau  \\
    + 2\mu_2(1-t)\diffeo_{0,t}^{\velocityfield} . \template\, \velocityfield(t,\Cdot) 
  \end{multline}
and the $\Xspace{2}$-gradient with regard to the 
velocity field $\velocityfield$ of $\GoalFunctionalV_C(\template,\Cdot)$ is 
\begin{multline}\label{eq:Energy_functional_gradient_2}
  \grad^{\,\Vspace}_{\velocityfield}\!\!\GoalFunctionalV_C(\template, \velocityfield)(t,\Cdot)  
    = \Koperator\biggl(\diffeo_{0,t}^{\velocityfield} . \template \int_{t}^{1} \grad\Bigl(\Bigl[\partial\DataDisc_{\ForwardOp_{\tau}, g_{\tau}} \bigl(\diffeo_{0,\tau}^{\velocityfield} . \template\bigr) + \mu_2 (1-\tau) \vert\velocityfield(\tau, \cdot)\vert^2\Bigr]\bigl( \gelement{t,\tau}{\velocityfield}\bigr)\Bigr) \dint \tau  \\
    + 2\mu_2(1-t)\diffeo_{0,t}^{\velocityfield} . \template\, \velocityfield(t,\Cdot) \biggr)
  \end{multline}
for $0\leq t \leq 1$ and where $\Koperator(\varphi) = \int_{\Omega} \kernel(\Cdot, y)\varphi(y) \dint y$.
Moreover, the gradient with regard to the template $\template$ of $\GoalFunctionalV_C(\Cdot,\velocityfield)$  is   
\begin{equation}\label{eq:Energy_functional_gradient_template_2}
  \grad_{\template}\GoalFunctionalV_C(\template, \velocityfield)  \\
= \int_0^1 \Bigl[\partial\DataDisc_{\ForwardOp_t, g_t} \bigl(\diffeo_{0,t}^{\velocityfield} . \template\bigr) + \mu_2 (1-t) \vert\diffoperator\velocityfield(t, \cdot)\vert^2\Bigr]\bigl( \gelement{0,t}{\velocityfield}\bigr) \dint t + \mu_1 \partial \RegFunc_1(\template), 
  \end{equation}
where $\partial\RegFunc_1$ denotes the gradient of $\RegFunc_1 \colon \RecSpace \to \Real$.
The optimality conditions for \cref{eq:VarReg_LDDMM_2} are formulated as 
\begin{equation}\label{eq:OptCond_2}
 \begin{cases} 
    \grad^{\,\Vspace}_{\velocityfield}\GoalFunctionalV_{C}(\template, \velocityfield) = 0, & \\[0.5em]
    \grad_{\template}\GoalFunctionalV_{C}(\template, \velocityfield) - \lambda = 0, & \\[0.5em]
    \lambda \geq 0, \quad \template \geq 0, \quad \lambda\template = 0, 
   \end{cases}
\end{equation}
where $\lambda$ denotes the Lagrange multiplier. 
\end{theorem}
\begin{proof}
Applying the results in \cref{lem:DiffEvolutionOperator_tt} and \cref{lem:data_matching_derivative}, we immediately have  
\begin{multline*}
  \frac{d}{d \epsilon}  \GoalFunctionalV_C(\template, \velocityfield + \epsilon\velocityfieldother)  \Bigl\vert_{\epsilon=0}  
=  \int_{0}^{1} \int_{0}^{t}\Bigl\langle \diffeo_{0,\tau}^{\velocityfield} . \template \grad\Bigl(\partial\DataDisc_{\ForwardOp_t, g_t} \bigl(\diffeo_{0,t}^{\velocityfield} . \template\bigr)\bigl( \gelement{\tau,t}{\velocityfield}\bigr)\Bigr),  \velocityfieldother(\tau, \Cdot) \Bigr\rangle_{\LpSpace^2(\domain,\Real^n)} \dint \tau \dint t \\
    + \mu_2\int_{0}^{1} \int_{0}^{t}\int_{0}^{\tau}\Bigl\langle \diffeo_{0,\iota}^{\velocityfield} . \template \grad\Bigl(\vert\diffoperator\velocityfield(\tau, \cdot)\vert^2\bigl( \gelement{\iota,\tau}{\velocityfield}\bigr)\Bigr),  \velocityfieldother(\iota, \Cdot) \Bigr\rangle_{\LpSpace^2(\domain,\Real^n)} \dint \iota \dint \tau \dint t \\
    + 2\mu_2 \int_{0}^{1}\int_{0}^{t}\bigl\langle \diffeo_{0,\tau}^{\velocityfield} . \template\, \velocityfield(\tau,\Cdot),  \velocityfieldother(\tau, \Cdot)\bigr\rangle_{\LpSpace^2(\domain,\Real^n)}\dint \tau \dint t. 
  \end{multline*}  
Changing the order of integration in the above equation gives
 \begin{multline}\label{eq:Energy_functional_derivative_1}
  \frac{d}{d \epsilon}  \GoalFunctionalV_C(\template, \velocityfield + \epsilon\velocityfieldother)  \Bigl\vert_{\epsilon=0} 
   =  \int_{0}^{1} \Bigl\langle \diffeo_{0,\tau}^{\velocityfield} . \template \int_{\tau}^{1} \grad\Bigl(\partial\DataDisc_{\ForwardOp_t, g_t} \bigl(\diffeo_{0,t}^{\velocityfield} . \template\bigr)\bigl( \gelement{\tau,t}{\velocityfield}\bigr)\Bigr)  \dint t,  \velocityfieldother(\tau, \Cdot) \Bigr\rangle_{\LpSpace^2(\domain,\Real^n)}\dint \tau \\
+ \mu_2\int_{0}^{1}  \Bigl\langle \diffeo_{0,\tau}^{\velocityfield} . \template \int_{\tau}^{1} (1-t) \grad\Bigl(\vert\diffoperator\velocityfield(t, \cdot)\vert^2\bigl( \gelement{\tau,t}{\velocityfield}\bigr)\Bigr) \dint t,  \velocityfieldother(\tau, \Cdot) \Bigr\rangle_{\LpSpace^2(\domain,\Real^n)} \dint \tau\\
    + 2\mu_2 \int_{0}^{1}\bigl\langle (1-\tau)\diffeo_{0,\tau}^{\velocityfield} . \template\, \velocityfield(\tau,\Cdot),  \velocityfieldother(\tau, \Cdot)\bigr\rangle_{\LpSpace^2(\domain,\Real^n)} \dint \tau. 
     \end{multline}  
As $\Vspace$ is an \ac{RKHS} with a reproducing kernel represented 
by $\kernel \colon \domain \times \domain \to \Matrix_{+}^{n \times n}$, then 
\begin{equation}\label{eq:RKHS_L2}
 \langle \vfield, \vfieldother \rangle_{\LpSpace^2(\domain,\Real^n)} 
   = \biggl\langle \int_{\domain} \kernel(\Cdot, y) \vfield(y) \dint y, \vfieldother \biggr\rangle_{\Vspace} \quad\text{for $ \vfield, \vfieldother \in \Vspace$}.
\end{equation}
Using \cref{eq:Energy_functional_derivative_1} and \cref{eq:RKHS_L2}, we prove \cref{eq:Energy_functional_gradient_2}.
Moreover, it is simple to obtain the results of \cref{eq:Energy_functional_gradient_template_2} and \cref{eq:OptCond_2}. Therefore their proofs are omitted.  
\end{proof}

For simplicity, let us introduce the following notation 
\begin{align}\label{eq:h_t_ti_general}
&h_{\tau, t}^{\template, \velocityfield} := \begin{cases} 
    \partial\DataDisc_{\ForwardOp_t, g_t} \bigl(\diffeo_{0,t}^{\velocityfield} . \template\bigr)\bigl( \gelement{\tau,t}{\velocityfield}\bigr), \quad 0 \leq \tau \leq t \leq 1,    & \\[0.5em]
    0, \quad t < \tau, &  
   \end{cases} \\
\label{eq:eta_t_ti_general}
&\eta_{\tau, t}^{\velocityfield} := \begin{cases} 
    \int_{\tau}^{t}  \vert\velocityfield(\iota, \cdot)\vert^2\bigl( \gelement{\tau,\iota}{\velocityfield}\bigr) \dint \iota, \quad 0 \leq \tau \leq t \leq 1,    & \\[0.5em]
    0, \quad t < \tau, &  
   \end{cases} \\
\label{eq:Middle_func_deriv_2_general}
&\velocityfield_{\tau, t} := \begin{cases} 
   \velocityfield(\tau,\Cdot), \quad 0 \leq \tau \leq t \leq 1,    & \\[0.5em]
   0, \quad t < \tau. &  
   \end{cases}
\end{align}

\begin{theorem}\label{thm:energy_functional_time_discretized_derivative}
Let the assumptions in \cref{thm:energy_functional_derivative_2} hold and 
$\GoalFunctionalV_D \colon \RecSpace \times \Xspace{2} \to \Real$ denote the objective 
functional in \cref{eq:time_discretized_VarReg_LDDMM_2} of time-discretized version, i.e., 
\begin{equation}\label{eq:LDDMM_time_discretized}
\GoalFunctionalV_D(\template, \velocityfield) :=  \frac{1}{N}\sum_{i=1}^{N} \left[\DataDisc_{\ForwardOp_{t_i}, \data_{t_i}}\bigl(\diffeo_{0,t_i}^{\velocityfield} . \template \bigr)  + 
   \mu_2 \int_0^{t_i} \int_{\Omega}\diffeo_{0,\tau}^{\velocityfield} . \template(x)\vert\velocityfield(\tau, x)\vert^2\dint x\dint \tau \right]  + \mu_1 \RegFunc_1(\template).
\end{equation} 
The gradient of $\GoalFunctionalV_D$ with regard to the velocity field $\velocityfield$ is 
\begin{equation}\label{eq:Energy_functional_time_discretized_gradient_l2}
  \grad_{\velocityfield}\GoalFunctionalV_D(\template, \velocityfield)  
    =  \frac{1}{N}\sum_{\{i\geq 1 : t_i \geq t\}} \biggl( \gelement{0, t}{\velocityfield} . \template \Bigl[ \grad \bigl(h_{t, t_i}^{\template,\velocityfield}  + \mu_2 \eta_{t, t_i}^{\velocityfield}\bigr) 
    + 2\mu_2 \velocityfield_{t, t_i} \Bigr]  \biggr) 
  \end{equation}
and the $\Xspace{2}$-gradient of $\GoalFunctionalV_D$ with regard to the 
velocity field $\velocityfield$ is 
\begin{equation}\label{eq:Energy_functional_time_discretized_gradient}
  \grad^{\,\Vspace}_{\velocityfield}\!\!\GoalFunctionalV_D(\template, \velocityfield)  
    =  \frac{1}{N}\sum_{\{i\geq 1 : t_i \geq t\}} \Koperator\biggl( \gelement{0, t}{\velocityfield} . \template \Bigl[ \grad \bigl(h_{t, t_i}^{\template,\velocityfield}  + \mu_2 \eta_{t, t_i}^{\velocityfield}\bigr) 
    + 2\mu_2 \velocityfield_{t, t_i} \Bigr] \biggr). 
  \end{equation}
Moreover, the gradient of $\GoalFunctionalV_D$ with regard to the template $\template$ is   
\begin{equation}\label{eq:Energy_functional_time_discretized_gradient_template_2}
  \grad_{\template}\GoalFunctionalV_D(\template, \velocityfield)  \\ 
= \frac{1}{N}\sum_{i=1}^{N} \bigl(h_{0, t_i}^{\template, \velocityfield}  + \mu_2 \eta_{0, t_i}^{\velocityfield} \bigr) + \mu_1 \partial \RegFunc_1(\template). 
  \end{equation}
Consequently, the optimality conditions for \cref{eq:LDDMM_time_discretized} are formulated as 
\begin{equation}\label{eq:OptCond_2_time_discretized}
 \begin{cases} 
    \grad^{\,\Vspace}_{\velocityfield}\GoalFunctionalV_{D}(\template, \velocityfield) = 0, & \\[0.5em]
    \grad_{\template}\GoalFunctionalV_{D}(\template, \velocityfield) - \lambda = 0, & \\[0.5em]
    \lambda \geq 0, \quad \template \geq 0, \quad \lambda\template = 0, 
   \end{cases}
\end{equation}
where $\lambda$ denotes the Lagrange multiplier. 
\end{theorem}
\begin{proof}
By \cref{lem:DiffEvolutionOperator_tt}, we derive   
\begin{multline*}
\frac{d}{d \epsilon}  \GoalFunctionalV_D(\template, \velocityfield + \epsilon\velocityfieldother)\Bigl\vert_{\epsilon=0} 
=  \frac{1}{N}\sum_{i=1}^{N} \int_{0}^{t_i}\Bigl\langle \gelement{0, \tau}{\velocityfield} . \template  \grad\Bigl(\partial\DataDisc_{\ForwardOp_{t_i}, g_{t_i}} \bigl(\diffeo_{0,t_i}^{\velocityfield} . \template\bigr)\bigl( \gelement{\tau,t_i}{\velocityfield}\bigr)\Bigr),  \velocityfieldother(\tau, \Cdot) \Bigr\rangle_{\LpSpace^2(\domain,\Real^n)} \dint \tau \\
    + \frac{\mu_2}{N}\sum_{i=1}^{N} \int_{0}^{t_i}\int_{0}^{\tau}\Bigl\langle \diffeo_{0,\iota}^{\velocityfield} . \template \grad\Bigl(\vert\diffoperator\velocityfield(\tau, \cdot)\vert^2\bigl( \gelement{\iota,\tau}{\velocityfield}\bigr)\Bigr),  \velocityfieldother(\iota, \Cdot) \Bigr\rangle_{\LpSpace^2(\domain,\Real^n)} \dint \iota \dint \tau\\
    + \frac{2\mu_2}{N} \sum_{i=1}^{N}\int_{0}^{t_i}\bigl\langle \diffeo_{0,\tau}^{\velocityfield} . \template\, \velocityfield(\tau,\Cdot),  \velocityfieldother(\tau, \Cdot)\bigr\rangle_{\LpSpace^2(\domain,\Real^n)}\dint \tau \\
  = \int_{0}^{1}\Bigl\langle \frac{1}{N}\sum_{\{i\geq 1 : t_i \geq t\}} \gelement{0, t}{\velocityfield} . \template \grad h_{t, t_i}^{\template,\velocityfield},  \velocityfieldother(t, \Cdot) \Bigr\rangle_{\LpSpace^2(\domain,\Real^n)} \dint t \\
    + \int_{0}^{1}\Bigl\langle \frac{\mu_2}{N}\sum_{\{i\geq 1 : t_i \geq t\}} \diffeo_{0,t}^{\velocityfield} . \template \grad \eta_{t, t_i}^{\velocityfield},  \velocityfieldother(t, \Cdot) \Bigr\rangle_{\LpSpace^2(\domain,\Real^n)} \dint t \\
    + \int_{0}^{1}\Bigl\langle \frac{2\mu_2}{N} \sum_{\{i\geq 1 : t_i \geq t\}} \velocityfield_{t, t_i} \diffeo_{0,t}^{\velocityfield} . \template,  \velocityfieldother(t, \Cdot)\Bigr\rangle_{\LpSpace^2(\domain,\Real^n)}\dint t.  
  \end{multline*}
Using \cref{eq:h_t_ti_general}--\cref{eq:Middle_func_deriv_2_general}, the last equation is achieved. 
Using \cref{eq:RKHS_L2} and the obtained result above, we prove \cref{eq:Energy_functional_time_discretized_gradient}. 
Moreover, the result of \cref{eq:Energy_functional_time_discretized_gradient_template_2} is straightforward. 
\end{proof} 
 
%\section*{Acknowledgments}
%The authors would like to thank xx for his helpful discussions.

\bibliographystyle{plain}
\bibliography{shapereferences}

\end{document}